\documentclass[graybox,envcountsect]{svmult}

\usepackage{type1cm}        
\usepackage{makeidx}         
\usepackage{graphicx}        
\usepackage{multicol}        
\usepackage[bottom]{footmisc}

\usepackage{newtxtext}       %
\usepackage[varvw]{newtxmath}       
\usepackage{amsmath,mathtools}
\usepackage{verbatim}
\usepackage{bm,mathrsfs}
\usepackage{color}
\usepackage[shortlabels]{enumitem}
\usepackage{esint} 
\usepackage[title]{appendix}
\usepackage[T1]{fontenc}
\usepackage{hyperref}

\spnewtheorem{thm}[theorem]{Theorem}{\bfseries}{\itshape}
\spnewtheorem{lem}[theorem]{Lemma}{\bfseries}{\itshape}
\spnewtheorem{prop}[theorem]{Proposition}{\bfseries}{\itshape}
\spnewtheorem{cor}[theorem]{Corollary}{\bfseries}{\itshape}
\spnewtheorem{defn}[theorem]{Definition}{\bfseries}{\rmfamily}
\spnewtheorem{assum}[theorem]{Assumption}{\bfseries}{\rmfamily}
\spnewtheorem{rmk}[theorem]{Remark}{\bfseries}{\rmfamily}
\spnewtheorem{exam}[theorem]{Example}{\bfseries}{\rmfamily}
\spnewtheorem{conj}[theorem]{Conjecture}{\bfseries}{\rmfamily}
\spnewtheorem{prbm}[theorem]{Problem}{\bfseries}{\rmfamily}
\spnewtheorem{framework}[theorem]{Framework}{\bfseries}{\rmfamily}
\spnewtheorem*{notation}{Notation}{\bfseries}{\rmfamily}

\allowdisplaybreaks[4]

\newcommand{\norm}[1]{\left\lVert#1\right\rVert}
\newcommand{\trinorm}[1]{{\left\vert\kern-0.25ex\left\vert\kern-0.25ex\left\vert #1
    \right\vert\kern-0.25ex\right\vert\kern-0.25ex\right\vert}}
\newcommand{\indicator}[1]{\mathbf{1}_{#1}} 
\newcommand{\closure}[1]{\overline{#1}}

\newcommand{\abs}[1]{\left\lvert#1\right\rvert} 
\newcommand{\id}{\operatorname{id}}
\newcommand{\supp}{\operatorname{supp}}
\newcommand{\sgn}{\operatorname{sgn}}
\newcommand{\contfunc}{C}
\newcommand{\diam}{\operatorname{diam}}
\newcommand{\dist}{\operatorname{dist}}

\newcommand{\rweight}{\rho} 
\newcommand{\hdim}{d_{\mathrm{f}}}
\newcommand{\phdim}{d_{\mathrm{f},p}}
\newcommand{\pwalk}{d_{\mathrm{w},p}}
\newcommand{\measure}{m}
\newcommand{\metric}{d}
\newcommand{\pmetric}{\widehat{R}}
\newcommand{\KS}{B_{p,\infty}^{\bm{k}}}
\newcommand{\KSform}{\mathcal{E}_{p}^{\bm{k}}}
\newcommand{\KSem}{\Gamma_{p}^{\bm{k}}}
\newcommand{\bclosureKS}{\mathcal{D}_{p,\infty}^{\bm{k},b}}
\newcommand{\cclosureKS}{\mathcal{D}_{p,\infty}^{\bm{k},c}}
\newcommand{\mr}[1]{{\tt \href{http://mathscinet.ams.org/mathscinet-getitem?mr=#1}{MR#1}}}
\newcommand{\arxiv}[1]{{\tt \href{http://arxiv.org/abs/#1}{arXiv:#1}}}

\makeindex            

\begin{document}

\title*{Korevaar--Schoen $p$-energy forms and associated $p$-energy measures on fractals}
\author{Naotaka Kajino\orcidID{0000-0002-0284-4608} and\\ Ryosuke Shimizu\orcidID{0009-0005-4039-3771}}
\institute{Naotaka Kajino \at Research Institute for Mathematical Sciences, Kyoto University, Kitashirakawa-Oiwake-cho, Sakyo-ku, Kyoto 606-8502, Japan \email{nkajino@kurims.kyoto-u.ac.jp}
\and Ryosuke Shimizu \at Waseda Research Institute for Science and Engineering, Waseda University, 3-4-1 Okubo, Shinjuku-ku, Tokyo 169-8555, Japan \email{r-shimizu@aoni.waseda.jp}}
%
%
\maketitle
\renewcommand{\theequation}{\thesection.\arabic{equation}} 

\abstract*{We construct good $p$-energy forms on metric measure spaces as pointwise subsequential limits of Besov-type $p$-energy functionals under certain geometric/analytic conditions.
Such forms are often called \emph{Korevaar--Schoen $p$-energy forms} in the literature. 
As an advantage of our approach, the associated $p$-energy measures are obtained and investigated. 
We also prove that our construction is applicable to the settings of Kigami [\emph{Mem.\ Eur.\ Math.\ Soc.}\ \textbf{5} (2023)] and Cao--Gu--Qiu [\emph{Adv.\ Math.}\ \textbf{405} (2022), no.\ 108517], yields Korevaar--Schoen $p$-energy forms comparable to the $p$-energy forms constructed in these papers, and can be further modified in the case of self-similar sets to obtain self-similar $p$-energy forms keeping most of the good properties of Korevaar--Schoen ones.}

\abstract{We construct good $p$-energy forms on metric measure spaces as pointwise subsequential limits of Besov-type $p$-energy functionals under certain geometric/analytic conditions.
Such forms are often called \emph{Korevaar--Schoen $p$-energy forms} in the literature. 
As an advantage of our approach, the associated $p$-energy measures are obtained and investigated. 
We also prove that our construction is applicable to the settings of Kigami [\emph{Mem.\ Eur.\ Math.\ Soc.}\ \textbf{5} (2023)] and Cao--Gu--Qiu [\emph{Adv.\ Math.}\ \textbf{405} (2022), no.\ 108517], yields Korevaar--Schoen $p$-energy forms comparable to the $p$-energy forms constructed in these papers, and can be further modified in the case of self-similar sets to obtain self-similar $p$-energy forms keeping most of the good properties of Korevaar--Schoen ones.}

\\[10pt]

\section*{Contents}
\setcounter{minitocdepth}{2}
\dominitoc

\section{Introduction}\label{sec.intr}
In this article, assuming that $(K,d)$ is a locally compact separable metric space and that $m$ is a Radon measure (i.e., a Borel measure finite on any compact subset) on $K$ with full topological support (i.e., strictly positive on any non-empty open subset), we consider \emph{Korevaar-Schoen}-type $p$-energy forms on $(K,d,m)$, where $p \in (1,\infty)$.
Namely, we are concerned with a functional
\begin{equation*}
    E_{p,s}(u) \coloneqq \limsup_{r \downarrow 0}\int_{K}\fint_{B_{d}(x,r)}\frac{\abs{u(x) - u(y)}^{p}}{r^{sp}}\,m(dy)m(dx), \quad u \in L^{p}(K,m), 
\end{equation*}
where $B_{d}(x,r) \coloneqq \{ y \in K \mid d(x,y) < r \}$ and $\fint_{A}(\,\cdot\,)\,dm \coloneqq \frac{1}{m(A)}\int_{A}(\,\cdot\,)\,dm$ for a Borel subset $A$ of $K$ with $m(A) \in (0,\infty)$. 
Here $s \in (0,\infty)$ is a parameter controlling the smoothness of functions. 
In the classical settings, the $n$-dimensional Euclidean space $(K,d,m) = (\mathbb{R}^{n},\abs{\,\cdot\,}, dx)$ for example, the choice $s = 1$ is natural.
Indeed, one can show (see, e.g., \cite[Corollary 6.3]{LPZ22+} and \cite[Theorem 7.13]{Haj03}; see also \cite[Theorem 3.5]{Gor22} for a related result) that there exists $C \in (0,\infty)$ such that the distributional gradient $\nabla u$ of any Sobolev function $u \in W^{1,p}(\mathbb{R}^n)$ satisfies
\[
C^{-1} \int_{\mathbb{R}^n} \abs{\nabla u}^p\,dx \le  \limsup_{r \downarrow 0}\int_{\mathbb{R}^n}\fint_{\abs{y - x} < r}\frac{\abs{u(x) - u(y)}^{p}}{r^{p}}\,dydx\le C  \int_{\mathbb{R}^n} \abs{\nabla u}^p\,dx.
\]
In particular, the domain of the functional $E_{p,1}$ is given by the $(1,p)$-Sobolev space $W^{1,p}(\mathbb{R}^{n})$ in this case.
Note that the functional $E_{p,1}$ can be considered as a variant of the functional considered by Korevaar and Schoen in \cite{KS}, where they constructed a $(1,p)$-Sobolev space $W^{1,p}(\Omega,X)$ of maps from a domain $\Omega$ in a Riemannian manifold to a complete metric space $X$. 
On the basis of an idea in \cite{KS}, Koskela and MacManus \cite{KoMa} introduced a $(1,p)$-Sobolev space $\mathcal{L}^{1,p}$ on any metric measure space satisfying the volume doubling property and the Poincar\'e inequality (in terms of weak upper gradients), a so-called \emph{PI-space}, as the domain of a functional similar to $E_{p,1}$, and showed that $\mathcal{L}^{1,p}$ coincides with the $(1,p)$-Sobolev spaces introduced by Haj\l asz \cite{Haj96} and Haj\l asz--Koskela \cite{HK95}; see \cite[Theorem 4.5]{KoMa}.
For any PI-space $(K,d,m)$, one can show (see, e.g., \cite[Corollary 6.3]{LPZ22+} and \cite[Corollary 10.4.6]{HKST}) that $\mathcal{L}^{1,p} = \{ u \in L^{p}(K,m) \mid E_{p,1}(u) < \infty \}$, and it turns out that the exponent $s=1$ is critical in the sense that for every $s > 1$, any function $u \in L^{p}(K,m)$ with $E_{p,s}(u) < \infty$ is constant $m$-a.e.\ if $K$ is connected.
(See \cite[Chapter 10]{HKST} for various ways to define $(1,p)$-Sobolev spaces on $(K,d,m)$ and relations among them.)
Recently, for more general $(K,d,m)$ which may not be a PI-space, Baudoin \cite{Bau22+} proposed to define a $(1,p)$-Sobolev space $\mathrm{KS}^{1,p}$ as the domain $\{ u \in L^{p}(K,m) \mid E_{p,s}(u) < \infty \}$ of $E_{p,s}$ with $s = s_{p}$, where $s_{p}$ is the \emph{critical $L^{p}$-Besov exponent} defined by
\begin{equation*}
s_{p} \coloneqq \sup\{ s \in (0,\infty) \mid \text{$E_{p,s}(u) < \infty$ for some non-constant $u \in L^{p}(K,m)$} \},
\end{equation*}
and discussed some properties of $\mathrm{KS}^{1,p}$ such as Sobolev-type embeddings.

The aim of this article is to construct as nice a $p$-energy form $\mathcal{E}_{p}^{\mathrm{KS}}$ comparable to $E_{p,s_{p}}$ as possible.
Such $\mathcal{E}_{p}^{\mathrm{KS}}$ is desired to satisfy at least the following \emph{generalized $p$-contraction property} (see Definition \ref{defn.GC}): if $q_{1} \in (0,p]$, $q_{2} \in [p,\infty]$, $n_{1},n_{2} \in \mathbb{N}$ and $T = (T_{1},\dots,T_{n_{2}}) \colon \mathbb{R}^{n_{1}} \to \mathbb{R}^{n_{2}}$ satisfies $T(0) = 0$ and $\norm{T(x) - T(y)}_{\ell^{q_{2}}} \le \norm{x - y}_{\ell^{q_{1}}}$ for any $x,y \in \mathbb{R}^{n_{1}}$, then for any $\bm{u} = (u_{1},\dots,u_{n_{1}}) \in (\mathrm{KS}^{1,p})^{n_{1}}$, 
\begin{equation}\label{intro.GCP}
	T(\bm{u}) \in (\mathrm{KS}^{1,p})^{n_{2}} \, \text{ and } \,
	\norm{\bigl(\mathcal{E}_{p}^{\mathrm{KS}}(T_{l}(\bm{u}))^{1/p}\bigr)_{l = 1}^{n_{2}}}_{\ell^{q_{2}}} \le \norm{\bigl(\mathcal{E}_{p}^{\mathrm{KS}}(u_{k})^{1/p}\bigr)_{k = 1}^{n_{1}}}_{\ell^{q_{1}}}. 
\end{equation}
The property \eqref{intro.GCP} has been introduced in \cite{KS.gc} as arguably the strongest possible form of contraction properties of $L^{p}$-like energy forms.
As revealed in \cite{KS.gc}, \eqref{intro.GCP} plays important roles in developing \emph{nonlinear potential theory} in general frameworks including typical self-similar fractals, on which one can construct $p$-energy forms via discrete approximations as established in \cite{CGQ22,HPS04,Kig23,MS+,Shi24} (see also \cite{GYZ23} for a different approach). 
A problem with $E_{p,s}$ is that $E_{p,s}$ may not satisfy \eqref{intro.GCP} because of the operation of taking limsup. 
To avoid this issue, we would like to take a limit (in some sense) of the Besov-type functionals 
\begin{equation}\label{intro.Besov-p-s}
	E_{p,s}(u,r) \coloneqq \int_{K}\fint_{B_{d}(x,r)}\frac{\abs{u(x) - u(y)}^{p}}{r^{sp}}\,m(dy)m(dx)
\end{equation}
as $r \downarrow 0$. 
This strategy does not work for all $s \in (0,\infty)$, but does work in the critical case $s = s_{p}$ in the presence of the following \emph{weak monotonicity} type estimate, which turns out to hold in many situations: there exists a constant $C \in [1,\infty)$ such that for any $u \in L^{p}(K,m)$ with $\sup_{r > 0}E_{p,s_{p}}(u,r) < \infty$,
\begin{equation}\label{wm.intro}
    \sup_{r > 0}E_{p,s_{p}}(u,r) \le C\liminf_{r \downarrow 0}E_{p,s_{p}}(u,r).
\end{equation}
This condition \eqref{wm.intro} was introduced in \cite{Bau22+} (see Example \ref{ex.KS}).
Our first main result, Theorem \ref{thm.KS-energy}, gives a desired $p$-energy form $\mathcal{E}_{p}^{\mathrm{KS}}$ as a subsequential limit of $\{ E_{p,s_{p}}(\,\cdot\,,r) \}_{r > 0}$ under the assumption of \eqref{wm.intro}.
More precisely, in Theorem \ref{thm.KS-energy}, we establish a subsequential limit of the energy functionals given by
\[
\int_{K}\fint_{B_{d}(x,r)}\frac{\sgn\bigl(u(x) - u(y)\bigr)\abs{u(x) - u(y)}^{p - 1}(v(x) - v(y))}{r^{s_{p}p}}\,m(dy)m(dx),
\]
that is, we directly construct a two-variable version $\mathcal{E}_{p}^{\mathrm{KS}}(u; v)$, which is the counterpart of 
\begin{equation}\label{p-form.Euc}
    (u,v) \mapsto \int_{\mathbb{R}^{n}}\abs{\nabla u}^{p - 2}\langle \nabla u, \nabla v\rangle\,dx
\end{equation}
in the Euclidean case, where $\langle \,\cdot\,,\,\cdot\, \rangle$ denotes the inner product on $\mathbb{R}^{n}$.
An advantage of our construction is that we can obtain a good quantitative estimate on the continuity of $\mathcal{E}_{p}^{\mathrm{KS}}(u; v)$ with respect to the nonlinear part $u$.
Namely, unlike our earlier result in \cite{KS.gc} (see \eqref{eq:form-nonlinear-Hoelder-GC} below), the present construction of $\mathcal{E}_{p}^{\mathrm{KS}}$ allows us to achieve the best H\"{o}lder continuity exponent as expected from the formal expression \eqref{p-form.Euc}, i.e., to show that there exists a constant $C \in (0,\infty)$ such that for any $u_{1},u_{2},v \in \mathrm{KS}^{1,p}$, 
\begin{equation*}
	\abs{\mathcal{E}_{p}^{\mathrm{KS}}(u_{1}; v) - \mathcal{E}_{p}^{\mathrm{KS}}(u_{2}; v)}
	\le C\biggl[\max_{i \in \{ 1,2 \}}\mathcal{E}_{p}^{\mathrm{KS}}(u_{i})\biggr]^{\frac{(p - 2)^{+}}{p}}\mathcal{E}_{p}^{\mathrm{KS}}(u_{1} - u_{2})^{\frac{(p - 1) \wedge 1}{p}}\mathcal{E}_{p}^{\mathrm{KS}}(v)^{\frac{1}{p}}
\end{equation*}
(see \eqref{KS-conti1}), which is not known for the $p$-energy forms constructed in the preceding works \cite{CGQ22,GYZ23,HPS04,Kig23,MS+,Shi24}. 
See Section \ref{sec.Kslimit} for details. 

Another superiority of our direct approach is that we can introduce the \emph{$p$-energy measures} associated with $\mathcal{E}_{p}^{\mathrm{KS}}$. 
Roughly speaking, for each $u \in \mathrm{KS}^{1,p}$, the $p$-energy measure $\Gamma_{p}^{\mathrm{KS}}\langle u \rangle$ is a Radon measure on $K$ playing the same role as $\abs{\nabla u}^{p}\,dx$ in the Euclidean case. 
Since we have no counterpart of $\abs{\nabla u}$, it is highly non-trivial to construct  $\Gamma_{p}^{\mathrm{KS}}\langle u \rangle$; indeed, it is not known how to construct canonical $p$-energy measures associated with a given $p$-energy form without relying on the self-similarity of the underlying space and the $p$-energy form (see \cite[p.~113]{Kig23} and \cite[Problem 12.5]{MS+}). 
However, our construction of $\mathcal{E}_{p}^{\mathrm{KS}}$ allows us to employ a naive approach as described below. 
From the Leibniz and chain rules for the usual gradient operator $\nabla$ on $\mathbb{R}^{n}$, we easily see that for any $\varphi,u \in \contfunc^{1}(\mathbb{R}^{n})$, 
\begin{align*}
	\varphi\abs{\nabla u}^{p} 
	= \abs{\nabla u}^{p - 2}\langle \nabla u, \nabla(u\varphi) \rangle - \left(\frac{p - 1}{p}\right)^{p - 1}\abs{\nabla \bigl(\abs{u}^{\frac{p}{p - 1}}\bigr)}^{p - 2}\bigl\langle \nabla\bigl(\abs{u}^{\frac{p}{p - 1}}\bigr),\nabla\varphi \bigr\rangle. 
\end{align*}
Since $\mathcal{E}_{p}^{\mathrm{KS}}(u; v)$ is expected to be the counterpart of $\int_{\mathbb{R}^{n}}\abs{\nabla u}^{p - 2}\langle \nabla u, \nabla v \rangle\,dx$, the $p$-energy measure $\Gamma_{p}^{\mathrm{KS}}\langle u \rangle$ of $u \in \mathrm{KS}^{1,p}$ associated with $\mathcal{E}_{p}^{\mathrm{KS}}$ should be characterized as a unique Radon measure on $K$ such that for any $\varphi \in \mathrm{KS}^{1,p} \cap \contfunc_{c}(K)$, 
\begin{equation}\label{intro.pemformula}
	\int_{K}\varphi\,d\Gamma_{p}^{\mathrm{KS}}\langle u \rangle 
	= \mathcal{E}_{p}^{\mathrm{KS}}(u; u\varphi) - \left(\frac{p - 1}{p}\right)^{p - 1}\mathcal{E}_{p}^{\mathrm{KS}}\bigl(\abs{u}^{\frac{p}{p - 1}}; \varphi\bigr) \eqqcolon \Psi_{p,u}^{\mathrm{KS}}(\varphi). 
\end{equation}
In fact, in the case $p = 2$, this is exactly the same as the definition of energy measures in the theory of regular symmetric Dirichlet forms (see \cite[(3.2.14)]{FOT}). 
By virtue of our direct construction, we can show that $\Psi_{p,u}^{\mathrm{KS}}$ is a bounded positive linear functional on $\mathrm{KS}^{1,p} \cap \contfunc_{c}(K)$ and we obtain $\Gamma_{p}^{\mathrm{KS}}\langle u \rangle$ by applying the Riesz--Markov--Kakutani representation theorem under the assumption that $\mathrm{KS}^{1,p} \cap \contfunc_{c}(K)$ is dense in $\contfunc_{c}(K)$ with respect to the uniform norm. 
We also establish some basic properties of $\Gamma_{p}^{\mathrm{KS}}\langle u \rangle$ like the generalized $p$-contraction property and the chain rule. 
See Section \ref{sec.pEMKS} for details. 

As mentioned above, our construction of $\mathcal{E}_{p}^{\mathrm{KS}}(u; v)$ and $\Gamma_{p}^{\mathrm{KS}}\langle u \rangle$ relies heavily on the assumption of the weak monotonicity estimate \eqref{wm.intro}, and fortunately it turns out that \eqref{wm.intro} holds in many situations.
As proved in \cite[Theorem 5.1]{Bau22+} (see also \cite[Corollary 6.3]{LPZ22+}), \eqref{wm.intro} holds on any PI-spaces.
Besides, \eqref{wm.intro} has been proved for the Vicsek set and for the Sierpi\'{n}ski gasket in \cite[Theorems 6.3 and 6.6]{Bau22+}, for nested fractals in \cite{GYZ23+,CGYZ+}, for generalized Sierpi\'{n}ski carpets with $p$ strictly greater than the Ahlfors regular conformal dimension in \cite{Yan+}, and in a general setting including the Sierpi\'{n}ski carpet with any $p \in (1,\infty)$ in \cite[Theorem 7.1]{MS+}.
See also \cite{HJW05} for related results for the Sierpi\'{n}ski gasket.
As extensions of these results, we present two general settings where we can show \eqref{wm.intro}.
The first one described in Section \ref{sec.Kig} (see Assumptions \ref{assum.CH2} and \ref{assum.CH-ss}) is based on the notion of \emph{$p$-conductive homogeneity} due to \cite{Kig23}, and includes the settings of \cite[Theorems 3.21 and 4.6]{Kig23} except that we need to assume the Ahlfors regularity of $m$, which is not assumed in \cite[Theorem 3.21]{Kig23}. 
(This setting is very similar to that in \cite[Section 7]{MS+}, although there are indeed slight differences between the setting of discrete approximations of $(K,d)$ in \cite{Kig23} and that in \cite{MS+}.)
In particular, \emph{all} the examples of self-similar sets in \cite[Sections 4.4--4.6]{Kig23} and those planned to be treated in \cite{KO+} fall within the framework of our main results in Section \ref{sec.Kig} (see also Remark \ref{rmk.assm-pch}-\ref{it:rmk-pCH}).
The second one presented in Section \ref{sec.CGQ} (see Assumption \ref{assum.pRes}) treats the case of \emph{post-critically  finite self-similar structures}.
In particular, by virtue of the work \cite{CGQ22}, this framework includes all \emph{affine nested fractals}, which were covered only partially in \cite{Kig23} (see Remark \ref{rmk.CGQ-Kig}-\ref{it.difference-CGQ-Kig}).

Very recently, for any $p \in [1,\infty)$, Alonso-Ruiz and Baudoin \cite{AB24+} constructed $p$-energy forms and $p$-energy measures on PI-spaces as $\Gamma$-limits of $E_{p,1}$ and $\overline{\Gamma}$-limits of localized versions of $E_{p,1}$, respectively.
Their framework is very different from ours although we do not deal with the case $p = 1$.
Indeed, $s_{p} = 1$ on PI-spaces while $s_{p} > 1$ on generalized Sierpi\'{n}ski carpets and some Sierpi\'{n}ski gaskets as proved in \cite[Section 9]{KS.gc}.
Also, our construction of $p$-energy measures enables us to prove some fundamental properties of them, which were not shown in \cite{AB24+}. 

This article is organized as follows.
In Section \ref{sec.GC}, we introduce the notion of $p$-energy form and the generalized $p$-contraction property and recall some basic consequences of this property, following \cite{KS.gc}. 
In Section \ref{sec.Kslimit}, we present basic notation related to the Besov-type functionals \eqref{intro.Besov-p-s} and, under the assumptions of \eqref{wm.intro} and some mild conditions, we construct a good $p$-energy form $\mathcal{E}_{p}^{\mathrm{KS}}$ as a subsequential pointwise limit of $\{ E_{p,s_{p}}(\,\cdot\,,r) \}_{r > 0}$.
We also recall the notion of $p$-resistance form and present a sufficient condition for $\mathcal{E}_{p}^{\mathrm{KS}}$ to be a $p$-resistance form in the end of Section \ref{sec.Kslimit}. 
Section \ref{sec.pEMKS} is devoted to discussions on the $p$-energy measures associated with $\mathcal{E}_{p}^{\mathrm{KS}}$. 
(More precisely, we prove these results in Sections \ref{sec.Kslimit} and \ref{sec.pEMKS} in a synthetic way for a more general family of kernels.)
In Section \ref{sec.Kig}, we first recall from \cite{Kig23} the setting of $p$-conductively homogeneous compact metric spaces and then verify \eqref{wm.intro} for them under some geometric assumptions. 
In Section \ref{sec.CGQ}, we show \eqref{wm.intro} for post-critically finite self-similar structures under the assumption of the existence of nice self-similar $p$-resistance forms.  
In Sections \ref{sec.Kig} and \ref{sec.CGQ}, we also show \emph{localized energy estimates}, some estimates on localized versions $\int_{E}\fint_{B_{d}(x,r)}\frac{\abs{u(x) - u(y)}^{p}}{r^{ps_{p}}}\,m(dy)m(dx)$ of $E_{p,s_{p}}(u)$ for any Borel subset $E$ of $K$, and that our construction can be further modified in the case of self-similar sets to obtain self-similar $p$-energy forms keeping most of the good properties of Korevaar--Schoen ones.

\begin{notation}
    Throughout this paper, we use the following notation and conventions.
    \begin{enumerate}[label=\textup{(\arabic*)},align=left,leftmargin=*,topsep=2pt,parsep=0pt,itemsep=2pt]
	    \item For $[0,\infty]$-valued quantities $A$ and $B$, we write $A \lesssim B$ to mean that there exists an implicit constant $C \in (0,\infty)$ depending on some unimportant parameters such that $A \leq CB$. We write $A \asymp B$ if $A \lesssim B$ and $B \lesssim A$.
		\item For a set $A$, we let $\#A \in \mathbb{N} \cup \{ 0,\infty \}$ denote the cardinality of $A$.
	    \item We set $\sup\emptyset \coloneqq 0$ and $\inf\emptyset \coloneqq \infty$. We write $a \vee b \coloneqq \max\{ a, b \}$, $a \wedge b \coloneqq \min\{ a, b \}$ and $a^{+} \coloneqq a \vee 0$ for $a, b \in [-\infty,\infty]$, and we use the same notation also for $[-\infty,\infty]$-valued functions and equivalence classes of them. All numerical functions in this paper are assumed to be $[-\infty,\infty]$-valued.
	    \item Let $X$ be a non-empty set. We define $\id_{X}\colon X \to X$ by $\id_{X}(x) \coloneqq x$, $\indicator{A}=\indicator{A}^{X} \in \mathbb{R}^{X}$ for $A \subseteq X$ by $\indicator{A}(x)\coloneqq\indicator{A}^{X}(x)\coloneqq \begin{cases} 1 \quad &\text{if $x \in A$,} \\ 0 \quad &\text{if $x \not\in A$,} \end{cases}$ and set $\norm{u}_{\sup} \coloneqq \norm{u}_{\sup,X} \coloneqq \sup_{x \in X}\abs{u(x)}$ for $u \colon X \to [-\infty,\infty]$. 
	    \item We define $\sgn\colon\mathbb{R}\to\mathbb{R}$ by $\sgn(a) \coloneqq \indicator{(0, \infty)}(a) - \indicator{(-\infty,0)}(a)$. 
        \item Let $X$ be a topological space. The Borel $\sigma$-algebra of $X$ is denoted by $\mathcal{B}(X)$, the closure of $A \subseteq X$ in $X$ by $\closure{A}^{X}$, and we say that $A \subseteq X$ is \emph{relatively compact} in $X$ if and only if $\closure{A}^{X}$ is compact. We set $\contfunc(X) \coloneqq \{ u \in \mathbb{R}^{X} \mid \text{$u$ is continuous} \}$, $\supp_{X}[u] \coloneqq \closure{X \setminus u^{-1}(0)}^{X}$ for $u \in \contfunc(X)$, $\contfunc_{b}(X) \coloneqq \{ u \in \contfunc(X) \mid \norm{u}_{\sup} < \infty \}$, $\contfunc_{c}(X) \coloneqq \{ u \in \contfunc(X) \mid \text{$\supp_{X}[u]$ is compact} \}$, and $\contfunc_{0}(X) \coloneqq \closure{\contfunc_{c}(X)}^{\contfunc_{b}(X)} = \{ u \in \contfunc(X) \mid \text{$u^{-1}(\mathbb{R} \setminus (-\varepsilon,\varepsilon))$ is compact for any $\varepsilon \in (0,\infty)$}\}$, where $\contfunc_{b}(X)$ is equipped with the uniform norm $\norm{\,\cdot\,}_{\sup}$.
        \item Let $X$ be a topological space having a countable open base. For a Borel measure $m$ on $X$ and a Borel measurable function $f \colon X \to [-\infty,\infty]$ or an $m$-equivalence class $f$ of such functions, we let $\supp_{m}[f]$ denote the support of the measure $\abs{f}\,dm$, that is, the smallest closed subset $F$ of $X$ such that $\int_{X \setminus F}\abs{f}\,dm = 0$. 
        \item Let $(X,d)$ be a metric space. We set $B_{d}(x,r) \coloneqq \{ y \in X \mid d(x,y) < r \}$ for $(x,r) \in X \times (0,\infty)$, $(A)_{d,r} \coloneqq \bigcup_{x \in A}B_{d}(x,r)$ for $A\subseteq X$ and $r\in(0,\infty)$, and $\diam(A,d) \coloneqq \sup_{x,y \in A}d(x,y)$ and $\dist_{d}(A,B) \coloneqq \inf\{ d(x,y) \mid x \in A, y \in B \}$ for $A,B\subseteq X$.  
        \item Let $(X,\mathcal{B},m)$ be a measure space. We set $f_{A} \coloneqq \fint_{A}f\,dm \coloneqq \frac{1}{m(A)}\int_{A}f\,dm$ for $f \in L^{1}(X,m)$ and $A \in \mathcal{B}$ with $m(A) \in (0,\infty)$, and set $m|_{A} \coloneqq m|_{\mathcal{B}|_{A}}$ for $A \in \mathcal{B}$, where $\mathcal{B}|_{A} \coloneqq \{ B \cap A \mid B \in \mathcal{B} \}$. When $m$ is $\sigma$-finite, the product measure space of $(X,\mathcal{B},m)$ and itself is denoted by $(X \times X,\mathcal{B} \otimes \mathcal{B},m \times m)$.
    \end{enumerate}
\end{notation}

\section{$p$-Energy forms and generalized $p$-contraction property}\label{sec.GC}
\setcounter{equation}{0}
In this section, following \cite{KS.gc}, we recall the generalized $p$-contraction property and some basic consequences of it.
Throughout this section, we fix $p \in (1,\infty)$, a measure space $(K,\mathcal{B},m)$, a linear subspace $\mathcal{F}$ of $L^{0}(K,m)\coloneqq L^{0}(K,\mathcal{B},m)$, where
\[
L^{0}(K,\mathcal{B},m) \coloneqq \{ \text{the $m$-equivalence class of $u$} \mid \text{$u \colon K \to \mathbb{R}$, $u$ is $\mathcal{B}$-measurable} \},
\]
and a functional $\mathcal{E} \colon \mathcal{F} \to [0, \infty)$ which is \emph{$p$-homogeneous}, i.e., satisfies $\mathcal{E}(au) = \abs{a}^{p}\mathcal{E}(u)$ for any $(a,u) \in \mathbb{R}\times \mathcal{F}$.
(Note that the pair $(\mathcal{B},m)$ is arbitrary. In the case where $\mathcal{B} = 2^{K}$ and $m$ is the counting measure on $K$, we have $L^{0}(K,\mathcal{B},m) = \mathbb{R}^{K}$.) 

Let us recall the definitions of a $p$-energy form and the generalized $p$-contraction property introduced in \cite{KS.gc}.
We adopt here a less restrictive definition of a $p$-energy form than those in the preceding works
\cite{CGQ22,HPS04,Kig23,MS+,Shi24} on the construction of $p$-energy forms, in order to deal with
a wider class of $L^{p}$-type energy functionals including $E_{p,s}(\cdot,r)$ in \eqref{intro.Besov-p-s}
and $\int_{K}\varphi\,d\Gamma_{p}^{\mathrm{KS}}\langle \,\cdot\, \rangle$
in \eqref{intro.pemformula} in a unified framework. 

\begin{defn}[$p$-Energy form; {\cite[Definition 3.1]{KS.gc}}]\label{defn.p-form}
    The pair $(\mathcal{E},\mathcal{F})$ is said to be a \emph{$p$-energy form} on $(K,m)$ if and only if $\mathcal{E}^{1/p}$ is a seminorm on $\mathcal{F}$.
\end{defn}

\begin{defn}[Generalized $p$-contraction property; {\cite[Definition 2.1]{KS.gc}}]\label{defn.GC}
    The pair $(\mathcal{E},\mathcal{F})$ is said to satisfy the \emph{generalized $p$-contraction property}, \ref{GC} for short, if and only if the following hold:
    if $n_{1},n_{2} \in \mathbb{N}$, $q_{1} \in (0,p]$, $q_{2} \in [p,\infty]$ and $T = (T_{1},\dots,T_{n_{2}}) \colon \mathbb{R}^{n_{1}} \to \mathbb{R}^{n_{2}}$ satisfies
    \begin{equation}\label{GC-cond}
        T(0) = 0 \quad \text{and} \quad \norm{T(x)-T(y)}_{\ell^{q_{2}}}
        \le \norm{x - y}_{\ell^{q_{1}}} \quad \text{for any $x, y \in \mathbb{R}^{n_{1}}$,}
    \end{equation}
    then for any $\bm{u} = (u_{1},\dots,u_{n_{1}}) \in \mathcal{F}^{n_{1}}$ we have
    \begin{equation}\label{GC}
        T(\bm{u}) \in \mathcal{F}^{n_{2}} \quad \text{and} \quad
        \norm{\bigl(\mathcal{E}(T_{l}(\bm{u}))^{1/p}\bigr)_{l = 1}^{n_{2}}}_{\ell^{q_{2}}} \le \norm{\bigl(\mathcal{E}(u_{k})^{1/p}\bigr)_{k = 1}^{n_{1}}}_{\ell^{q_{1}}}. \tag*{\textup{(GC)$_{p}$}}
    \end{equation}
\end{defn}

See \cite[Sections 2 and 3]{KS.gc} for details on consequences of \ref{GC}. 
Here, in Propositions \ref{prop.GClist}, \ref{prop.c-diff} and \ref{prop.form-basic}, we recall some results from \cite{KS.gc} that will be used in this paper.
\begin{prop}[{\cite[Proposition 2.2]{KS.gc}}]\label{prop.GClist}
	Suppose that $(\mathcal{E},\mathcal{F})$ satisfies \ref{GC}. 
	\begin{enumerate}[label=\textup{(\arabic*)},align=left,leftmargin=*,topsep=2pt,parsep=0pt,itemsep=2pt]
		\item\label{GC.tri} $\mathcal{E}^{1/p}$ satisfies the triangle inequality. In particular, $\mathcal{E}$ is convex on $\mathcal{F}$.
        \item\label{GC.lip} Let $\varphi \in \contfunc(\mathbb{R})$ satisfy $\varphi(0) = 0$ and $\abs{\varphi(t) - \varphi(s)} \le \abs{t - s}$ for any $s,t \in \mathbb{R}$.
        Then $\varphi(u) \in \mathcal{F}$ and $\mathcal{E}(\varphi(u)) \le \mathcal{E}(u)$ for any $u \in \mathcal{F}$.
        Furthermore, for any $u \in \mathcal{F} \cap L^{\infty}(K,m)$ and $\Phi \in \contfunc^{1}(\mathbb{R})$ with $\Phi(0) = 0$, we have $\Phi(u) \in \mathcal{F}$ and 
        \begin{equation}\label{compos}
    		\mathcal{E}(\Phi(u)) \le \sup\bigl\{ \abs{\Phi'(t)}^{p} \bigm| t \in \mathbb{R}, \abs{t} \le \norm{u}_{L^{\infty}(K,m)} \bigr\}\mathcal{E}(u).
    	\end{equation}
        \item\label{GC.sadd} For any $u,v \in \mathcal{F}$, we have $u \wedge v, u \vee v \in \mathcal{F}$ and 
        \begin{equation}\label{sadd}
            \mathcal{E}(u \vee v) + \mathcal{E}(u \wedge v) \le \mathcal{E}(u) + \mathcal{E}(v).
        \end{equation}
        \item\label{GC.leipniz} For any $u,v \in \mathcal{F} \cap L^{\infty}(K,m)$, we have $uv \in \mathcal{F} \cap L^{\infty}(K,m)$ and  
        \begin{equation}\label{leibniz}
            \mathcal{E}(uv)^{1/p} \le \norm{v}_{L^{\infty}(K,m)}\mathcal{E}(u)^{1/p} + \norm{u}_{L^{\infty}(K,m)}\mathcal{E}(v)^{1/p}. 
        \end{equation} 
        \item \label{GC.Cp} If $p \in (1,2]$, then for any $u,v \in \mathcal{F}$, 
        \begin{equation}\label{Cp.small}
        	2\bigl(\mathcal{E}(u)^{1/(p - 1)} + \mathcal{E}(v)^{1/(p - 1)}\bigr)^{p - 1} 
        	\le \mathcal{E}(u+v) + \mathcal{E}(u-v)
        	\le 2\bigl(\mathcal{E}(u) + \mathcal{E}(v)\bigr). 
        \end{equation}
        If $p \in [2,\infty)$, then for any $u,v \in \mathcal{F}$, 
        \begin{equation}\label{Cp.large}
        	2\bigl(\mathcal{E}(u)^{1/(p - 1)} + \mathcal{E}(v)^{1/(p - 1)}\bigr)^{p - 1} 
        	\ge \mathcal{E}(u+v) + \mathcal{E}(u-v)
        	\ge 2\bigl(\mathcal{E}(u) + \mathcal{E}(v)\bigr). 
        \end{equation}
    \end{enumerate}
\end{prop}

\begin{prop}[{\cite[Proposition 3.5]{KS.gc}}]\label{prop.c-diff}
	Suppose that $(\mathcal{E},\mathcal{F})$ satisfies \ref{GC}. 
	Then for any $u,v \in \mathcal{F}$, 
	\begin{equation}\label{c-diff}
		\mathcal{E}(u + v) + \mathcal{E}(u - v) - 2\mathcal{E}(u) 
		\le 2\bigl((p - 1) \wedge 1\bigr)\Bigl(\mathcal{E}(u)^{\frac{1}{p - 1}} + \mathcal{E}(v)^{\frac{1}{p - 1}}\Bigr)^{(p - 2)^{+}}\mathcal{E}(v)^{1 \wedge \frac{1}{p - 1}}.
	\end{equation}
	In particular, in view of the convexity of $\mathcal{E}$ from Proposition \ref{prop.GClist}-\ref{GC.tri}, $\mathbb{R} \ni t \mapsto \mathcal{E}(u + tv) \in [0,\infty)$ is differentiable and 
    \begin{equation}\label{e:frechet}
		\lim_{s \to 0}\sup_{h \in \mathcal{F}; \mathcal{E}(h) \le 1}\abs{\frac{\mathcal{E}(u + sh) - \mathcal{E}(u)}{s} - \left.\frac{d}{dt}\mathcal{E}(u + th)\right|_{t = 0}} = 0. 
	\end{equation}
\end{prop}

\begin{prop}[{\cite[Theorem 3.6]{KS.gc}}]\label{prop.form-basic}
	Suppose that $(\mathcal{E},\mathcal{F})$ satisfies \ref{GC}. 
	For any $u,v \in \mathcal{F}$, we define  
    \begin{equation}\label{e:defn-pform-two}
		\mathcal{E}(u; v) \coloneqq \frac{1}{p}\left.\frac{d}{dt}\mathcal{E}(u + tv)\right|_{t = 0},   
	\end{equation}
	which exists by Proposition \ref{prop.c-diff}. 
	Then for any $u,u_{1},u_{2},v \in \mathcal{F}$,
	$\mathcal{E}(u; \,\cdot\,) \colon \mathcal{F} \to \mathbb{R}$ is linear, $\mathcal{E}(u; u) = \mathcal{E}(u)$, $\mathcal{E}(au; v) = \sgn(a)\abs{a}^{p - 1}\mathcal{E}(u; v)$ for any $a \in \mathbb{R}$, 
	\begin{gather}\label{e:form.constantdiff}
		\mathcal{E}(u; h) = 0 \quad \text{and} \quad \mathcal{E}(u + h; v) = \mathcal{E}(u; v) \quad \text{for any $h \in \mathcal{E}^{-1}(0)$,} \\
		\abs{\mathcal{E}(u; v)} \le \mathcal{E}(u)^{(p - 1)/p}\mathcal{E}(v)^{1/p}, \label{e:form-holder} \\
		\abs{\mathcal{E}(u_{1}; v)-\mathcal{E}(u_{2}; v)} \leq c_{p}\biggl[\max_{i\in\{1,2\}}\mathcal{E}(u_{i})\biggr]^{\frac{p-1-\alpha_{p}}{p}}\mathcal{E}(u_{1}-u_{2})^{\alpha_{p}/p}\mathcal{E}(v)^{1/p}
	\label{eq:form-nonlinear-Hoelder-GC}
	\end{gather}
	for $\alpha_{p} \coloneqq \frac{(p-1)\wedge 1}{p}$ and some $c_{p}\in(0,\infty)$
	determined solely and explicitly by $p$.
\end{prop}

\begin{notation}
Throughout this paper, for any $p$-energy form $(\mathcal{E},\mathcal{F})$ on $(K,m)$ satisfying \ref{GC} and any $u,v\in \mathcal{F}$, we let $\mathcal{E}(u;v) \in \mathbb{R}$ denote the element given by \eqref{e:defn-pform-two}.
\end{notation}

\section{Construction and properties of Korevaar--Schoen $p$-energy forms}\label{sec.Kslimit}
\setcounter{equation}{0}
In this section, we show the existence of \emph{Korevaar--Schoen $p$-energy forms}, i.e., pointwise subsequential limits of the Besov-type $p$-energy functionals \eqref{intro.Besov-p-s} under the assumption of the weak monotonicity estimate \eqref{wm.intro}, and give some basic properties of the limit $p$-energy forms. 
To be precise, we will prove these results for a more general \emph{family of kernels} in a synthetic way in order to apply the results in this section to construct self-similar $p$-energy forms later in Sections \ref{sec.Kig} and \ref{sec.CGQ}.

Throughout this section, we fix a separable metric space $(K,d)$ with $\#K \ge 2$ and a $\sigma$-finite Borel measure $m$ on $K$ with full topological support.
Under this setting, the map from $\contfunc(K)$ to $L^{0}(K,m)$ defined by taking $u \in \contfunc(K)$ to its $m$-equivalence class is injective and hence gives a canonical embedding of $\contfunc(K)$ into $L^{0}(K,m)$ as a subalgebra, and we will consider $\contfunc(K)$ as a subset of $L^{0}(K,m)$ through this embedding without further notice. 

We also fix $p \in (1,\infty)$ throughout this section unless otherwise stated. We will state some definitions and statements below for any $p \in [1,\infty)$, but on each such occasion we will explicitly declare that we let $p \in [1,\infty)$.

First, we introduce a function space determined by a family of kernels $\{ k_{r} \}_{r > 0}$.
\begin{defn}\label{defn.kernel}
	Let $p \in [1, \infty)$ and let $\bm{k} = \{ k_{r} \}_{r > 0}$ be a family of $[0,\infty]$-valued Borel measurable functions on $K \times K$. 
    We define a linear subspace $B_{p,\infty}^{\bm{k}}$ of $L^{p}(K,m)$ by
    \begin{equation}\label{Bpinftyk}
        B_{p,\infty}^{\bm{k}} \coloneqq \Biggl\{ f \in L^{p}(K,m) \Biggm| \sup_{r > 0}\int_{K}\int_{K}\abs{f(x) - f(y)}^{p}k_{r}(x,y)\,m(dy)m(dx) < \infty \Biggr\}
    \end{equation}
    and equip $B_{p,\infty}^{\bm{k}}$ with the norm $\norm{\,\cdot\,}_{B_{p,\infty}^{\bm{k}}}$ defined by 
    \[
    \norm{f}_{B_{p,\infty}^{\bm{k}}} \coloneqq \norm{f}_{L^{p}(K,m)} + \left(\sup_{r > 0}\int_{K}\int_{K}\abs{f(x) - f(y)}^{p}k_{r}(x,y)\,m(dy)m(dx)\right)^{1/p}.
    \]
    Also for each $r \in (0,\infty)$, we define $J_{p,r} \colon L^{p}(K,m) \to [0,\infty]$ by
    \begin{equation*}
    J_{p,r}^{\bm{k}}(f) \coloneqq \int_{K}\int_{K}\abs{f(x) - f(y)}^{p}k_{r}(x,y)\,m(dy)m(dx), \quad f \in L^{p}(K,m), 
    \end{equation*}
    and set $D(J_{p,r}^{\bm{k}}) \coloneqq \{ f \in L^{p}(K,m) \mid J_{p,r}^{\bm{k}}(f) < \infty \}$. 
\end{defn}

In the rest of this section, we fix a family of kernels $\bm{k} = \{ k_{r} \}_{r > 0}$
as in Definition \ref{defn.kernel}.
To state some basic properties of $J_{p,r}^{\bm{k}}$, let us recall the reverse Minkowski inequality (see, e.g., \cite[Theorem 2.13]{AF}).

\begin{prop}[Reverse Minkowski inequality]\label{prop.reverse}
    Let $(Y,\mathcal{A},\mu)$ be a measure space\footnote{In the book \cite{AF}, the reverse Minkowski inequality is stated and proved only for the $L^{r}$-space over non-empty open subsets of the Euclidean space equipped with the Lebesgue measure, but the same proof works for any measure space.} and let $r \in (0,1]$.
    Then for any $\mathcal{A}$-measurable functions $f,g \colon Y \to [0,\infty]$,
    \begin{equation}\label{reverse}
        \left(\int_{Y}f^{r}\,d\mu\right)^{1/r} + \left(\int_{Y}g^{r}\,d\mu\right)^{1/r} \le \left(\int_{Y}(f + g)^{r}\,d\mu\right)^{1/r}.
    \end{equation}
\end{prop}

For ease of notation, we define $\gamma_{p} \colon \mathbb{R} \to \mathbb{R}$ by 
\begin{equation*}
   	\gamma_{p}(a) \coloneqq \sgn(a)\abs{a}^{p - 1}. 
\end{equation*}
The following proposition is elementary. 
\begin{prop}\label{prop.Besov-GC}
	For any $r \in (0,\infty)$, $(J_{p,r}^{\bm{k}},D(J_{p,r}^{\bm{k}}))$ is a $p$-energy form on $(K,m)$
	satisfying \ref{GC}, and for any $f,g \in D(J_{p,r}^{\bm{k}})$, 
	\begin{equation}\label{e:Besov-compatible}
		J_{p,r}^{\bm{k}}(f; g) = \int_{K}\int_{K}\gamma_{p}\bigl(f(x) - f(y)\bigr)(g(x) - g(y))k_{r}(x,y)\,m(dy)m(dx). 
	\end{equation}
\end{prop}
\begin{proof}
	Suppose that $T = (T_{1},\dots,T_{n_{2}}) \colon \mathbb{R}^{n_{1}} \to \mathbb{R}^{n_{2}}$ satisfies \eqref{GC-cond} and that $q_{2} < \infty$.
    Then for any $\bm{u} = (u_{1},\dots,u_{n_{1}}) \in D(J_{p,r}^{\bm{k}})^{n_{1}}$ and any $r \in (0,\infty)$,
    \begin{align}\label{KSpre.GC}
        &\sum_{l = 1}^{n_{2}}J_{p,r}^{\bm{k}}\bigl(T_{l}(\bm{u})\bigr)^{q_{2}/p} \nonumber \\
        &\overset{\eqref{reverse}}{\le} \left(\int_{K}\int_{K}\Biggl[\sum_{l = 1}^{n_{2}}\abs{T_{l}(\bm{u}(x)) - T_{l}(\bm{u}(y))}^{q_{2}}\Biggr]^{p/q_{2}}k_{r}(x,y)\,m(dy)m(dx)\right)^{q_{2}/p} \nonumber \\ 
        &\overset{\eqref{GC-cond}}{\le} \left(\int_{K}\int_{K}\Biggl[\sum_{k = 1}^{n_{1}}\abs{u_{k}(x) - u_{k}(y)}^{q_{1}}\Biggr]^{p/q_{1}}k_{r}(x,y)\,m(dy)m(dx)\right)^{q_{2}/p} \nonumber \\
        &\overset{\text{($\ast$)}}{\le} \left(\sum_{k = 1}^{n_{1}}\left(\int_{K}\int_{K}\abs{u_{k}(x) - u_{k}(y)}^{p}k_{r}(x,y)\,m(dy)m(dx)\right)^{q_{1}/p}\right)^{q_{2}/q_{1}} \nonumber \\
        &= \left(\sum_{k = 1}^{n_{1}}J_{p,r}^{\bm{k}}(u_{k})^{q_{1}/p}\right)^{q_{2}/q_{1}}.
    \end{align}
    Here we used the triangle inequality for the norm of $L^{p/q_{1}}(K \times K, m_{r}(dxdy))$ in ($\ast$), where $m_{r}(dxdy) \coloneqq k_{r}(x,y)\,m(dy)m(dx)$. 
    The proof for the case $q_{2} = \infty$ is similar, so $(J_{p,r}^{\bm{k}},D(J_{p,r}^{\bm{k}}))$ is a $p$-energy form on $(K,m)$ satisfying \ref{GC}. 
    The equality \eqref{e:Besov-compatible} follows from the dominated convergence theorem. 
\end{proof}

Similarly, we can show the next proposition.
\begin{prop}\label{prop.GC-each}
    Let $r \in (0,\infty)$ and define $N_{p,r}^{\bm{k}}(f) \coloneqq \norm{f}_{L^{p}(K,m)}^{p} + J_{p,r}^{\bm{k}}(f)$ for $f \in D(J_{p,r}^{\bm{k}})$.
    Then $(N_{p,r}^{\bm{k}},D(J_{p,r}^{\bm{k}}))$ is a $p$-energy form on $(K,m)$ satisfying \ref{GC}.
    In particular, for any $f,g \in D(J_{p,r}^{\bm{k}})$ with $N_{p,r}^{\bm{k}}(f) \vee N_{p,r}^{\bm{k}}(g) \le 1$,
    \begin{equation}\label{pCla-each}
        N_{p,r}^{\bm{k}}(f + g)
        \le \Bigl(2^{p \vee \frac{p}{p - 1}} - N_{p,r}^{\bm{k}}(f - g)\Bigr)^{(p - 1) \wedge 1}.
    \end{equation}
\end{prop}
\begin{proof}
	A similar estimate as \eqref{KSpre.GC} shows that $(N_{p,r}^{\bm{k}},D(J_{p,r}^{\bm{k}}))$ is a $p$-energy form on $(K,m)$ satisfying \ref{GC}. 
	The desired estimate \eqref{pCla-each} immediately follows from Proposition \ref{prop.GClist}-\ref{GC.Cp}. 
\end{proof}

Let us introduce a couple of important conditions on $\bm{k}$.
\begin{defn}
    \begin{enumerate}[label=\textup{(\arabic*)},align=left,leftmargin=*,topsep=2pt,parsep=0pt,itemsep=2pt]
        \item We say that $\bm{k} = \{ k_{r} \}_{r > 0}$ is \emph{asymptotically local} if and only if there exists $\{ \delta(r) \}_{r > 0} \subseteq (0,\infty)$ such that $\lim_{r \downarrow 0}\delta(r) = 0$ and 
        \begin{equation}\label{e:asylocal}
        	\lim_{r \downarrow 0}\int_{K}\int_{K \setminus B_{d}(x,\delta(r))}k_{r}(x,y)\,m(dy)m(dx) = 0. 
        \end{equation} 
        \item Let $p \in [1,\infty)$.
        We say that \ref{KSwm} holds if and only if there exists $C \in [1,\infty)$ such that
        \begin{equation}\label{KSwm}
            \sup_{r > 0}J_{p,r}^{\bm{k}}(f) \le C \liminf_{r \downarrow 0}J_{p,r}^{\bm{k}}(f) \quad \text{for any $f \in B_{p,\infty}^{\bm{k}}$.} \tag*{\textup{(WM)$_{p,\bm{k}}$}}
        \end{equation}
    \end{enumerate}
\end{defn}

The next theorem states that the normed space $B_{p,\infty}^{\bm{k}}$ equipped with $\norm{\,\cdot\,}_{B_{p,\infty}^{\bm{k}}}$ is a nice Banach space.
Our proof is very similar to that for the case of the $(1,p)$-Korevaar--Schoen--Sobolev space $\mathrm{KS}^{1,p}$ given in \cite[Theorems 3.1 and 4.4]{Bau22+}.
We present a complete proof here to make this paper self-contained.
\begin{thm}\label{thm.banach}
    For any $p \in [1,\infty)$, the normed space $B_{p,\infty}^{\bm{k}}$ is a Banach space.
    Moreover, if $p \in (1,\infty)$ and \ref{KSwm} holds, then $B_{p,\infty}^{\bm{k}}$ is reflexive and separable.
\end{thm}
\begin{proof}
    Let $\{ f_{n} \}_{n \in \mathbb{N}}$ be a Cauchy sequence in $B_{p,\infty}^{\bm{k}}$.
    Then there exists a $L^{p}$-limit $f \in L^{p}(K,m)$ of $\{ f_{n} \}_{n \in \mathbb{N}}$.
    For any $\varepsilon > 0$ there exists $N \in \mathbb{N}$ such that $\norm{f_{n} - f_{n'}}_{B_{p,\infty}^{\bm{k}}} < \varepsilon$ for any $n,n' \ge N$.
    By using Fatou's lemma, we see that $J_{p,r}^{\bm{k}}(f - f_{n}) \le \varepsilon^{p}$ for any $n \ge N$ and hence
    \[
    J_{p,r}^{\bm{k}}(f)^{1/p} \le J_{p,r}^{\bm{k}}(f - f_{n})^{1/p} + J_{p,r}^{\bm{k}}(f_{n})^{1/p} \le \varepsilon + \sup_{n \in \mathbb{N}}\norm{f_{n}}_{B_{p,\infty}^{\bm{k}}}.
    \]
    Therefore, $f \in B_{p,\infty}^{\bm{k}}$ and $\{ f_{n} \}_{n}$ converges to $f$ in $B_{p,\infty}^{\bm{k}}$, i.e., $B_{p,\infty}^{\bm{k}}$ is a Banach space.

    Next we assume that $p \in (1,\infty)$ and that \ref{KSwm} holds.
    Then $\trinorm{f}_{B_{p,\infty}^{\bm{k}}} \coloneqq \bigl(\norm{f}_{L^{p}(K,m)}^{p} + \limsup_{r \downarrow 0}J_{p,r}^{\bm{k}}(f)\bigr)^{1/p}$ is a norm on $B_{p,\infty}^{\bm{k}}$ that is equivalent to $\norm{\,\cdot\,}_{B_{p,\infty}^{\bm{k}}}$.
    We will show that $\trinorm{\,\cdot\,}_{B_{p,\infty}^{\bm{k}}}$ is uniformly convex (see \cite[Definition 1]{Cla36} for the definition) and thus $B_{p,\infty}^{\bm{k}}$ is reflexive by the Milman--Pettis theorem (see, e.g., \cite[Theorem 2.49]{HKST}).
    Let $\varepsilon > 0$ and $f,g \in B_{p,\infty}^{\bm{k}}$ with $\trinorm{f}_{B_{p,\infty}^{\bm{k}}} \vee \trinorm{g}_{B_{p,\infty}^{\bm{k}}} < 1$ and $\trinorm{f - g}_{B_{p,\infty}^{\bm{k}}} > \varepsilon$.
    By \cite[Lemma 4.11]{Bau22+}, it suffices to find $\delta \in (0,\infty)$ that is independent of $f,g$ such that $\trinorm{f + g}_{B_{p,\infty}^{\bm{k}}} \le 2(1 - \delta)$.
    Choose $r_{0} \in (0,\infty)$ so that
    \[
    \norm{f}_{L^{p}(K,m)}^{p} + J_{p,r}^{\bm{k}}(f) < 1, \quad \norm{g}_{L^{p}(K,m)}^{p} + J_{p,r}^{\bm{k}}(g) < 1 \quad \text{for any $r \in (0,r_{0})$}.
    \]
    Since \ref{KSwm} implies that
    \begin{align*}
    	\varepsilon^{p}
    	&< \norm{f - g}_{L^{p}(K,m)}^{p} + \limsup_{r \downarrow 0}J_{p,r}^{\bm{k}}(f - g) \\
    	&\le C\Bigl(\norm{f - g}_{L^{p}(K,m)}^{p} + \liminf_{r \downarrow 0}J_{p,r}^{\bm{k}}(f - g)\Bigr),
    \end{align*}
    there exists $r_{1} \in (0,\infty)$ such that
    \[
    \norm{f - g}_{L^{p}(K,m)}^{p} + J_{p,r}^{\bm{k}}(f - g) > C^{-1}\varepsilon^{p} \quad \text{for any $r \in (0,r_{1})$.}
    \]
    Hence, for any $r \in (0, r_{0} \wedge r_{1})$, by using \eqref{pCla-each}, we see that
    \[
    \norm{f + g}_{L^{p}(K,m)}^{p} + J_{p,r}^{\bm{k}}(f + g)
    < \Bigl[2^{p \vee \frac{p}{p - 1}} - C^{-1}\varepsilon^{p}\Bigr]^{(p - 1) \wedge 1},
    \]
    which implies $\norm{f + g}_{L^{p}(K,m)}^{p} + J_{p,r}^{\bm{k}}(f + g) \le 2^{p}(1 - \delta)$ for some $\delta \in (0,\infty)$ depending only on $p,C,\varepsilon$.
    The desired uniform convexity is proved.
	
	Since $L^{p}(K,m)$ is separable and the inclusion map of $B_{p,\infty}^{\bm{k}}$ into $L^{p}(K,m)$ is a continuous linear injection, $B_{p,\infty}^{\bm{k}}$ is separable by \cite[Proposition 4.1]{AHM23}. 
\end{proof}

To obtain the local H\"{o}lder continuity with exponent $(p - 1) \wedge 1$ of the Korevaar--Schoen $p$-energy forms (see Theorem \ref{thm.KS-energy}-\ref{KS.conti} below), we will need the following elementary inequality (see also \cite[Proof of Corollary 5.8]{Lin}).
\begin{lem}\label{lem.p-1}
	For any $a, b \in \mathbb{R}$,
	\[
	\abs{\gamma_{p}(a) - \gamma_{p}(b)}
	\le
	\begin{cases}
		2\abs{a - b}^{p - 1} \quad &\text{if $p \in (1,2]$,} \\
		(p - 1)\bigl(\abs{a}^{p - 2} \vee \abs{b}^{p - 2}\bigr)\abs{a - b} \quad &\text{if $p \in (2,\infty)$.}
	\end{cases}
	\]
\end{lem}
\begin{proof}
    The desired estimate is evident when $\abs{a} \wedge \abs{b} = 0$, so we can assume that $0 < \abs{b} \le \abs{a}$ by exchanging $a$ and $b$ if necessary.
    The proof is divided into the following five cases.

    \noindent \textbf{Case 1:} $p \in (1,2]$ and $ab < 0$.

    We can assume that $b < 0 < a$ by considering $-a, -b$ instead of $a, b$ respectively if necessary.
    Note that $\abs{a} \le \abs{a - b}$.
    We see that
    \begin{align*}
        \abs{\gamma_{p}(a) - \gamma_{p}(b)}
        = a^{p - 1} - (-b)^{p - 1}
        \le 2\abs{a}^{p - 1}
        \le 2\abs{a - b}^{p - 1}.
    \end{align*}

    \noindent \textbf{Case 2:} $p \in (1,2]$, $ab > 0$ and $\abs{a - b} \le \abs{b}$.

    By the same reason as the previous case, we can assume that $a \ge b > 0$.
    Noting that $\abs{a - b}^{p - 2} \ge \abs{b}^{p - 2}$, we have
    \begin{align*}
        \abs{\gamma_{p}(a) - \gamma_{p}(b)}
        = a^{p - 1} - b^{p - 1}
        &= (p - 1)\int_{b}^{a}t^{p - 2}\,dt \\
        &\le (p - 1)\abs{b}^{p - 2}\abs{a - b}
        \le (p - 1)\abs{a - b}^{p - 1}.
    \end{align*}

    \noindent \textbf{Case 3:} $p \in (1,2]$, $ab > 0$ and $\abs{a - b} \ge \abs{b}$.

    Similar to the previous cases, we can assume that $a \ge b > 0$.
    Then $\abs{a - b} \ge \abs{b}$ is equivalent to $b \in (0,a/2]$, whence $\frac{a}{2} \le a - b = \abs{a - b}$.
    Now we see that
    \begin{align*}
        \abs{\gamma_{p}(a) - \gamma_{p}(b)}
        = a^{p - 1} - b^{p - 1}
        \le a^{p - 1}
        \le 2^{p - 1}\abs{a - b}^{p - 1}.
    \end{align*}

    \noindent \textbf{Case 4:} $p \in (2,\infty)$ and $ab < 0$.

    In this case, we have $\abs{b}^{p - 2} \le \abs{a}^{p - 2}$ by $p - 2 \ge 0$.
    We can assume that $b < 0 < a$ similarly to Case 1.
    Then
    \begin{align*}
        \abs{\gamma_{p}(a) - \gamma_{p}(b)}
        = \abs{a}^{p - 2}a - \abs{b}^{p - 2}b
        \le \abs{a}^{p - 2}a - \abs{a}^{p - 2}b
        = \abs{a}^{p - 2}\abs{a - b}.
    \end{align*}

    \noindent \textbf{Case 5:} $p \in (2,\infty)$ and $ab > 0$.

    Similar to Cases 2 and 3, we can assume that $a \ge b > 0$.
    Then
    \begin{align*}
        \abs{\gamma_{p}(a) - \gamma_{p}(b)}
        = a^{p - 1} - b^{p - 1}
        = (p - 1)\int_{b}^{a}t^{p - 2}\,dt
        \le (p - 1)\abs{a}^{p - 2}\abs{a - b}. 
    \end{align*}
	
	The above five cases complete the proof.
\end{proof}

Now we can state and prove the first main theorem of this paper as follows.
Recall that we have fixed $p \in (1,\infty)$.
\begin{thm}\label{thm.KS-energy}
	Suppose that \ref{KSwm} holds.
	Then any sequence $\{ \tilde{r}_{n} \}_{n \in \mathbb{N}} \subseteq (0,\infty)$ with $\tilde{r}_{n} \to 0$ has a subsequence $\{ r_{n} \}_{n \in \mathbb{N}}$ such that the following limit exists in $[0,\infty)$ for any $f \in \KS$:
	\begin{equation}\label{e:defn-KSform}
    	\KSform(f) \coloneqq \lim_{n \to \infty}J_{p,r_{n}}^{\bm{k}}(f).
    \end{equation}
    Moreover, for any such $\{ r_{n} \}_{n \in \mathbb{N}}$, the functional $\KSform \colon \KS \to [0,\infty)$ defined by \eqref{e:defn-KSform} satisfies the following properties:
	\begin{enumerate}[label=\textup{(\alph*)},align=left,leftmargin=*,topsep=2pt,parsep=0pt,itemsep=2pt]
		\item \label{KS.p-form} $(\mathcal{E}_{p}^{\bm{k}},B_{p,\infty}^{\bm{k}})$ is a $p$-energy form on $(K,m)$ such that 
		\begin{equation}\label{KSene-comp}
			C^{-1}\sup_{r > 0}J_{p,r}^{\bm{k}}(f) \le \mathcal{E}_{p}^{\bm{k}}(f) \le C\liminf_{r \downarrow 0}J_{p,r}^{\bm{k}}(f) \quad \text{for any $f \in B_{p,\infty}^{\bm{k}}$},
		\end{equation}
		where $C \in (0,\infty)$ is the same as in \ref{KSwm}.
        In particular, if $m(K)<\infty$, then $\indicator{K} \in B_{p,\infty}^{\bm{k}}$ and $\mathcal{E}_{p}^{\bm{k}}(\indicator{K}) = 0$. 
        \item \label{KS.GC} $(\mathcal{E}_{p}^{\bm{k}},B_{p,\infty}^{\bm{k}})$ satisfies \ref{GC}. Furthermore, for any $f,g \in B_{p,\infty}^{\bm{k}}$, $\{ J_{p,r_{n}}^{\bm{k}}(f; g) \}_{n \in \mathbb{N}}$ is convergent in $\mathbb{R}$ and 
        \begin{equation}\label{KS-compatible}
        	\KSform(f; g) = \lim_{n \to \infty}J_{p,r_{n}}^{\bm{k}}(f; g). 
        \end{equation}
        \item\label{it:difffunc} \textup{(Function-wise generalized $p$-contraction property)} Let $n_{1},n_{2} \in \mathbb{N}$, $q_{1} \in (0,p]$, $q_{2} \in [p,\infty]$, $\bm{u} = (u_{1},\dots,u_{n_{1}}) \in (B_{p,\infty}^{\bm{k}})^{n_{1}}$ and $\bm{v} = (v_{1},\dots,v_{n_{2}}) \in L^{p}(K,m)^{n_{2}}$. If 
        \begin{equation}\label{e:FGC.cond}
        	\norm{\bm{v}(x) - \bm{v}(y)}_{\ell^{q_{2}}} \le \norm{\bm{u}(x) - \bm{u}(y)}_{\ell^{q_{1}}} \quad \text{for $m \times m$-a.e.\ $(x,y) \in K \times K$,}
        \end{equation}
        then $\bm{v} \in (\KS)^{n_{2}}$ and 
        \begin{equation}\label{e:FGC}
        	\norm{\bigl(\mathcal{E}_{p}^{\bm{k}}(v_{l})^{1/p}\bigr)_{l = 1}^{n_{2}}}_{\ell^{q_{2}}} \le \norm{\bigl(\mathcal{E}_{p}^{\bm{k}}(u_{k})^{1/p}\bigr)_{k = 1}^{n_{1}}}_{\ell^{q_{1}}}. 
        \end{equation}
        \item \label{KS.conti} \textup{(Local H\"{o}lder continuity)} There exists $C_{p} \in (0,\infty)$ determined solely and explicitly by $p$ such that for any $f_{1},f_{2}, g \in B_{p,\infty}^{\bm{k}}$,
        \begin{equation}\label{KS-conti1}
            \abs{\mathcal{E}_{p}^{\bm{k}}(f_{1};g) - \mathcal{E}_{p}^{\bm{k}}(f_{2};g)}
            \le C_{p}\biggl[\max_{i \in \{ 1,2 \}}\mathcal{E}_{p}^{\bm{k}}(f_{i})\biggr]^{\frac{(p - 2)^{+}}{p}}\mathcal{E}_{p}^{\bm{k}}(f_{1} - f_{2})^{\frac{(p - 1) \wedge 1}{p}}\mathcal{E}_{p}^{\bm{k}}(g)^{\frac{1}{p}}.
        \end{equation}
		\item \label{KS.slbdd} \textup{(Strong locality)} Suppose that $\bm{k}$ is asymptotically local. 
			\begin{enumerate}[label=\textup{(\roman*)},align=left,leftmargin=*,topsep=2pt,parsep=0pt,itemsep=2pt]
				\item \label{it:KS.sl1} Let $f_{1}, f_{2}, g \in B_{p,\infty}^{\bm{k}}$. If $\supp_{m}[f_1 - a_1\indicator{K}] \cap \supp_{m}[f_2 - a_2\indicator{K}] = \emptyset$ and either $\supp_{m}[f_1 - a_1\indicator{K}]$ or $\supp_{m}[f_2 - a_2\indicator{K}]$ is compact for some $a_1,a_2 \in \mathbb{R}$, then
					\begin{align}
            			\mathcal{E}_{p}^{\bm{k}}(f_1 + f_2 + g) + \mathcal{E}_{p}^{\bm{k}}(g) &= \mathcal{E}_{p}^{\bm{k}}(f_1 + g) + \mathcal{E}_{p}^{\bm{k}}(f_2 + g), \label{e:KS-sl1} \\
						\mathcal{E}_{p}^{\bm{k}}(f_1 + f_2; g) &= \mathcal{E}_{p}^{\bm{k}}(f_1; g) + \mathcal{E}_{p}^{\bm{k}}(f_2; g).
        			\label{e:KS-sl1-cor}
					\end{align}
				\item \label{it:KS.sl2} Let $f_{1},f_{2},g \in \KS$. If $\supp_{m}[f_{1} - f_{2} - a\indicator{K}] \cap \supp_{m}[g - b\indicator{K}] = \emptyset$ and either $\supp_{m}[f_{1} - f_{2} - a\indicator{K}]$ or $\supp_{m}[g - b\indicator{K}]$ is compact for some $a,b \in \mathbb{R}$, then 
					\begin{equation}\label{e:KS-sl2}
    					\KSform(f_{1}; g) = \KSform(f_{2}; g)
						\quad \text{and} \quad
						\KSform(g; f_{1}) = \KSform(g; f_{2}).
    				\end{equation}
			\end{enumerate}
        \item \label{KS.inv} \textup{(Invariance)} Let $T \colon K \to K$ be Borel measurable and preserve $m$, i.e., satisfy $T^{-1}(A) \in \mathcal{B}(K)$ and $m(T^{-1}(A)) = m(A)$ for any $A \in \mathcal{B}(K)$. If $\bm{k}$ is \emph{$T$-invariant}, i.e., $k_{r}(T(x),T(y)) = k_{r}(x,y)$ for $m \times m$-a.e.\ $(x,y) \in K \times K$ for each $r \in (0,\infty)$, then $f \circ T \in B_{p,\infty}^{\bm{k}}$ and $\mathcal{E}_{p}^{\bm{k}}(f \circ T) = \mathcal{E}_{p}^{\bm{k}}(f)$ for any $f \in B_{p,\infty}^{\bm{k}}$. 
 	\end{enumerate}
\end{thm}
\begin{defn}[$\bm{k}$-Korevaar--Schoen $p$-energy form]\label{defn.KS-energy}
	Suppose that \ref{KSwm} holds. 
	For each sequence $\{ r_{n} \}_{n \in \mathbb{N}} \subseteq (0,\infty)$ as in Theorem \ref{thm.KS-energy}, the $p$-energy form $(\KSform,\KS)$ on $(K,m)$ defined by \eqref{e:defn-KSform} is called the \emph{$\bm{k}$-Korevaar--Schoen $p$-energy form} on $(K,m)$ along $\{ r_{n} \}_{n \in \mathbb{N}}$. 
\end{defn}
\begin{rmk}
	Advantages of our $p$-energy form $(\mathcal{E}_{p}^{\bm{k}},B_{p,\infty}^{\bm{k}})$ on $(K,m)$ are \ref{it:difffunc} and \ref{KS.conti}.  
    The estimate \eqref{KS-conti1} with the H\"{o}lder continuity exponent $(p - 1) \wedge 1$ is not known for the $p$-energy forms constructed in \cite{CGQ22,HPS04,Kig23,MS+,Shi24}. 
    (As stated in Proposition \ref{prop.form-basic}, the existence of the derivative as in \eqref{e:defn-pform-two} and its local H\"{o}lder continuity \eqref{eq:form-nonlinear-Hoelder-GC} with exponent $\frac{(p - 1) \wedge 1}{p}$ for $p$-energy forms satisfying \ref{GC} have been proved in \cite{KS.gc}.) 
    We also do not know whether \ref{it:difffunc} holds for the $p$-energy forms constructed in \cite{CGQ22,HPS04,Kig23,MS+,Shi24}. 
\end{rmk}
\begin{proof}[Proof of Theorem \ref{thm.KS-energy}]
	Fix a sequence of positive numbers $\{ \tilde{r}_{n} \}_{n \in \mathbb{N}}$ with $\tilde{r}_{n} \to 0$.
	Since $B_{p,\infty}^{\bm{k}}$ is separable by Theorem \ref{thm.banach}, there exists a countable dense subset $\mathscr{C}$ of $B_{p,\infty}^{\bm{k}}$.
	A standard diagonal argument yields a subsequence $\{ r_{n} \}_{n \in \mathbb{N}}$ of $\{ \tilde{r}_{n} \}_{n \in \mathbb{N}}$ so that $\lim_{n \to \infty}J_{p,r_{n}}^{\bm{k}}(u)$ exists in $\mathbb{R}$ for any $u \in \mathscr{C}$.
    Let $\varepsilon > 0$, $f \in \KS$ and pick $f_{\ast} \in \mathscr{C}$ satisfying $\sup_{r > 0}J_{p,r}^{\bm{k}}(f - f_{\ast})^{1/p} < \varepsilon$.
    Then for any $k, l \in \mathbb{N}$, by using the triangle inequality for $J_{p,r}^{\bm{k}}(\,\cdot\,)^{1/p}$,
    \begin{align*}
        &\abs{J_{p,r_{k}}^{\bm{k}}(f)^{1/p} - J_{p,r_{l}}^{\bm{k}}(f)^{1/p}} \\
        &\le J_{p,r_{k}}^{\bm{k}}(f - f_{\ast})^{1/p} + \abs{J_{p,r_{k}}^{\bm{k}}(f_{\ast})^{1/p} - J_{p,r_{l}}^{\bm{k}}(f_{\ast})^{1/p}} + J_{p,r_{l}}^{\bm{k}}(f - f_{\ast})^{1/p} \\
        &\le 2\varepsilon + \abs{J_{p,r_{k}}^{\bm{k}}(f_{\ast})^{1/p} - J_{p,r_{l}}^{\bm{k}}(f_{\ast})^{1/p}}. 
    \end{align*}
    Letting $k \wedge l \to \infty$ in this inequality, we obtain
    \[
    \limsup_{k \wedge l \to \infty}\abs{J_{p,r_{k}}^{\bm{k}}(f)^{1/p} - J_{p,r_{l}}^{\bm{k}}(f)^{1/p}} \le 2\varepsilon,
    \]
    which proves that $\{ J_{p,r_{n}}^{\bm{k}}(f) \}_{n \in \mathbb{N}}$ is a Cauchy sequence in $[0,\infty)$ and hence is convergent in $[0,\infty)$.
    Now we define $\KSform \colon \KS \to [0,\infty)$ by $\KSform(f) \coloneqq \lim_{n \to \infty}J_{p,r_{n}}^{\bm{k}}(f)$.

	Clearly, $(\mathcal{E}_{p}^{\bm{k}},B_{p,\infty}^{\bm{k}})$ is a $p$-energy form on $(K,m)$ satysfying \eqref{KSene-comp} by \ref{KSwm}.
    Let us show \ref{KS.GC}, \ref{it:difffunc}, \ref{KS.conti} and \ref{KS.slbdd} because the other properties \ref{KS.p-form} and \ref{KS.inv} are immediate from the expression of $J_{p,r}^{\bm{k}}(\,\cdot\,)$ and the definition of $\mathcal{E}_{p}^{\bm{k}}$.

    \ref{KS.GC},\ref{it:difffunc}:
    Obviously, \ref{it:difffunc} implies \ref{GC} for $(\KSform,\KS)$, so we first show \ref{it:difffunc}. 
    For simplicity, we consider the case $q_{2} < \infty$ (the case $q_{2} = \infty$ is similar).
    Let $\bm{u} = (u_{1},\dots,u_{n_{1}}) \in (B_{p,\infty}^{\bm{k}})^{n_{1}}$ and $\bm{v} = (v_{1},\dots,v_{n_{2}}) \in L^{p}(K,m)^{n_{2}}$ satisfy \eqref{e:FGC.cond}.
    Then the same argument as in \eqref{KSpre.GC} shows that for any $r \in (0,\infty)$, 
    \begin{equation}\label{e:KSform.FGC.pre}
    	\sum_{l = 1}^{n_{2}}J_{p,r}^{\bm{k}}(v_{l})^{q_{2}/p} 
    \le \left(\sum_{k = 1}^{n_{1}}J_{p,r}^{\bm{k}}(u_{k})^{q_{1}/p}\right)^{q_{2}/q_{1}},  
    \end{equation}
    which implies that $v_{l} \in \KS$ for any $l \in \{ 1,\dots,n_{2} \}$. 
    Using \eqref{e:defn-KSform} to take the limit of \eqref{e:KSform.FGC.pre} with $r = r_{n}$ as $n\to\infty$, we obtain $\sum_{l = 1}^{n_{2}}\KSform(v_{l})^{q_{2}/p} 
    \le \bigl(\sum_{k = 1}^{n_{1}}\KSform(u_{k})^{q_{1}/p}\bigr)^{q_{2}/q_{1}}$. 
    This completes the proof of \ref{it:difffunc}.
	
	Next we prove \eqref{KS-compatible}.
    We know that $\mathcal{E}_{p}^{\bm{k}}$ is Fr\'{e}chet differentiable on $B_{p,\infty}^{\bm{k}}$ by \eqref{e:frechet} in Proposition \ref{prop.c-diff}. 
    Also, by combining \eqref{c-diff} (in Proposition \ref{prop.c-diff}) for $J_{p,r}^{\bm{k}}$ and the convexity of $t \mapsto J_{p,r}^{\bm{k}}(f + tg)$, we see that for any $t \in (0, 1)$, 
    \begin{align*}
        &\abs{\frac{J_{p,r}^{\bm{k}}(f + tg) - J_{p,r}^{\bm{k}}(f)}{t} - pJ_{p,r}^{\bm{k}}(f;g)} \\
        &= \abs{\frac{J_{p,r}^{\bm{k}}(f + tg) - J_{p,r}^{\bm{k}}(f)}{t} - \left.\frac{d}{dt}J_{p,r}^{\bm{k}}(f + tg)\right|_{t = 0}} \\
        &\le \frac{J_{p,r}^{\bm{k}}(f + tg) + J_{p,r}^{\bm{k}}(f - tg) - 2J_{p,r}^{\bm{k}}(f)}{t}
        \overset{\eqref{c-diff}}{\le} O_{t}(f; g),
    \end{align*}
    where $O_{t}(f; g) = C_{p,f,g}t^{(p - 1) \wedge \frac{1}{p - 1}}$ for some constant $C_{p,f,g}$ which depends only on $p$, $\abs{f}_{B_{p,\infty}^{\bm{k}}}$ and $\abs{g}_{B_{p,\infty}^{\bm{k}}}$.
    Hence we see that
    \begin{align*}
    	&\limsup_{n \to \infty}\abs{\KSform(f; g) - J_{p,r_{n}}^{\bm{k}}(f;g)} \\
    	&\le \lim_{n \to \infty}\abs{\KSform(f; g) - \frac{1}{p}\cdot\frac{J_{p,r_{n}}^{\bm{k}}(f + tg) - J_{p,r_{n}}^{\bm{k}}(f)}{t}} + \frac{1}{p}O_{t}(f; g) \\
    	&= \abs{\KSform(f; g) - \frac{1}{p}\cdot\frac{\KSform(f + tg) - \KSform(f)}{t}} + \frac{1}{p}O_{t}(f; g)
    	\xrightarrow[t \downarrow 0]{} 0, 
    \end{align*}
    which shows \eqref{KS-compatible}. 

    \ref{KS.conti}:
    This is immediate from \eqref{e:Besov-compatible}, H\"{o}lder's inequality, Lemma \ref{lem.p-1} and \eqref{KS-compatible}.
    
    \ref{KS.slbdd}: 
    By \cite[Propositions 3.29 and 3.30]{KS.gc}, it suffices to show \eqref{e:KS-sl1}. 
    For simplicity, for $u \in L^{p}(K,m)$ and $E \in \mathcal{B}(K)$, define 
    \[
    \widetilde{J}_{p,r}^{\bm{k}}(u\:\vert\:E) \coloneqq \int_{E}\int_{B_{d}(x,\delta(r))}\abs{u(x) - u(y)}^{p}k_{r}(x,y)\,m(dy)m(dx), 
    \]
    and set $A_{i} \coloneqq \supp_{m}[f_{i} - a_{i}\indicator{K}]$ for $i \in \{ 1,2 \}$. 
    We also set $\widetilde{J}_{p,r}^{\bm{k}}(u) \coloneqq \widetilde{J}_{p,r}^{\bm{k}}(u\:\vert\:K)$. 
    Note that there exists $r_{0} \in (0,\infty)$ such that $\dist_{d}(A_{1},A_{2}) > 2\delta(r)$ for any $r \in (0,r_{0})$ since either $A_{1}$ or $A_{2}$ is compact. 
    Set $N_{r} \coloneqq K \setminus ((A_{1})_{d,\delta(r)} \cup (A_{2})_{d,\delta(r)})$ for $r \in (0,\infty)$. 
    Then for any $r \in (0,r_{0})$,
    \begin{align}
        &\widetilde{J}_{p,r}^{\bm{k}}(f_1 + f_2 + g) + \widetilde{J}_{p,r}^{\bm{k}}(g) \nonumber \\
        &= \widetilde{J}_{p,r}^{\bm{k}}(f_1 + g\:\vert\:(A_{1})_{d,\delta(r)}) + \widetilde{J}_{p,r}^{\bm{k}}(f_2 + g\:\vert\:(A_{2})_{d,\delta(r)}) + \widetilde{J}_{p,r}^{\bm{k}}(g\:\vert\:N_{r}) + \widetilde{J}_{p,r}^{\bm{k}}(g) \nonumber \\
        &= \widetilde{J}_{p,r}^{\bm{k}}(f_1 + g\:\vert\:(A_{1})_{d,\delta(r)}) + \widetilde{J}_{p,r}^{\bm{k}}(g\:\vert\:(A_{2})_{d,\delta(r)} \cup N_{r}) \nonumber \\
        &\quad + \widetilde{J}_{p,r}^{\bm{k}}(f_2 + g\:\vert\:(A_{2})_{d,\delta(r)}) + \widetilde{J}_{p,r}^{\bm{k}}(g\:\vert\:(A_{1})_{d,\delta(r)} \cup N_{r}) \nonumber \\
        &= \widetilde{J}_{p,r}^{\bm{k}}(f_1 + g) + \widetilde{J}_{p,r}^{\bm{k}}(f_2 + g). 
	\label{eq:KS-sl1-proof}
    \end{align}
    Noting that $\lim_{n \to \infty}\widetilde{J}_{p,r_{n}}^{\bm{k}}(u) = \KSform(u)$ for any $u \in \KS \cap L^{\infty}(K,m)$ by \eqref{e:defn-KSform} and the asymptotic locality of $\bm{k}$, we obtain \eqref{e:KS-sl1} by letting $r \coloneqq r_{n}$ and $n \to \infty$ in \eqref{eq:KS-sl1-proof} provided $f_{1},f_{2},g \in \KS \cap L^{\infty}(K,m)$.
	Finally, since $(-n) \vee (u \wedge n) \in \KS \cap L^{\infty}(K,m)$, $\lim_{n \to \infty}\KSform(u - (-n) \vee (u \wedge n)) = 0$ by \cite[Corollary 3.17]{KS.gc} and $\supp_{m}[u - c\indicator{K}] = \supp_{m}[(-n) \vee (u \wedge n) - c\indicator{K}]$ for any $u \in \KS$ and any $(n,c) \in \mathbb{N} \times \mathbb{R}$ with $n > \abs{c}$, \eqref{e:KS-sl1} extends to the remaining case $\{f_{1},f_{2},g\} \not\subseteq L^{\infty}(K,m)$ by the triangle inequality for $\KSform(\,\cdot\,)^{1/p}$, completing the proof.
\end{proof}

Next we would like to state further properties of $\bm{k}$-Korevaar--Shoen $p$-energy forms in the ``strongly $p$-recurrent'' case. 
To this end, we recall the notion of $p$-resistance form introduced in \cite{KS.gc} (see \cite{Kig01,Kig12} for the theory for the case $p = 2$).%
\begin{defn}[$p$-Resistance form]\label{defn.RFp}
    Let $K$ be a non-empty set.
    The pair $(\mathcal{E}, \mathcal{F})$ of $\mathcal{F} \subseteq \mathbb{R}^{K}$ and $\mathcal{E} \colon \mathcal{F} \to [0,\infty)$ is said to be a \emph{$p$-resistance form} on $K$ if and only if it satisfies the following conditions \ref{RF1}-\ref{RF5}:
    \begin{enumerate}[label=\textup{(RF\arabic*)$_p$},align=left,leftmargin=*,topsep=2pt,parsep=0pt,itemsep=2pt]
    \item\label{RF1} $\mathcal{F}$ is a linear subspace of $\mathbb{R}^{K}$ (containing $\mathbb{R}\indicator{K}$) and $\mathcal{E}(\,\cdot\,)^{1/p}$ is a seminorm on $\mathcal{F}$ satisfying $\{ u \in \mathcal{F} \mid \mathcal{E}(u) = 0 \} = \mathbb{R}\indicator{K}$.
    \item\label{RF2} The quotient normed space $(\mathcal{F}/\mathbb{R}\indicator{K}, \mathcal{E}^{1/p})$ is a Banach space.
    \item\label{RF3} If $x \neq y \in K$, then there exists $u \in \mathcal{F}$ such that $u(x) \neq u(y)$.
    \item\label{RF4} For any $x, y \in K$,
    \begin{equation}\label{R-def}
        R_{\mathcal{E}}(x,y) \coloneqq R_{(\mathcal{E},\mathcal{F})}(x,y) \coloneqq \sup\biggl\{ \frac{\abs{u(x) - u(y)}^{p}}{\mathcal{E}(u)} \biggm| u \in \mathcal{F} \setminus \mathbb{R}\indicator{K} \biggr\} < \infty.
    \end{equation}
    \item\label{RF5} $(\mathcal{E},\mathcal{F})$ satisfies \ref{GC}.
    \end{enumerate}
\end{defn}

We also need to recall the following standard notions on the metric $d$ and the measure $m$.
\begin{defn}\label{defn.doubling}
    Let $Q \in (0,\infty)$.
    \begin{enumerate}[label=\textup{(\arabic*)},align=left,leftmargin=*,topsep=2pt,parsep=0pt,itemsep=2pt]
    	\item The metric $d$ is said to be \emph{metric doubling} if and only if for any $\delta \in (0,1)$ there exists $N \in \mathbb{N}$ such that for any $(x,r) \in K \times (0,\infty)$ we can find $\{ x_{j} \}_{j = 1}^{N} \subseteq K$ so that $B_{d}(x,r) \subseteq \bigcup_{j = 1}^{N}B_{d}(x_{j},\delta r)$. 
    	\item The measure $m$ is said to be \emph{volume doubling with growth exponent $Q$} (with respect to the metric $d$) if and only if there exists $C_{\mathrm{D}}' \in (0,\infty)$ such that 
   		    \begin{equation}\label{VD-growth}
        		m(B_{d}(x,s)) \le C_{\mathrm{D}}'\Bigl(\frac{s}{r}\Bigr)^{Q} m(B_{d}(x,r)) < \infty \quad \text{for any $x \in K$ and any $0 < r \le s$.}
    		\end{equation}
			Note that $m$ is volume doubling with growth exponent $Q'$ for some $Q' \in (0,\infty)$ if and only if $m$ is \emph{volume doubling}, i.e., there exists $C_{\mathrm{D}} \in (0,\infty)$ such that 
   		    \begin{equation}\label{VD}
        		m(B_{d}(x,2s)) \le C_{\mathrm{D}}m(B_{d}(x,s)) < \infty \quad \text{for any $(x,s) \in K \times (0,\infty)$.}
    		\end{equation}
    	\item The measure $m$ is said to be \emph{$Q$-Ahlfors regular} (with respect to the metric $d$) if and only if there exists $C_{\mathrm{AR}} \in [1,\infty)$ such that 
    \begin{equation}\label{AR}
        C_{\mathrm{AR}}^{-1}\,s^{Q} \le m(B_{d}(x,s)) \le C_{\mathrm{AR}}\,s^{Q} \quad \text{for any $(x,s) \in K \times (0,2\diam(K,d))$.}
    \end{equation}
    \end{enumerate}
\end{defn}
The $Q$-Ahlfors regularity of $m$ clearly implies that $m$ is volume doubling with growth exponent $Q$, and it is also well known that the volume doubling property of $m$ with respect to $d$ implies the metric doubling property of $d$.

Now we give a sufficient condition for a $\bm{k}$-Korevaar--Schoen $p$-energy form $(\KSform,\KS)$ on $(K,m)$ to be a $p$-resistance form on $K$.
\begin{prop}\label{prop.KS-RF}
    Suppose that there exist $Q,\beta_{p} \in (0,\infty)$ with $\beta_{p} > Q$ such that the following hold:  
    \begin{enumerate}[label=\textup{(\roman*)},align=left,leftmargin=*,topsep=2pt,parsep=0pt,itemsep=2pt]
        \item The measure $m$ satisfies $m(K) < \infty$ and is volume doubling with growth exponent $Q \in (0,\infty)$.
        \item\label{it:KS-RF-KSwm} \ref{KSwm} holds.
        \item\label{it:RF-irred} $\bigl\{ u \in B_{p,\infty}^{\bm{k}} \bigm| \sup_{r > 0}J_{p,r}^{\bm{k}}(u) = 0 \bigr\} = \mathbb{R}\indicator{K}$.
        \item $\KS \subseteq \contfunc(K)$, and there exists $C \in (0,\infty)$ such that for any $f \in \KS$ and any $x,y \in K$, 
        \begin{equation}\label{morrey}
            \abs{f(x) - f(y)} \le Cd(x,y)^{(\beta_{p} - Q)/p}\sup_{r > 0}J_{p,r}^{\bm{k}}(f)^{1/p}, \quad x,y \in K.
        \end{equation}
        \item There exists $C \in (0,\infty)$ such that for any $(x,s) \in K \times (0,\infty)$ with $B_{d}(x,s) \neq K$,
        \begin{align}\label{capu}
            \inf
            &\biggl\{ \sup_{r > 0}J_{p,r}^{\bm{k}}(f) \biggm| f \in \contfunc(K), \supp_{K}[f] \subseteq B_{d}(x, 2s), \text{$f \ge 1$ on $B_{d}(x,s)$} \biggr\} \nonumber \\
    		&\le C\frac{m(B_{d}(x,s))}{s^{\beta_{p}}}.  
        \end{align}
    \end{enumerate}
    Then any $\bm{k}$-Korevaar--Schoen $p$-energy form $(\mathcal{E}_{p}^{\bm{k}},B_{p,\infty}^{\bm{k}})$ on $(K,m)$, which exists by \ref{it:KS-RF-KSwm} and Theorem \ref{thm.KS-energy}, is a $p$-resistance form on $K$.
    If in addition $m$ is $Q$-Ahlfors regular, then there exist $\alpha_{0},\alpha_{1} \in (0,\infty)$ such that for any such $(\mathcal{E}_{p}^{\bm{k}},B_{p,\infty}^{\bm{k}})$,
    \begin{equation}\label{Rp-compl}
        \alpha_{0}d(x,y)^{\beta_p - Q} \le R_{\mathcal{E}_{p}^{\bm{k}}}(x,y) \le \alpha_{1}d(x,y)^{\beta_p - Q} \quad \text{for any $x,y \in K$.}
    \end{equation}
\end{prop}
\begin{proof}
    Let $(\mathcal{E}_{p}^{\bm{k}},B_{p,\infty}^{\bm{k}})$ be a $\bm{k}$-Korevaar--Schoen $p$-energy form on $(K,m)$.
    We shall show that $(\mathcal{E}_{p}^{\bm{k}},B_{p,\infty}^{\bm{k}})$ is a $p$-resistance form on $K$.
    \ref{RF1} and \ref{RF5} are clear from Theorem \ref{thm.KS-energy} and \ref{it:RF-irred}.
    The condition \eqref{capu} immediately implies \ref{RF3}.
    By \eqref{morrey} and the lower inequality in \eqref{KSene-comp} we have $R_{\mathcal{E}_{p}^{\bm{k}}}(x,y) \lesssim d(x,y)^{\beta_{p} - Q}$ for any $x,y \in K$, whence \ref{RF4} and the upper estimate in \eqref{Rp-compl} hold.
    In particular, $\sup_{x,y \in K}R_{\mathcal{E}_{p}^{\bm{k}}}(x,y) < \infty$.
    To prove \ref{RF2}, we see from \eqref{morrey} that for any $f \in B_{p,\infty}^{\bm{k}}$,
    \begin{align}\label{e:gap}
        \int_{K}\abs{f(x) - \fint_{K}f\,dm}^{p}\,m(dx)
        &\le \int_{K}\fint_{K}\abs{f(x) - f(y)}^{p}\,m(dy)m(dx) \nonumber \\
        &\lesssim \biggl(\sup_{x,y \in K}R_{\mathcal{E}_{p}^{\bm{k}}}(x,y)\biggr)\mathcal{E}_{p}^{\bm{k}}(f)m(K).
    \end{align}
    Let $\{ f_{n} \}_{n \in \mathbb{N}} \subseteq B_{p,\infty}^{\bm{k}}$ be a Cauchy sequence in $(B_{p,\infty}^{\bm{k}}/\mathbb{R}\indicator{K}, \mathcal{E}_{p}^{\bm{k}}(\,\cdot\,)^{1/p})$ with $\fint_{K}f_{n}\,dm = 0$.
    Then \eqref{e:gap} implies that $\{ f_{n} \}_{n \in \mathbb{N}}$ is a Cauchy sequence in $L^{p}(K,m)$, and thus $\{ f_{n} \}_{n \in \mathbb{N}}$ is a Cauchy sequence in $B_{p,\infty}^{\bm{k}}$.
    Since $B_{p,\infty}^{\bm{k}}$ is a Banach space by Theorem \ref{thm.banach}, we conclude that $(B_{p,\infty}^{\bm{k}}/\mathbb{R}\indicator{K}, \mathcal{E}_{p}^{\bm{k}}(\,\cdot\,)^{1/p})$ is also a Banach space.

    Next we show the lower estimate in \eqref{Rp-compl} under the assumption that $m$ is $Q$-Ahlfors regular.
    Let $x,y \in K$ and let $s > 0$ satisfy $d(x,y) > 2s \ge 2^{-1}d(x,y)$.
    Then $B_{d}(x,s) \neq \emptyset$.
    By \eqref{capu}, there exists $f \in B_{p,\infty}^{\bm{k}} \cap \contfunc(K)$ such that $\supp_{K}[f] \subseteq B_{d}(x,s)$, $f \ge 1$ on $B_{d}(x,2s)$ and $\mathcal{E}_{p}^{\bm{k}}(f) \le C_{1}s^{Q - \beta_{p}}$, where $C_{1} \in (0,\infty)$ depends only on $C$ in \eqref{capu} and $C_{\mathrm{AR}}$ in \eqref{AR}.
    Hence we have
    \[
    R_{\mathcal{E}_{p}^{\bm{k}}}(x,y) \ge \mathcal{E}_{p}^{\bm{k}}(f)^{-1} \ge C_{1}^{-1}s^{\beta_{p} - Q} \gtrsim d(x,y)^{\beta_{p} - Q}. 
    \qedhere\]
\end{proof}

\begin{exam}[Korevaar--Schoen--Sobolev space]\label{ex.KS}
    In addition to the setting specified at the beginning of this section, we suppose that $K$ is connected and that $m(B_{d}(x,r)) < \infty$ for any $(x,r) \in K \times (0,\infty)$. 
    For $s > 0$, define $\bm{k}^{s} = \{ k_{r}^{s} \}_{r > 0}$ by
    \begin{equation}\label{KSkernel}
    	k_{r}^{s}(x,y) \coloneqq \frac{\indicator{B_{d}(x,r)}(y)}{r^{ps}m(B_{d}(x,r))}, \quad x,y \in K.
    \end{equation}
    Clearly, $\bm{k}^{s}$ is asymptotically local.
	We define the \emph{Besov--Lipschitz space} $B_{p, \infty}^{s}$ by $B_{p, \infty}^{s} \coloneqq B_{p, \infty}^{\bm{k}^{s}}$.
	Then the \emph{critical $L^p$-Besov exponent} $s_{p}$ of $(K, d, m)$ is defined as
	\begin{equation}\label{Lp-Besov}
		s_{p} \coloneqq \sup\bigl\{ s \in (0,\infty) \bigm| \text{$B_{p, \infty}^{s}$ contains a non-constant function} \bigr\}.
	\end{equation}
	We call $\mathrm{KS}^{1,p} \coloneqq B_{p, \infty}^{s_{p}}$ the \emph{$(1,p)$-Korevaar--Schoen--Sobolev space} on $(K,d,m)$.
	We also write $\mathrm{KS}^{1,p}(K,d,m)$ for $\mathrm{KS}^{1,p}$ when we would like to clarify the underlying metric measure space $(K,d,m)$.
	If $m$ is $Q$-Ahlfors regular with respect to $d$ for some $Q \in (0,\infty)$, then $\bm{k}^{s_{p},Q} = \bigl\{ k_{r}^{s_{p}, Q} \bigr\}_{r > 0}$ given by 
	\[
	k_{r}^{s_{p}, Q}(x,y) \coloneqq r^{- ps_{p} - Q}\indicator{B_{d}(x,r)}(y), \quad x,y \in K, 
	\]
	which again is obviously asymptotically local, also corresponds to the $(1,p)$-Korevaar--Schoen--Sobolev space, i.e., $B_{p,\infty}^{\bm{k}^{s_{p},Q}} = \mathrm{KS}^{1,p}$. 
	If \hyperref[KSwm]{\textup{(WM)$_{p,\bm{k}^{s_{p}}}$}} holds, then we write $\mathcal{E}_{p}^{\mathrm{KS}}$ instead of $\mathcal{E}_{p}^{\bm{k}^{s_{p}}}$ and call each $\bm{k}^{s_{p}}$-Korevaar--Schoen $p$-energy form $(\mathcal{E}_{p}^{\mathrm{KS}},\mathrm{KS}^{1,p})$ on $(K,m)$ a \emph{Korevaar--Schoen $p$-energy form} on $(K,d,m)$. 
	
	It is not easy in general to verify \ref{KSwm} and \eqref{capu} for the family of kernels $\bm{k} = \bm{k}^{s_{p}}$; see Sections \ref{sec.Kig} and \ref{sec.CGQ} for some settings in which we can prove these conditions. 
	On the other hand, a reasonable sufficient condition for \eqref{morrey} is known.
	In fact, if $m$ is volume doubling with growth exponent $Q \in (0,\infty)$ and $ps_{p} > Q$, then \eqref{morrey} holds for $\mathrm{KS}^{1,p}$; see, e.g., \cite[Theorem 5.1]{AB21} or \cite[Theorem 3.2]{Bau22+}.
	
	Let us give a couple of other examples of families of kernels $\bm{k}$, whose associated Besov spaces $B_{p,\infty}^{\bm{k}}$ coincide with $\mathrm{KS}^{1,p}$ under suitable assumptions. 
	The first one $\bm{k}^{\#} = \{ k^{\#}_{r} \}_{r > 0}$ is a variant of $\bm{k}^{s_{p}}$ obtained by replacing $r^{ps_{p}}$ in \eqref{KSkernel} for $s = s_{p}$ with $d(x,y)^{ps_{p}}$, i.e., defined by 
	\begin{equation}\label{KSkernel.dist}
		k_{r}^{\#}(x,y) \coloneqq \frac{\indicator{B_{d}(x,r)}(y)}{d(x,y)^{ps_{p}}m(B_{d}(x,r))}, \quad x,y \in K,
	\end{equation}
	so that $\bm{k}^{\#}$ is clearly asymptotically local. 
	When $m$ is volume doubling and $(K,d,m)$ is equipped with a pair of $p$-energy form and $p$-energy measures satisfying a suitable Poincar\'{e} inequality and a capacity upper estimate as in the cases of many examples including the Sierpi\'{n}ski carpet, one can show that $\bm{k}^{\#}$ satisfies $B_{p,\infty}^{\bm{k}^{\#}} = \mathrm{KS}^{1,p}$ and \hyperref[KSwm]{\textup{(WM)$_{p,\bm{k}^{\#}}$}}; see \cite[Corollary 1.14]{Shi24+} for details. 
	As Proposition \ref{prop:KSidentify} in Appendix \ref{app:WMdistance}, we give an alternative elementary proof that a Poincar\'e-type inequality as given in \eqref{e:KSPI} implies $B_{p,\infty}^{\bm{k}^{\#}} = \mathrm{KS}^{1,p}$ and \hyperref[KSwm]{\textup{(WM)$_{p,\bm{k}^{\#}}$}}. 
	
	The second family of kernels $\bm{k}^{\mathrm{heat}}$ is a mollification of $\bm{k}^{s_{p}}$ obtained by replacing $m(B_{d}(x,r))^{-1} \indicator{B_{d}(x,r)}(y)$ in \eqref{KSkernel} for $s = s_{p}$ with the heat kernel of a diffusion on $K$. 
	Namely, assuming that $(K,d)$ is locally compact, that $m$ is a Radon measure on $K$, and that $(K,d,m)$ is equipped with a strongly local regular symmetric Dirichlet form $(\mathcal{E},\mathcal{F})$ on $L^{2}(K,m)$ which has a heat kernel $\{ q_{t} \}_{t > 0}$\footnote{I.e., a family $\{ q_{t} \}_{t > 0}$ of $[0,\infty]$-valued Borel measurable functions on $K \times K$ such that $T_{t}f = \int_{K}q_{t}(\cdot,y)f(y)\,m(dy)$ $m$-a.e.\ on $K$ for any $t \in (0,\infty)$ and any $f \in L^{2}(K,m)$, where $\{T_{t}\}_{t>0}$ denotes the Markovian semigroup on $L^{2}(K,m)$ associated with $(\mathcal{E},\mathcal{F})$; see \cite[Sections 1.1, 1.3 and 1.4]{FOT} for the definitions of the relevant notions from the theory of symmetric Dirichlet forms.}, we define $\bm{k}^{\mathrm{heat}} = \{ k^{\mathrm{heat}}_{r} \}_{r > 0}$ by 
	\begin{equation}\label{KSkernel.heat}
	k^{\mathrm{heat}}_{r}(x,y) \coloneqq \frac{q_{r^{\beta}}(x,y)}{r^{ps_{p}}}, \quad x,y \in K, 
	\end{equation}
	where $\beta \in (1,\infty)$ is a parameter to be suitably chosen depending on $(K,d,m,\mathcal{E},\mathcal{F})$.
	This family of kernels has been considered in \cite{AB21,ABCRST,Bau22+,GYZ23+,PP10} under the assumptions that $(K,d)$ is complete and that the following \emph{(full off-diagonal) sub-Gaussian heat kernel estimates with walk dimension $\beta$} hold:
	there exist $C_{1}, c_{1}, C_{2}, c_{2} \in (0,\infty)$ such that for each $t \in (0,\infty)$, 
	\begin{align}\label{e:HKE}
		&\frac{C_{1}}{m(B(x,t^{1/\beta}))}\exp\Biggl(-c_{1}\biggl(\frac{d(x,y)^{\beta}}{t}\biggr)^{\frac{1}{\beta - 1}}\Biggr) 
		\le q_{t}(x,y) \nonumber \\
		&\qquad\qquad\le \frac{C_{2}}{m(B(x,t^{1/\beta}))}\exp\Biggl(-c_{2}\biggl(\frac{d(x,y)^{\beta}}{t}\biggr)^{\frac{1}{\beta - 1}}\Biggr) 
		\quad \text{for $m$-a.e.\ $x,y \in K$;}
	\end{align}
	note that \eqref{e:HKE} implies that $m$ is volume doubling (see, e.g., \cite[Remark 1.2-(1)]{Kaj20}). 
	In particular, under these assumptions, it has been proved in \cite{GYZ23+} that $B_{p,\infty}^{\bm{k}^{\mathrm{heat}}} = \mathrm{KS}^{1,p}$ (\cite[Lemmas 3.3 and 3.4]{GYZ23+}) and that \hyperref[KSwm]{\textup{(WM)$_{p,\bm{k}^{s_{p}}}$}} and \hyperref[KSwm]{\textup{(WM)$_{p,\bm{k}^{\mathrm{heat}}}$}} are equivalent to each other (\cite[Theorem 1.7]{GYZ23+}).
	It is also easy to see that, if $m$ is volume doubling, the upper inequality in \eqref{e:HKE} holds and $m(K) < \infty$, then $\bm{k}^{\mathrm{heat}}$ is asymptotically local. 
	On the other hand, even if \eqref{e:HKE} holds, $\bm{k}^{\mathrm{heat}}$ is not necessarily asymptotically local when $m(K) = \infty$, as can be seen from the case of the canonical Dirichlet form on $\mathbb{R}^{n}$, \eqref{p-form.Euc} for $p=2$ with domain $W^{1,2}(\mathbb{R}^{n})$, where $s_{p}=1$ as mentioned in the introduction, $\beta=2$ and $q_{t}(x,y)=(4\pi t)^{-n/2}e^{-\abs{x-y}^{2}/(4t)}$ for any $(t,x,y) \in (0,\infty) \times \mathbb{R}^{n} \times \mathbb{R}^{n}$. 
\end{exam}


\section{Associated $p$-energy measures and chain rule}\label{sec.pEMKS}
\setcounter{equation}{0}
Next in this section, we introduce the $p$-energy measures associated with a given $\bm{k}$-Korevaar--Schoen $p$-energy form $(\KSform,\KS)$, and show their basic properties.

Throughout this section, as in the previous section, we fix $p \in (1,\infty)$, a separable metric space $(K,d)$ with $\#K \ge 2$ and a $\sigma$-finite Borel measure $m$ on $K$ with full topological support.  
In addition, we suppose that $(K,d)$ is locally compact. 
We also fix a family of kernels $\bm{k} = \{ k_{r} \}_{r > 0}$ as in Definition \ref{defn.kernel}, suppose that $\bm{k}$ is asymptotically local and that \ref{KSwm} holds, and fix an arbitrary sequence $\{ r_{n} \}_{n \in \mathbb{N}} \subseteq (0,\infty)$ as in Theorem \ref{thm.KS-energy}, so that we have the $\bm{k}$-Korevaar--Schoen $p$-energy form $(\KSform,\KS)$ on $(K,m)$ along $\{ r_{n} \}_{n \in \mathbb{N}}$ defined by \eqref{e:defn-KSform}.
For ease of notation, we set 
\[
m_{n}(dxdy) \coloneqq k_{r_{n}}(x,y)\,m(dy)m(dx). 
\]
For each $u\in B_{p,\infty}^{\bm{k}} \cap L^{\infty}(K,m)$, define a linear map $\Psi_{p}^{\bm{k}}(u; \,\cdot\,) \colon B_{p,\infty}^{\bm{k}} \cap L^{\infty}(K,m) \to \mathbb{R}$ by, for each $\varphi \in B_{p,\infty}^{\bm{k}} \cap L^{\infty}(K,m)$, 
\begin{equation}\label{emfunc}
    \Psi_{p}^{\bm{k}}(u; \varphi) \coloneqq \mathcal{E}_{p}^{\bm{k}}(u; u\varphi) - \left(\frac{p - 1}{p}\right)^{p - 1}\mathcal{E}_{p}^{\bm{k}}\bigl(\abs{u}^{\frac{p}{p - 1}}; \varphi\bigr).
\end{equation}
(Note that $u\varphi,\abs{u}^{\frac{p}{p - 1}} \in B_{p,\infty}^{\bm{k}}$ by Theorem \ref{thm.KS-energy}-\ref{KS.GC} and Proposition \ref{prop.GClist}-\ref{GC.leipniz},\ref{GC.lip}.)
\begin{thm}\label{thm.KSfunctional}
	Let $u \in B_{p,\infty}^{\bm{k}} \cap \contfunc_{b}(K)$ and $\varphi \in B_{p,\infty}^{\bm{k}} \cap L^{\infty}(K,m)$. 
    If $\{ u,\varphi \} \cap \contfunc_{c}(K) \neq \emptyset$, then 
    \begin{gather}
        \begin{split}
        \Psi_{p}^{\bm{k}}(u; \varphi) &= \lim_{n \to \infty}\int_{K}\int_{K}\abs{u(x) - u(y)}^{p}\varphi(x)k_{r_{n}}(x,y)\,m(dy)m(dx) \\
        &= \lim_{n \to \infty}\int_{K}\int_{K}\abs{u(x) - u(y)}^{p}\varphi(y)k_{r_{n}}(x,y)\,m(dy)m(dx), \label{emfunc.expression}
        \end{split}\\
        \abs{\Psi_{p}^{\bm{k}}(u; \varphi)} \le \norm{\varphi}_{L^{\infty}(K,m)}\mathcal{E}_{p}^{\bm{k}}(u).
    \label{KSpem.bdd}
	\end{gather}
    In particular, if in addition $\varphi \ge 0$, then $\Psi_{p}^{\bm{k}}(u; \varphi) \ge 0$. 
\end{thm}
\begin{proof}
    First, we observe that
    \begin{align}\label{emfunctional}
        &\Psi_{p,n}^{\bm{k}}(u; \varphi) \coloneqq J_{p,r_{n}}^{\bm{k}}(u; u\varphi) - \left(\frac{p - 1}{p}\right)^{p - 1}J_{p,r_{n}}^{\bm{k}}\bigl(\abs{u}^{\frac{p}{p - 1}}; \varphi\bigr) \nonumber \\
        &= \int_{K \times K}\biggl[\abs{u(x) - u(y)}^{p}\varphi(x) + \gamma_{p}\bigl(u(x) - u(y)\bigr) \cdot (\varphi(x) - \varphi(y))u(y) \nonumber \\
        &\quad - \left(\frac{p - 1}{p}\right)^{p - 1}\gamma_{p}\Bigl(\abs{u(x)}^{\frac{p}{p - 1}} - \abs{u(y)}^{\frac{p}{p - 1}}\Bigr) \cdot (\varphi(x) - \varphi(y))\biggr]\,m_{n}(dxdy).
    \end{align}
   	Define $F_{n} \in \mathcal{B}(K \times K)$ by
    \begin{equation*}
    	F_{n} \coloneqq \{ (x,y) \in K \times K \mid \text{$d(x,y) < \delta(r_n)$ and $(\varphi(x),\varphi(y)) \not= (0,0)$} \},
    \end{equation*} 
    and set 
    \begin{align*}
    	&I_{p,n}^{\bm{k}}(u; \varphi) \\
    	&\coloneqq \int_{F_{n}}\biggl[\abs{u(x) - u(y)}^{p}\varphi(x) + \gamma_{p}\bigl(u(x) - u(y)\bigr) \cdot (\varphi(x) - \varphi(y))u(y) \nonumber \\
        &\quad - \left(\frac{p - 1}{p}\right)^{p - 1}\gamma_{p}\Bigl(\abs{u(x)}^{\frac{p}{p - 1}} - \abs{u(y)}^{\frac{p}{p - 1}}\Bigr) \cdot (\varphi(x) - \varphi(y))\biggr]\,m_{n}(dxdy).
    \end{align*}
    Note that $\lim_{n \to \infty}\bigl(\Psi_{p,n}^{\bm{k}}(u; \varphi) - I_{p,n}^{\bm{k}}(u; \varphi)\bigr) = 0$ by \eqref{e:asylocal} and $\norm{u}_{\sup} \vee \norm{\varphi}_{L^{\infty}(K,m)} < \infty$. 
    Since $\closure{F_{n}}^{K \times K}$ is compact for sufficiently large $n \in \mathbb{N}$ when $\varphi \in B_{p,\infty}^{\bm{k}} \cap \contfunc_{c}(K)$, $u$ is uniformly continuous on $\{ x \in K \mid \text{$(x,y) \in F_{n}$ or $(y,x) \in F_{n}$ for some $y \in K$} \}$ for such $n$. 
    By combining this observation with the uniform continuity of $t \mapsto \abs{t}^{1/(p - 1)}\sgn(t)$ on $u(K)$, for any $\varepsilon > 0$, we can find $N \in \mathbb{N}$ such that
    \begin{align}\label{approx.power}
        &\abs{\frac{p - 1}{p}\Bigl(\abs{u(x)}^{\frac{p}{p - 1}} - \abs{u(y)}^{\frac{p}{p - 1}}\Bigr) - \bigl(u(x) - u(y)\bigr)\abs{u(y)}^{\frac{1}{p - 1}}\sgn\bigl(u(y)\bigr)} \nonumber \\
        &= \abs{\int^{u(x)}_{u(y)}\Bigl[\abs{t}^{\frac{1}{p - 1}}\sgn(t) - \abs{u(y)}^{\frac{1}{p - 1}}\sgn\bigl(u(y)\bigr)\Bigr]\,dt}
        \le \varepsilon\abs{u(x) - u(y)}
    \end{align}
    for any $(x,y) \in \bigcup_{n \ge N}F_{n}$.
    Using Lemma \ref{lem.p-1}, \eqref{approx.power} and H\"{o}lder's inequality, we can find $C_{p,u} \in (0,\infty)$ depending only on $p$ and $\norm{u}_{\sup}$ such that
    \begin{align*}
        &\sup_{n \ge N}\Biggl|\int_{F_{n}}\biggl[\gamma_{p}\bigl(u(x) - u(y)\bigr) \cdot (\varphi(x) - \varphi(y))u(y) \\
        &\quad - \left(\frac{p - 1}{p}\right)^{p - 1}\gamma_{p}\Bigl(\abs{u(x)}^{\frac{p}{p - 1}} - \abs{u(y)}^{\frac{p}{p - 1}}\Bigr) \cdot (\varphi(x) - \varphi(y))\biggr]\,m_{n}(dxdy)\Biggr| \\
        &\le C_{p,u}\varepsilon^{(p - 1) \wedge 1}\mathcal{E}_{p}^{\bm{k}}(u)^{\frac{(p - 1) \wedge 1}{p}}\mathcal{E}_{p}^{\bm{k}}(\varphi)^{\frac{1}{p}} \eqqcolon C_{p,u,\varphi}\varepsilon^{(p - 1) \wedge 1}.
    \end{align*}
    Therefore, \eqref{emfunctional} implies that for any $n \ge N$,
    \begin{align*}
        &\biggl\lvert \Psi_{p,n}^{\bm{k}}(u; \varphi) - \int_{K \times K}\lvert u(x) - u(y)\rvert^{p}\varphi(x)\,m_{n}(dxdy)\biggr\rvert \nonumber \\
        &\le \bigl\lvert\Psi_{p,n}^{\bm{k}}(u; \varphi) - I_{p,n}^{\bm{k}}(u; \varphi)\bigr\rvert + \int_{F_{n}^{c}}\lvert u(x) - u(y)\rvert^{p}\varphi(x)\,m_{n}(dxdy) + C_{p,u,\varphi}\varepsilon^{(p - 1) \wedge 1},
    \end{align*}
    which together with $\lim_{n \to \infty}\Psi_{p,n}^{\bm{k}}(u; \varphi) = \Psi_{p}^{\bm{k}}(u; \varphi)$ and \eqref{e:asylocal} yields the first equality in \eqref{emfunc.expression}. 
    The second equality in \eqref{emfunc.expression} can be shown similarly by using the expression 
    \begin{align*} 
        \Psi_{p,n}^{\bm{k}}(u; \varphi) 
        &= \int_{K \times K}\biggl[\abs{u(x) - u(y)}^{p}\varphi(y) + \gamma_{p}\bigl(u(x) - u(y)\bigr) \cdot (\varphi(x) - \varphi(y))u(x) \nonumber \\
        &\, - \left(\frac{p - 1}{p}\right)^{p - 1}\gamma_{p}\Bigl(\abs{u(x)}^{\frac{p}{p - 1}} - \abs{u(y)}^{\frac{p}{p - 1}}\Bigr) \cdot (\varphi(x) - \varphi(y))\biggr]\,m_{n}(dxdy)
    \end{align*}
    instead of \eqref{emfunctional}. 
    Now the estimate \eqref{KSpem.bdd} is clear from \eqref{emfunc.expression}.
\end{proof}

By Theorem \ref{thm.KSfunctional}, we can associate to the functional $\Psi_{p}^{\bm{k}}(u; \,\cdot\,)$ a unique Radon measure $\KSem\langle u \rangle$ on $K$ under the additional assumption that $\KS \cap \contfunc_{c}(K)$ is dense in $(\contfunc_{c}(K), \norm{\,\cdot\,}_{\sup})$, as follows. 
\begin{thm}\label{thm.KSpem-exist}
	Suppose that $B_{p,\infty}^{\bm{k}} \cap \contfunc_{c}(K)$ is dense in $(\contfunc_{c}(K), \norm{\,\cdot\,}_{\sup})$. 
	Let $u \in B_{p,\infty}^{\bm{k}} \cap \contfunc_{b}(K)$. 
	Then there exists a unique positive Radon measure $\Gamma_{p}^{\bm{k}}\langle u \rangle$ on $K$ such that for any $\varphi \in \KS \cap \contfunc_{c}(K)$, 
    \begin{equation}\label{KSpem.characterize}
        \int_{K}\varphi\,d\KSem\langle u \rangle
        = \KSform(u; u\varphi) - \left(\frac{p - 1}{p}\right)^{p - 1}\KSform\bigl(\abs{u}^{\frac{p}{p - 1}}; \varphi\bigr). 
    \end{equation}
    Moreover, $\KSem\langle u \rangle(K) \le \KSform(u) < \infty$, and for any $\varphi \in \contfunc_{0}(K)$, 
    \begin{equation}\label{KSpem.limext}
    	\int_{K}\varphi\,d\Gamma_{p}^{\bm{k}}\langle u \rangle 
    	= \lim_{n \to \infty}\int_{K}\int_{K}\abs{u(x) - u(y)}^{p}\varphi(x)k_{r_{n}}(x,y)\,m(dy)m(dx).
    \end{equation}
\end{thm}
\begin{defn}[$p$-Energy measure associated with a $\bm{k}$-Korevaar--Schoen $p$-energy form $(\KSform,\KS)$]\label{d:KSpem}
	Suppose that $\KS \cap \contfunc_{c}(K)$ is dense in $(\contfunc_{c}(K),\norm{\,\cdot\,}_{\sup})$, and let $u \in \KS \cap \contfunc_{b}(K)$.
	The positive Radon measure $\KSem\langle u \rangle$ on $K$ as in Theorem \ref{thm.KSpem-exist} is called the \emph{$p$-energy measure of $u$ associated with $(\KSform,\KS)$}.  
\end{defn}
\begin{proof}[Proof of Theorem \ref{thm.KSpem-exist}]
	By virtue of \eqref{KSpem.bdd}, we can extend $\Psi_{p}^{\bm{k}}(u;\,\cdot\,)$ to a bounded linear functional on $\contfunc_{0}(K)$ in a standard way as follows. 
	Let $u \in B_{p,\infty}^{\bm{k}} \cap \contfunc_{b}(K)$, let $\varphi \in \contfunc_{0}(K)$ and choose $\{ \varphi_{j} \}_{j \in \mathbb{N}} \subseteq B_{p,\infty}^{\bm{k}} \cap \contfunc_{c}(K)$ so that $\lim_{j \to \infty}\norm{\varphi - \varphi_{j}}_{\sup} = 0$. 
	Then $\{ \Psi_{p}^{\bm{k}}(u; \varphi_{j}) \}_{j \in \mathbb{N}}$ is a Cauchy sequence in $\mathbb{R}$ since $\abs{\Psi_{p}^{\bm{k}}(u; \varphi_{j}) - \Psi_{p}^{\bm{k}}(u; \varphi_{j'})} \le \norm{\varphi_{j} - \varphi_{j'}}_{\sup}\mathcal{E}_{p}^{\bm{k}}(u)$ for any $j,j' \in \mathbb{N}$ by \eqref{KSpem.bdd}. 
	Now we define $\widetilde{\Psi}_{p}^{\bm{k}}(u; \varphi) \coloneqq \lim_{j \to \infty}\Psi_{p}^{\bm{k}}(u; \varphi_{j})$, which does not depend on the choice of $\{ \varphi_{j} \}_{j \in \mathbb{N}}$. 
	Clearly, we have $\abs{\widetilde{\Psi}_{p}^{\bm{k}}(u; \varphi)} \le \norm{\varphi}_{\sup}\mathcal{E}_{p}^{\bm{k}}(u)$. 
	If $\varphi \ge 0$, then we obtain $\widetilde{\Psi}_{p}^{\bm{k}}(u; \varphi) \ge 0$ by considering $\{ \varphi_{j}^{+} \}_{j}$ instead of $\{ \varphi_{j} \}_{j}$. 
	By applying the Riesz--Markov--Kakutani representation theorem (see, e.g., \cite[Theorems 2.14 and 2.18]{Rud}), there exists a unique positive Radon measure $\Gamma_{p}^{\bm{k}}\langle u \rangle$ on $K$ satisfying 
	\begin{equation}\label{KSpem.defn}
		\widetilde{\Psi}_{p}^{\bm{k}}(u; \psi) = \int_{K}\psi\,d\Gamma_{p}^{\bm{k}}\langle u \rangle \quad \text{for any $\psi \in \contfunc_{c}(K)$.}
	\end{equation}
	In particular, $\KSem\langle u \rangle$ satisfies \eqref{KSpem.characterize} for any $\varphi \in \KS \cap \contfunc_{c}(K)$ by \eqref{KSpem.defn} and \eqref{emfunc}. 
	
	Next, to show the claimed uniqueness of $\Gamma_{p}^{\bm{k}}\langle u \rangle$ and
	$\Gamma_{p}^{\bm{k}}\langle u \rangle(K) \le \mathcal{E}_{p}^{\bm{k}}(u)$,
	let $\mu$ be a positive Radon measure on $K$ satisfying \eqref{KSpem.characterize}
	with $\mu$ in place of $\Gamma_{p}^{\bm{k}}\langle u \rangle$ for any $\varphi \in B_{p,\infty}^{\bm{k}} \cap \contfunc_{c}(K)$.
	Then for any compact subset $F$ of $K$, noting \eqref{Bpinftyk} and the assumption that $B_{p,\infty}^{\bm{k}} \cap \contfunc_{c}(K)$ is dense in $(\contfunc_{c}(K),\norm{\,\cdot\,}_{\sup})$,
	we can choose $\varphi \in B_{p,\infty}^{\bm{k}} \cap \contfunc_{c}(K)$ so that $\indicator{F}\leq\varphi\leq\indicator{K}$ on $K$,
	hence $\mu(F)\leq\int_{K}\varphi\,d\mu=\Psi_{p}^{\bm{k}}(u; \varphi)\leq\mathcal{E}_{p}^{\bm{k}}(u)$
	by \eqref{KSpem.bdd} and thus $\mu(K)\leq\mathcal{E}_{p}^{\bm{k}}(u)<\infty$.
	In particular, $\contfunc_{0}(K)\ni\psi\mapsto\int_{K}\psi\,d\mu$ is a
	bounded linear functional on $C_{0}(K)$ which coincides with $\Psi_{p}^{\bm{k}}(u;\,\cdot\,)$
	on $B_{p,\infty}^{\bm{k}} \cap \contfunc_{c}(K)$ and thus
	with $\widetilde{\Psi}_{p}^{\bm{k}}(u;\,\cdot\,)$ on $C_{0}(K)$, and
	therefore $\mu=\Gamma_{p}^{\bm{k}}\langle u \rangle$ by the uniqueness
	of a positive Radon measure on $K$ satisfying \eqref{KSpem.defn}.
	
	Lastly, we shall prove \eqref{KSpem.limext}. 
	Note that \eqref{KSpem.limext} is true for $\varphi \in \KS \cap \contfunc_{c}(K)$ by \eqref{emfunc.expression} in Theorem \ref{thm.KSfunctional}. 
	As in the first paragraph of this proof, let $\varphi \in \contfunc_{0}(K)$ and choose $\{ \varphi_{j} \}_{j \in \mathbb{N}} \subseteq B_{p,\infty}^{\bm{k}} \cap \contfunc_{c}(K)$ so that $\lim_{j \to \infty}\norm{\varphi - \varphi_{j}}_{\sup} = 0$.
	Let $\varepsilon > 0$ and choose $N \in \mathbb{N}$ so that $\norm{\varphi - \varphi_{j}}_{\sup}\mathcal{E}_{p}^{\bm{k}}(u) < \varepsilon$ for any $j \ge N$.  
	Then, for any $n \in \mathbb{N}$ and any $j \ge N$,  
	\begin{align*}
		&\abs{\int_{K}\varphi\,d\Gamma_{p}^{\bm{k}}\langle u \rangle - \int_{K \times K}\abs{u(x) - u(y)}^{p}\varphi(x)\,m_{n}(dxdy)} \\
		&\le \abs{\widetilde{\Psi}_{p}^{\bm{k}}(u; \varphi) - \widetilde{\Psi}_{p}^{\bm{k}}(u; \varphi_{j})} + \abs{\widetilde{\Psi}_{p}^{\bm{k}}(u; \varphi_{j}) - \Psi_{p,n}^{\bm{k}}(u; \varphi_{j})} \\
		&\quad + \norm{\varphi - \varphi_{j}}_{\sup}\int_{K \times K}\abs{u(x) - u(y)}^{p}\,m_{n}(dxdy) \\
		&\le 2\varepsilon + \abs{\Psi_{p}^{\bm{k}}(u; \varphi_{j}) - \Psi_{p,n}^{\bm{k}}(u; \varphi_{j})}, 
	\end{align*}
	where $\Psi_{p,n}^{\bm{k}}(u;\,\cdot\,)$ is the same as in \eqref{emfunctional}. 
	Hence we have 
	\[
	\limsup_{n \to \infty}\abs{\int_{K}\varphi\,d\Gamma_{p}^{\bm{k}}\langle u \rangle - \int_{K \times K}\abs{u(x) - u(y)}^{p}\varphi(x)\,m_{n}(dxdy)} 
	\le 2\varepsilon, 
	\]
	which proves \eqref{KSpem.limext}. 
\end{proof}

In the rest of this section, we always suppose in addition that $B_ {p,\infty}^{\bm{k}} \cap \contfunc_{c}(K)$ is dense in $(\contfunc_{c}(K),\norm{\,\cdot\,}_{\sup})$. 

Note that both the boundedness and the continuity of $u$ are essential in Theorem \ref{thm.KSpem-exist}; the former is required for the right-hand side of \eqref{KSpem.characterize} to make sense, and the latter has been used heavily in the proof of Theorem \ref{thm.KSfunctional} above.
Next we would like to extend $\Gamma_{p}^{\bm{k}}\langle u \rangle$ to a wider range of $u$. 
Let us use the following notation for simplicity. 
\begin{defn}
	We define closed linear subspaces $\bclosureKS$ and $\cclosureKS$ of $B_{p,\infty}^{\bm{k}}$ by
	\begin{equation}
	\bclosureKS \coloneqq \closure{B_{p,\infty}^{\bm{k}} \cap \contfunc_{b}(K)}^{B_{p,\infty}^{\bm{k}}}
	\quad \text{and} \quad
	\cclosureKS \coloneqq \closure{B_{p,\infty}^{\bm{k}} \cap \contfunc_{c}(K)}^{B_{p,\infty}^{\bm{k}}}. 
	\end{equation}
\end{defn}

By virtue of the expression \eqref{emfunc.expression}, we can show the generalized $p$-contraction property \ref{GC} for $(\int_{K}\varphi\,d\Gamma_{p}^{\bm{k}}\langle \,\cdot\, \rangle,B_{p,\infty}^{\bm{k}} \cap \contfunc_{b}(K))$ for any $\varphi \in \contfunc_{c}(K)$ with $\varphi \geq 0$, which further allows us to extend $\Gamma_{p}^{\bm{k}}\langle u \rangle$ canonically to $u \in \bclosureKS$. 
\begin{thm}\label{thm.KSpem-GC}
	For any $u \in \bclosureKS$, there exists a unique positive Radon measure $\Gamma_{p}^{\bm{k}}\langle u \rangle$ on $K$ such that for any $\{ u_{n} \}_{n \in \mathbb{N}} \subseteq B_{p,\infty}^{\bm{k}} \cap \contfunc_{b}(K)$ with $\lim_{n \to \infty}\mathcal{E}_{p}^{\bm{k}}(u - u_{n}) = 0$ and any Borel measurable function $\varphi \colon K \to [0,\infty)$ with $\norm{\varphi}_{\sup} < \infty$, 
	\begin{equation}\label{KSpem.extension}
		\int_{K}\varphi\,d\Gamma_{p}^{\bm{k}}\langle u \rangle
		= \lim_{n \to \infty}\int_{K}\varphi\,d\Gamma_{p}^{\bm{k}}\langle u_{n} \rangle, 
	\end{equation}
	and $\KSem\langle u \rangle$ further satisfies $\KSem\langle u \rangle(K) \le \KSform(u)$. 
	Moreover, for each such $\varphi$, $(\int_{K}\varphi\,d\Gamma_{p}^{\bm{k}}\langle \,\cdot\, \rangle,\bclosureKS)$ is a $p$-energy form on $(K,m)$ satisfying \ref{GC}.
\end{thm}
\begin{proof}
	First, for any $\varphi \in \contfunc_{c}(K)$ with $\varphi \ge 0$, we will show that $(\int_{K}\varphi\,d\Gamma_{p}^{\bm{k}}\langle \,\cdot\, \rangle,B_{p,\infty}^{\bm{k}} \cap \contfunc_{b}(K))$ satisfies \ref{GC}.
	Throughout this proof, we fix $n_{1},n_{2} \in \mathbb{N}$, $q_{1} \in (0,p]$, $q_{2} \in [p,\infty]$ and $T = (T_{1},\dots,T_{n_{2}}) \colon \mathbb{R}^{n_{1}} \to \mathbb{R}^{n_{2}}$ satisfying \eqref{GC-cond}. 
	Let us consider the case $q_{2} < \infty$ since the proof for the case $q_{2} = \infty$ is similar.
	Let $\bm{u} = (u_1,\dots,u_{n_{1}}) \in \bigl(B_{p,\infty}^{\bm{k}} \cap \contfunc_{b}(K)\bigr)^{n_{1}}$. 
    Note that $T_{l}(\bm{u}) \in B_{p,\infty}^{\bm{k}} \cap \contfunc_{b}(K)$ for each $l \in \{ 1,\dots,n_{2} \}$.
    For any $n \in \mathbb{N}$, we see that
    \begin{align*}
        &\sum_{l = 1}^{n_{2}}\left(\int_{K \times K}\abs{T_{l}(\bm{u}(x)) - T_{l}(\bm{u}(y))}^{p}\varphi(x)\,m_{n}(dxdy)\right)^{q_{2}/p} \\
        &\overset{\eqref{reverse}}{\le} \left(\int_{K \times K}\Biggl[\sum_{l = 1}^{n_{2}}\abs{T_{l}(\bm{u}(x)) - T_{l}(\bm{u}(y))}^{q_{2}}\Biggr]^{p/q_{2}}\varphi(x)\,m_{n}(dxdy)\right)^{q_{2}/p} \\
        &\overset{\eqref{GC-cond}}{\le} \left(\int_{K \times K}\Biggl[\sum_{k = 1}^{n_{1}}\abs{u_{k}(x) - u_{k}(y)}^{q_{1}}\Biggr]^{p/q_{1}}\varphi(x)\,m_{n}(dxdy)\right)^{q_{2}/p} \\
        &\overset{\text{($\ast$)}}{\le} \left(\sum_{k = 1}^{n_{1}}\left(\int_{K \times K}\abs{u_{k}(x) - u_{k}(y)}^{p}\varphi(x)\,m_{n}(dxdy)\right)^{q_{1}/p}\right)^{q_{2}/q_{1}},
    \end{align*}
    where we used the triangle inequality for the norm of $L^{p/q_{1}}(K \times K, m_{n})$ in ($\ast$).
    By letting $n \to \infty$, we obtain from \eqref{KSpem.limext} that 
    \begin{equation}\label{KSpem.GCeach}
    	\norm{\left(\left(\int_{K}\varphi\,d\Gamma_{p}^{\bm{k}}\langle T_{l}(\bm{u}) \rangle\right)^{1/p}\right)_{l = 1}^{n_{2}}}_{\ell^{q_{2}}} 
    	\le \norm{\left(\left(\int_{K}\varphi\,d\Gamma_{p}^{\bm{k}}\langle u_{k} \rangle\right)^{1/p}\right)_{k = 1}^{n_{1}}}_{\ell^{q_{1}}}. 
    \end{equation}
    
    Next we will extend \eqref{KSpem.GCeach} to any Borel measurable function $\varphi \colon K \to [0,\infty]$. 
    Let us start with the case $\varphi = \indicator{A}$, where $A \in \mathcal{B}(K)$. 
    By \cite[Theorem 2.18]{Rud}, there exist sequences $\{ K_{n} \}_{n \in \mathbb{N}}$ and $\{ U_{n} \}_{n \in \mathbb{N}}$ such that $K_{n} \subseteq A \subseteq U_{n}$, $K_{n}$ is compact, $U_{n}$ is open and $\lim_{n \to \infty}\max_{v \in \{ T_{l}(\bm{u}) \}_{l} \cup \{ u_{k} \}_{k}}\Gamma_{p}^{\bm{k}}\langle v \rangle(U_{n} \setminus K_{n}) = 0$. 
    By Urysohn's lemma, we can pick $\varphi_{n} \in \contfunc_{c}(K)$ so that $0 \le \varphi_{n} \le 1$, $\varphi_{n}|_{K_{n}} = 1$ and $\supp_{K}[\varphi_{n}] \subseteq U_{n}$. 
    Applying \eqref{KSpem.GCeach} for $\varphi_{n}$, we obtain 
    \[
    \norm{\left(\Gamma_{p}^{\bm{k}}\langle T_{l}(\bm{u}) \rangle(K_{n})^{1/p}\right)_{l = 1}^{n_{2}}}_{\ell^{q_{2}}} 
    \le \norm{\left(\Gamma_{p}^{\bm{k}}\langle u_{k} \rangle(U_{n})^{1/p}\right)_{k = 1}^{n_{1}}}_{\ell^{q_{1}}}.
    \]
    By letting $n \to \infty$, we get \eqref{KSpem.GCeach} with $\varphi = \indicator{A}$. 
    Using the reverse Minkowski inequality on $\ell^{q_{1}/p}$ and the Minkowski inequality on $\ell^{q_{2}/p}$ (see also \cite[Proof of Proposition 2.9-(a)]{KS.gc}, where \ref{GC} is shown to be stable under addition), we see that \eqref{KSpem.GCeach} holds also for any non-negative Borel measurable simple function $\varphi$ on $K$. 
    We get the desired extension, \eqref{KSpem.GCeach} for any Borel measurable function $\varphi \colon K \to [0,\infty]$, by the monotone convergence theorem. 
   
    Now let us extend $p$-energy measures. 
    In the rest of this proof, let $\varphi \colon K \to [0,\infty)$ be a Borel measurable function such that $\norm{\varphi}_{\sup} < \infty$. 
    Let $u \in \bclosureKS$ and $\{ u_{n} \}_{n \in \mathbb{N}} \subseteq B_{p,\infty}^{\bm{k}} \cap \contfunc_{b}(K)$ satisfy $\lim_{n \to \infty}\mathcal{E}_{p}^{\bm{k}}(u - u_{n}) = 0$. 
    By Proposition \ref{prop.GClist}-\ref{GC.tri} for $(\int_{K}\varphi\,d\Gamma_{p}^{\bm{k}}\langle \,\cdot\, \rangle, B_{p,\infty}^{\bm{k}} \cap \contfunc_{b}(K))$, for any $n, n' \in \mathbb{N}$,
    \[
    \abs{\left(\int_{K}\varphi\,d\Gamma_{p}^{\bm{k}}\langle u_{n} \rangle\right)^{1/p} - \left(\int_{K}\varphi\,d\Gamma_{p}^{\bm{k}}\langle u_{n'} \rangle\right)^{1/p}}
    \le \norm{\varphi}_{\sup}^{1/p}\mathcal{E}_{p}^{\bm{k}}(u_{n} - u_{n'})^{1/p}, 
    \]
    which implies that the limit $\lim_{n \to \infty}\int_{K}\varphi\,d\Gamma_{p}^{\bm{k}}\langle u_{n} \rangle \eqqcolon I_{u}(\varphi)$ exists in $\mathbb{R}$ and it is independent of the choice of $\{ u_{n} \}_{n}$.
    In addition, by letting $n' \to \infty$ in the estimate above, we have that 
    \begin{equation}\label{KSpemint.unif}
    	\abs{\left(\int_{K}\varphi\,d\Gamma_{p}^{\bm{k}}\langle u_{n} \rangle\right)^{1/p} - I_{u}(\varphi)^{1/p}}
    	\le \norm{\varphi}_{\sup}^{1/p}\mathcal{E}_{p}^{\bm{k}}(u_{n} - u)^{1/p}.   
    \end{equation}
    Also, it is clear that $0 \le I_{u}(\varphi) \le \norm{\varphi}_{\sup}\mathcal{E}_{p}^{\bm{k}}(u)$ and that $I_{n}$ is linear in the sense that $I_{u}\bigl(\sum_{k = 1}^{N}a_{k}\varphi_{k}\bigr) = \sum_{k = 1}^{N}a_{k}I_{u}(\varphi_{k})$ for any $N \in \mathbb{N}$, $(a_{k})_{k = 1}^{N} \subseteq [0,\infty)$ and Borel measurable functions $\varphi_{k} \colon K \to [0,\infty)$ with $\norm{\varphi_{k}}_{\sup} < \infty$, $k \in \{ 1,\dots, N \}$.  
    Now we define $\Gamma_{p}^{\bm{k}}\langle u \rangle(A) \coloneqq I_{u}(\indicator{A}) \in [0,\infty)$ for $A \in \mathcal{B}(K)$, and show that $\Gamma_{p}^{\bm{k}}\langle u \rangle$ is a finite Borel measure on $K$.
    Clearly, $\Gamma_{p}^{\bm{k}}\langle u \rangle$ is finitely additive and $\Gamma_{p}^{\bm{k}}\langle u \rangle(K) \le \mathcal{E}_{p}^{\bm{k}}(u) < \infty$.  
    Hence it suffices to prove the countable additivity of $\Gamma_{p}^{\bm{k}}\langle u \rangle$. 
	By \eqref{KSpemint.unif}, for any $\varepsilon > 0$ there exists $N_{0} \in \mathbb{N}$ such that $\sup_{A \in \mathcal{B}(K)}\abs{\Gamma_{p}^{\bm{k}}\langle u \rangle(A)^{1/p} - \Gamma_{p}^{\bm{k}}\langle u_{n} \rangle(A)^{1/p}} < \varepsilon$ for any $n \ge N_{0}$. 
    Let $\{ A_{k} \}_{k \in \mathbb{N}} \subseteq \mathcal{B}(K)$ be a sequence of disjoint Borel sets, and set $B_{N} \coloneqq \bigcup_{k = N +1}^{\infty}A_{k}$ for each $N \in \mathbb{N}$.
    Then we see that for any $N \in \mathbb{N}$ and any $n \ge N_{0}$, 
    \[
    \abs{\Gamma_{p}^{\bm{k}}\langle u \rangle\left(\bigcup_{k \in \mathbb{N}}A_{k}\right) - \sum_{k = 1}^{N}\Gamma_{p}^{\bm{k}}\langle u \rangle(A_{k})}^{1/p}
    = \Gamma_{p}^{\bm{k}}\langle u \rangle(B_{N})^{1/p} 
        \le \varepsilon + \Gamma_{p}^{\bm{k}}\langle u_{n} \rangle(B_{N})^{1/p}, 
    \]
    whence $\lim_{N \to \infty} \abs{\Gamma_{p}^{\bm{k}}\langle u \rangle\left(\bigcup_{k \in \mathbb{N}}A_{k}\right) - \sum_{k = 1}^{N}\Gamma_{p}^{\bm{k}}\langle u \rangle(A_{k})} = 0$, proving the desired countable additivity. 
    
    Before showing \eqref{KSpem.extension}, i.e., $I_{u}(\varphi) = \int_{K}\varphi\,d\Gamma_{p}^{\bm{k}}\langle u \rangle$, we will extend \eqref{KSpem.GCeach} to the pair $(\int_{K}\varphi\,d\Gamma_{p}^{\bm{k}}\langle \,\cdot\, \rangle, \bclosureKS)$.
    To this end, we need to show that for any $\{ u_{n} \}_{n \in \mathbb{N}} \subseteq B_{p,\infty}^{\bm{k}} \cap \contfunc_{b}(K)$ converging weakly in $B_{p,\infty}^{\bm{k}}$ to $u \in \bclosureKS$ as $n \to \infty$, 
    \begin{equation}\label{fatou.pem}
    	\int_{K}\varphi\,d\Gamma_{p}^{\bm{k}}\langle u \rangle \le \liminf_{n \to \infty}\int_{K}\varphi\,d\Gamma_{p}^{\bm{k}}\langle u_{n} \rangle. 
    \end{equation}
    By extracting a subsequence of $\{ u_{n} \}_{n}$ if necessary, we can assume that the limit $\lim_{n \to \infty}\int_{K}\varphi\,d\Gamma_{p}^{\bm{k}}\langle u_{n} \rangle$ exists. 
    By Mazur's lemma (see, e.g., \cite[p. 19]{HKST}), there exist $N(n) \in \mathbb{N}$ and $\{ \alpha_{n,k} \}_{k = n}^{N(n)} \subseteq [0,1]$ with $N(n) > n$ and $\sum_{k = n}^{N(n)}\alpha_{n,k} = 1$ for each $n \in \mathbb{N}$ such that $v_{n} \coloneqq \sum_{k = n}^{N(n)}\alpha_{n,k}u_{k}$ converges to $u$ in $B_{p,\infty}^{\bm{k}}$ as $n \to \infty$. 
    We see from \ref{GC} and Proposition \ref{prop.GClist}-\ref{GC.tri} for $(\int_{K}\varphi\,d\Gamma_{p}^{\bm{k}}\langle \,\cdot\, \rangle, B_{p,\infty}^{\bm{k}} \cap \contfunc_{b}(K))$ that 
    \[
    \left(\int_{K}\varphi\,d\Gamma_{p}^{\bm{k}}\langle v_{n} \rangle\right)^{1/p} 
    \le \sum_{k = n}^{N(n)}\alpha_{n,k}\left(\int_{K}\varphi\,d\Gamma_{p}^{\bm{k}}\langle u_{k} \rangle\right)^{1/p},  
    \]
    which implies \eqref{fatou.pem} by letting $n \to \infty$.
    With this preparation, let us show that the pair $(\int_{K}\varphi\,d\Gamma_{p}^{\bm{k}}\langle \,\cdot\, \rangle, \bclosureKS)$ satisfies \ref{GC}. 
    Let $\bm{u} = (u_{1},\dots,u_{n_{1}}) \in \bigl(\bclosureKS\bigr)^{n_{1}}$.
    For each $k \in \{ 1,\dots,n_{1} \}$, fix $\{ u_{k,n} \}_{n \in \mathbb{N}} \subseteq B_{p,\infty}^{\bm{k}} \cap \contfunc_{b}(K)$ so that $\lim_{n \to \infty}\norm{u_{k} - u_{k,n}}_{B_{p,\infty}^{\bm{k}}} = 0$. 
    Set $\bm{u}_{n} \coloneqq (u_{1,n},\dots,u_{n_{1},n})$. 
    By \ref{GC} for $(\mathcal{E}_{p}^{\bm{k}},B_{p,\infty}^{\bm{k}})$ (see Theorem \ref{thm.KS-energy}-\ref{KS.GC}) and \eqref{GC-cond}, we know that $\{ T_{l}(\bm{u}_{n}) \}_{n}$ is bounded in $B_{p,\infty}^{\bm{k}}$ and that $\lim_{n \to \infty}\norm{T_{l}(\bm{u}_{n}) - T_{l}(\bm{u})}_{L^{p}} = 0$. 
    Since $B_{p,\infty}^{\bm{k}}$ is reflexive (see Theorem \ref{thm.banach}) and $B_{p,\infty}^{\bm{k}}$ is continuously embedded in $L^{p}(K,m)$, we see that $T_{l}(\bm{u}) \in \bclosureKS$ and that there exists a subsequence $\{ T_{l}(\bm{u}_{n_{j}}) \}_{j}$ such that $T_{l}(\bm{u}_{n_{j}})$ weakly converges to $T_{l}(\bm{u})$ in $B_{p,\infty}^{\bm{k}}$ as $j \to \infty$ for any $l \in \{ 1,\dots,n_{2} \}$. 
    By \eqref{fatou.pem}, we see that 
    \begin{align*}
    	\norm{\left(\left(\int_{K}\varphi\,d\Gamma_{p}^{\bm{k}}\langle T_{l}(\bm{u}) \rangle\right)^{1/p}\right)_{l = 1}^{n_{2}}}_{\ell^{q_{2}}} 
    	&\le \left(\sum_{l = 1}^{n_{2}}\liminf_{j \to \infty}\left(\int_{K}\varphi\,d\Gamma_{p}^{\bm{k}}\langle T_{l}(\bm{u}_{n_{j}}) \rangle\right)^{1/p}\right)^{1/q_{2}} \\
    	&\le \liminf_{j \to \infty}\left(\sum_{l = 1}^{n_{2}}\left(\int_{K}\varphi\,d\Gamma_{p}^{\bm{k}}\langle T_{l}(\bm{u}_{n_{j}}) \rangle\right)^{1/p}\right)^{1/q_{2}} \\
    	&\le  \liminf_{j \to \infty}\left(\sum_{k = 1}^{n_{1}}\left(\int_{K}\varphi\,d\Gamma_{p}^{\bm{k}}\langle u_{k,n_{j}} \rangle\right)^{1/p}\right)^{1/q_{1}} \\
    	&= \norm{\left(\left(\int_{K}\varphi\,d\Gamma_{p}^{\bm{k}}\langle u_{k} \rangle\right)^{1/p}\right)_{k = 1}^{n_{1}}}_{\ell^{q_{1}}}
    \end{align*}
    if $q_{2} < \infty$. 
    The case $q_{2} = \infty$ is similar, so $(\int_{K}\varphi\,d\Gamma_{p}^{\bm{k}}\langle \,\cdot\, \rangle, \bclosureKS)$ satisfies \ref{GC}. 
    
    Finally, we can prove \eqref{KSpem.extension}. 
    Let $\{ u_{n} \}_{n \in \mathbb{N}} \subseteq B_{p,\infty}^{\bm{k}} \cap \contfunc_{b}(K)$ be a sequence satisfying $\lim_{n \to \infty}\KSform(u - u_{n}) = 0$. 
    By Proposition \ref{prop.GClist}-\ref{GC.tri} for $(\int_{K}\varphi\,d\Gamma_{p}^{\bm{k}}\langle \,\cdot\, \rangle, \bclosureKS)$, we have 
    \[
    \abs{\left(\int_{K}\varphi\,d\Gamma_{p}^{\bm{k}}\langle u \rangle\right)^{1/p} - \left(\int_{K}\varphi\,d\Gamma_{p}^{\bm{k}}\langle u_{n} \rangle\right)^{1/p}}
    \le \norm{\varphi}_{\sup}^{1/p}\mathcal{E}_{p}^{\bm{k}}(u - u_{n})^{1/p}, 
    \]
    which together with \eqref{KSpemint.unif} implies that 
    \begin{align*}
    	&\abs{I_{u}(\varphi)^{1/p} - \left(\int_{K}\varphi\,d\Gamma_{p}^{\bm{k}}\langle u \rangle\right)^{1/p}} \\
    	&\le \abs{I_{u}(\varphi)^{1/p} - \left(\int_{K}\varphi\,d\Gamma_{p}^{\bm{k}}\langle u_{n} \rangle\right)^{1/p}} + \abs{\left(\int_{K}\varphi\,d\Gamma_{p}^{\bm{k}}\langle u_{n} \rangle\right)^{1/p} - \left(\int_{K}\varphi\,d\Gamma_{p}^{\bm{k}}\langle u \rangle\right)^{1/p}} \\
    	&\le 2\norm{\varphi}_{\sup}^{1/p}\mathcal{E}_{p}^{\bm{k}}(u - u_{n})^{1/p} 
    	\xrightarrow[n \to \infty]{} 0.  
    \end{align*}
    Hence we obtain \eqref{KSpem.extension}. 
\end{proof}

Thanks to Proposition \ref{prop.GClist}-\ref{GC.Cp} for $(\int_{K}\varphi\,d\Gamma_{p}^{\bm{k}}\langle \,\cdot\, \rangle, \bclosureKS)$, we can show the next result.
See \cite[Theorem 4.5 and Proposition 4.6]{KS.gc} for further details on $\Gamma_{p}^{\bm{k}}\langle u; v \rangle$ in the theorem below. 
\begin{thm}\label{thm.KSpem-diffble}
    Let $u,v \in \bclosureKS$. 
    Define $\Gamma_{p}^{\bm{k}}\langle u;v \rangle \colon \mathcal{B}(K) \to \mathbb{R}$ by 
    \begin{equation}\label{e:defn-KSpemtwo}
    	\Gamma_{p}^{\bm{k}}\langle u; v \rangle(A) \coloneqq \frac{1}{p}\left.\frac{d}{dt}\Gamma_{p}^{\bm{k}}\langle u + tv \rangle(A)\right|_{t = 0} \quad \text{for $A \in \mathcal{B}(K)$.}
    \end{equation}
    Then $\Gamma_{p}^{\bm{k}}\langle u; v \rangle$ is a signed Borel measure on $K$ and satisfies $\Gamma_{p}^{\bm{k}}\langle u; u \rangle = \Gamma_{p}^{\bm{k}}\langle u \rangle$.
    Moreover, for any $u,v \in \bclosureKS$ and any Borel measurable functions $\varphi, \psi \colon K \to [0,\infty]$,
    \begin{equation}\label{KSpem.diffble}
    	\int_{K}\varphi\,d\Gamma_{p}^{\bm{k}}\langle u; \,\cdot\, \rangle \colon \bclosureKS \to \mathbb{R} \text{ is the Fr\'{e}chet derivative of }
    	\frac{1}{p}\int_{K}\varphi\,d\Gamma_{p}^{\bm{k}}\langle \,\cdot\, \rangle \text{ at $u$}
	\end{equation}
	provided $\norm{\varphi}_{\sup} < \infty$, and
	\begin{equation}\label{KSpem.holder}
	\int_{K}\varphi\psi\,d\abs{\Gamma_{p}^{\bm{k}}\langle u; v \rangle} \le \left(\int_{K}\varphi^{\frac{p}{p - 1}}\,d\Gamma_{p}^{\bm{k}}\langle u \rangle\right)^{\frac{p - 1}{p}}\left(\int_{K}\psi^{p}\,d\Gamma_{p}^{\bm{k}}\langle v \rangle\right)^{\frac{1}{p}}.
    \end{equation}
\end{thm}
\begin{proof}
	It is proved in \cite[Theorem 4.5]{KS.gc} that $\Gamma_{p}^{\bm{k}}\langle u; v \rangle$ is a signed measure. 
	The statements \eqref{KSpem.diffble} and \eqref{KSpem.holder} follow from \cite[Propositions 4.6 and 4.8]{KS.gc}.
\end{proof}

As an important consequence of the strong locality of $(\KSform,\KS)$ obtained in Theorem \ref{thm.KS-energy}-\ref{KS.slbdd}, the inequality $\KSem\langle u \rangle(K) \leq \KSform(u)$ in Theorems \ref{thm.KSpem-exist} and \ref{thm.KSpem-GC} turns out to be an equality as long as $u \in \cclosureKS$.
Namely, we have the following proposition, which is the counterpart for $(\KSform,\KS)$ of the well-known equality \cite[Lemma 3.2.3]{FOT} for the strongly local part of a regular symmetric Dirichlet form.

\begin{prop}\label{prop.totalmass}
	If $u,v \in \cclosureKS$, then $\KSem\langle u; v \rangle(K) = \KSform(u; v)$. 
\end{prop}

\begin{proof}
	Since $(\Gamma_{p}^{\bm{k}}\langle \,\cdot\, \rangle(K),\bclosureKS)$ and $(\KSform,\KS)$ satisfy \ref{GC} by Theorems \ref{thm.KSpem-GC} and \ref{thm.KS-energy}-\ref{KS.GC}, thanks to the linearity of $\mathcal{E}(u; \,\cdot\,)$, \eqref{e:form-holder} and \eqref{eq:form-nonlinear-Hoelder-GC} from Proposition \ref{prop.form-basic} for $(\mathcal{E},\mathcal{F}) = (\Gamma_{p}^{\bm{k}}\langle \,\cdot\, \rangle(K),\bclosureKS),(\KSform,\KS)$ it suffices to consider the case $u,v \in \KS \cap \contfunc_{c}(K)$. 
	We first show that $\KSem\langle u \rangle(K) = \KSform(u)$ for any $u \in \KS \cap \contfunc_{c}(K)$. 
	Since $K$ is locally compact and we assume that $\KS \cap \contfunc_{c}(K)$ is dense in $(\contfunc_{c}(K),\norm{\,\cdot\,}_{\sup})$, by using Proposition \ref{prop.GClist}-\ref{GC.lip}, we can find an open neighborhood $U$ of the compact subset $\supp_{K}[u]$ of $K$ and $\varphi \in \KS \cap \contfunc_{c}(K)$ so that $0 \le \varphi \le 1$ and $\varphi(x) = 1$ for any $x \in U$. 
	Then $\supp_{K}[u] \cap \supp_{K}[\varphi - \indicator{K}] = \emptyset$.  
	By Theorem \ref{thm.KS-energy}-\ref{KS.slbdd}, we have $\KSform(u; u\varphi - u) = 0$ and $\KSform\bigl(\abs{u}^{\frac{p}{p - 1}}; \varphi\bigr) = 0$. 
	In particular, by \eqref{KSpem.characterize}, 
	\[
	\KSem\langle u \rangle(K) \ge \int_{K}\varphi\,\KSem\langle u \rangle = \KSform(u), 
	\]
	whence we have $\KSem\langle u \rangle(K) = \KSform(u)$. 
	
	Next let $u,v \in \KS \cap \contfunc_{c}(K)$.
	The argument in the previous paragraph implies that for any $t \in (0,1)$, 
	\[
	\frac{\KSem\langle u + tv \rangle(K) - \KSem\langle u \rangle(K)}{t}
	= \frac{\KSform(u + tv) - \KSform(u)}{t}. 
	\] 
	By letting $t \downarrow 0$ in this equality, we have $\KSem\langle u; v \rangle(K) = \KSform(u;v)$ by \eqref{e:defn-KSpemtwo} and \eqref{KS-compatible}. 
\end{proof}

We also have the following expression of $\int_{K}\varphi\,d\Gamma_{p}^{\bm{k}}\langle u; v \rangle$ if $\varphi \in \contfunc_{c}(K)$. 
In particular, we can deduce the analogues of Theorem \ref{thm.KS-energy}-\ref{it:difffunc},\ref{KS.conti} for $(\int_{K}\varphi\,d\Gamma_{p}^{\bm{k}}\langle \,\cdot\, \rangle, \bclosureKS)$.%
\begin{thm}\label{thm.KSpem-two}
	For any $u,v \in \bclosureKS$ and any $\varphi \in \contfunc_{c}(K)$, 
    \begin{align}
        &\int_{K}\varphi\,d\Gamma_{p}^{\bm{k}}\langle u; v \rangle \nonumber \\
        &= \lim_{n \to \infty}\int_{K}\int_{K}\gamma_{p}\bigl(u(x) - u(v)\bigr)(v(x) - v(y))\varphi(x)k_{r_{n}}(x,y)\,m(dy)m(dx) \label{KSpem.twoexp} \\
        &= \lim_{n \to \infty}\int_{K}\int_{K}\gamma_{p}\bigl(u(x) - u(v)\bigr)(v(x) - v(y))\varphi(y)k_{r_{n}}(x,y)\,m(dy)m(dx) \label{KSpem.twoexp.sym}.
    \end{align}
    In particular, the following hold: 
    \begin{enumerate}[label=\textup{(\alph*)},align=left,leftmargin=*,topsep=2pt,parsep=0pt,itemsep=2pt]
    	\item\label{it:KSpem.difffunc} Let $n_{1},n_{2} \in \mathbb{N}$, $q_{1} \in [1,p]$, $q_{2} \in [p,\infty]$, $\bm{u} = (u_{1},\dots,u_{n_{1}}) \in \bigl(\bclosureKS\bigr)^{n_{1}}$, $\bm{v} = (v_{1},\dots,v_{n_{2}}) \in L^{0 }(K,m)^{n_{2}}$, and let $\psi \colon K \to [0,\infty]$ be Borel measurable. If there exist $m$-versions of $\bm{u}$ and $\bm{v}$ such that $\norm{\bm{v}(x)}_{\ell^{q_{2}}} \le \norm{\bm{u}(x)}_{\ell^{q_{1}}}$ and $\norm{\bm{v}(x) - \bm{v}(y)}_{\ell^{q_{2}}} \le \norm{\bm{u}(x) - \bm{u}(y)}_{\ell^{q_{1}}}$ for any $(x,y) \in K \times K$, then $\bm{v} \in \bigl(\bclosureKS\bigr)^{n_{2}}$ and 
        \begin{equation}\label{KSpem.difffunc2}
        	\norm{\left(\left(\int_{K}\psi\,d\Gamma_{p}^{\bm{k}}\langle v_{l}\rangle\right)^{1/p}\right)_{l = 1}^{n_{2}}}_{\ell^{q_{2}}} 
        	\le \norm{\left(\left(\int_{K}\psi\,d\Gamma_{p}^{\bm{k}}\langle u_{k}\rangle\right)^{1/p}\right)_{k = 1}^{n_{1}}}_{\ell^{q_{1}}}.
        \end{equation}
    	\item\label{it:KSpem.conti} For any $u_{1},u_{2},v \in \bclosureKS$ and any Borel measurable function $\psi \colon K \to [0,\infty)$ with $\norm{\psi}_{\sup} < \infty$,
        \begin{align}\label{KSpem.conti}
    		&\mspace{-30mu}\abs{\int_{K}\psi\,d\Gamma_{p}^{\bm{k}}\langle u_{1}; v \rangle - \int_{K}\psi\,d\Gamma_{p}^{\bm{k}}\langle u_{2}; v \rangle} \nonumber \\
    		&\mspace{-30mu}\le C_{p}\biggl[\max_{i \in \{ 1,2 \}}\int_{K}\psi\,d\Gamma_{p}^{\bm{k}}\langle u_{i} \rangle\biggr]^{\frac{(p - 2)^{+}}{p}}\left(\int_{K}\psi\,d\Gamma_{p}^{\bm{k}}\langle u_{1} - u_{2} \rangle\right)^{\frac{(p - 1) \wedge 1}{p}}\left(\int_{K}\psi\,d\Gamma_{p}^{\bm{k}}\langle v \rangle\right)^{\frac{1}{p}}, 
    	\end{align}
    	where $C_{p}$ is the constant in Theorem \ref{thm.KS-energy}.
	\end{enumerate}
\end{thm}
\begin{proof}
	Throughout this proof, we fix $\varphi \in \contfunc_{c}(K)$. 
    We first show \eqref{KSpem.twoexp} in the case $u = v$. 
    Define 
    \[
    \mathcal{I}_{\varphi}^{n}\langle f \rangle \coloneqq \int_{K \times K}\abs{f(x) - f(y)}^{p}\varphi(x)\,m_{n}(dxdy) \quad \text{for $n \in \mathbb{N}$ and $f \in \bclosureKS$.}
    \]
    Fix $\{ u_{k} \}_{k \in \mathbb{N}} \subseteq B_{p,\infty}^{\bm{k}} \cap \contfunc_{b}(K)$ satisfying $\lim_{k \to \infty}\norm{u - u_{k}}_{B_{p,\infty}^{\bm{k}}} = 0$. 
    We easily have $\abs{\mathcal{I}_{\varphi}^{n}\langle u \rangle^{1/p} - \mathcal{I}_{\varphi}^{n}\langle u_{k} \rangle^{1/p}} \le \mathcal{I}_{\varphi}^{n}\langle u - u_{k} \rangle^{1/p} \le C^{1/p}\norm{\varphi}_{\sup}\mathcal{E}_{p}^{\bm{k}}(u - u_{k})^{1/p}$, where $C \in (0,\infty)$ is the constant in \eqref{KSene-comp}. 
    By \eqref{KSpem.extension} and Proposition \ref{prop.GClist}-\ref{GC.tri} for $(\int_{K}\varphi\,d\Gamma_{p}^{\bm{k}}\langle \,\cdot\, \rangle,\bclosureKS)$, we see that for any $n,k \in \mathbb{N}$,  
    \begin{align*}
    	&\abs{\left(\int_{K}\varphi\,d\Gamma_{p}^{\bm{k}}\langle u \rangle\right)^{1/p} - \mathcal{I}_{\varphi}^{n}\langle u \rangle^{1/p}} \\
    	&\le \abs{\left(\int_{K}\varphi\,d\Gamma_{p}^{\bm{k}}\langle u \rangle\right)^{1/p} - \left(\int_{K}\varphi\,d\Gamma_{p}^{\bm{k}}\langle u_{k} \rangle\right)^{1/p}} + \abs{\left(\int_{K}\varphi\,d\Gamma_{p}^{\bm{k}}\langle u_{k} \rangle\right)^{1/p} - \mathcal{I}_{\varphi}^{n}\langle u_{k} \rangle^{1/p}} \\
    	&\quad+ \abs{\mathcal{I}_{\varphi}^{n}\langle u \rangle^{1/p} - \mathcal{I}_{\varphi}^{n}\langle u_{k} \rangle^{1/p}} \\
    	&\le (1 + C^{1/p})\norm{\varphi}_{\sup}^{1/p}\mathcal{E}_{p}^{\bm{k}}(u - u_{k})^{1/p} + \abs{\left(\int_{K}\varphi\,d\Gamma_{p}^{\bm{k}}\langle u_{k} \rangle\right)^{1/p} - \mathcal{I}_{\varphi}^{n}\langle u_{k} \rangle^{1/p}}. 
    \end{align*} 
    Since $\lim_{n \to \infty}\abs{\bigl(\int_{K}\varphi\,d\Gamma_{p}^{\bm{k}}\langle u_{k} \rangle\bigr)^{1/p} - \mathcal{I}_{\varphi}^{n}\langle u_{k} \rangle^{1/p}} = 0$ by \eqref{emfunc.expression} and $k \in \mathbb{N}$ is arbitrary, we conclude that $\lim_{n \to \infty}\mathcal{I}_{\varphi}^{n}\langle u \rangle = \int_{K}\varphi\,d\Gamma_{p}^{\bm{k}}\langle u \rangle$. 
    
    Next we consider the general case $u \neq v$. 
    By Proposition \ref{prop.c-diff} and the convexity of $t \mapsto \mathcal{I}_{\varphi}^{n}\langle u + tv \rangle$, for any $t \in (0,1)$ and any $n \in \mathbb{N}$, 
    \begin{equation}\label{KSpem.two.c-diff}
        \frac{\mathcal{I}_{\varphi}^{n}\langle u + tv \rangle - \mathcal{I}_{\varphi}^{n}\langle u \rangle}{t} - \norm{\varphi}_{\sup}O_{t}(u; v) 
		\le \frac{d}{ds}\mathcal{I}_{\varphi}^{n}\langle u + sv \rangle\bigg\vert_{s = 0} 
        \le \frac{\mathcal{I}_{\varphi}^{n}\langle u + tv \rangle - \mathcal{I}_{\varphi}^{n}\langle u \rangle}{t}, 
    \end{equation}
    where $O_{t}(u; v) = C_{p,u,v}t^{(p - 1) \wedge \frac{1}{p - 1}}$ for some constant $C_{p,u,v} \in (0,\infty)$ which depends only on $p$, $\mathcal{E}_{p}^{\bm{k}}(u)$ and $\mathcal{E}_{p}^{\bm{k}}(v)$.
    Now we obtain \eqref{KSpem.twoexp} by noting that 
    \[
    \frac{d}{ds}\mathcal{I}_{\varphi}^{n}\langle u + sv \rangle\bigg\vert_{s = 0} = \int_{K \times K}\gamma_{p}\bigl(u(x) - u(v)\bigr)(v(x) - v(y))\varphi(x)\,m_{n}(dxdy)
    \]
    and letting first $n\to\infty$ in \eqref{KSpem.two.c-diff} and then $t\downarrow 0$ by using \eqref{KSpem.diffble}. 
    The equality \eqref{KSpem.twoexp.sym} can be shown similarly by considering 
    \[
    \widehat{\mathcal{I}}_{\varphi}^{n}\langle f \rangle \coloneqq \int_{K \times K}\abs{f(x) - f(y)}^{p}\varphi(y)\,m_{n}(dxdy) 
    \]
    instead of $\mathcal{I}_{\varphi}^{n}\langle f \rangle$ in the above arguments.  
    
    Lastly, let us show \ref{it:KSpem.difffunc} and \ref{it:KSpem.conti}. 
    
    \ref{it:KSpem.difffunc}: 
    By Theorem \ref{thm.KSpem-GC}, $(\KSform,\bclosureKS)$ is a $p$-energy form on $(K,m)$ satisfying \ref{GC}. 
    For each $l \in \{ 1,\dots,n_{2} \}$, by the argument in \cite[Proof of Corollary 2.4-(c)]{KS.gc}, we can find a $1$-Lipschitz map $T_{l} \colon (\mathbb{R}^{n_{1}},\norm{\,\cdot\,}_{\ell^{q_{1}}}) \to \mathbb{R}$ satisfying $T_{l}(0) = 0$ and $T_{l}(\bm{u}(x)) = v_{l}(x)$ for any $x \in K$. 
    By applying \ref{GC}, we have $v_{l} \in \bclosureKS$ and hence $\bm{v} \in \bigl(\bclosureKS\bigr)^{n_{2}}$. 
    Then the inequality \eqref{KSpem.difffunc2} in the case $\psi \in \contfunc_{c}(K)$ is immediate from \eqref{KSpem.twoexp}, and we can further extend \eqref{KSpem.difffunc2} to general $\psi$ in exactly the same way as the second paragraph of the proof of Theorem \ref{thm.KSpem-GC}.
    
    \ref{it:KSpem.conti}: 
    The estimate \eqref{KSpem.conti} in the case $\psi \in \contfunc_{c}(K)$ is immediate from \eqref{KSpem.twoexp}. 
    We can easily extend it to the desired case since $\contfunc_{c}(K)$ is dense in $L^{1}(K,\mu)$ for the finite Borel measure $\mu$ on $K$ given by
    \[
    \mu \coloneqq \abs{\KSem\langle u_{1}; v \rangle} + \abs{\KSem\langle u_{2}; v \rangle} + \KSem\langle u_{1} \rangle + \KSem\langle u_{2} \rangle + \KSem\langle u_{1} - u_{2} \rangle + \KSem\langle v \rangle.
    \qedhere\]
\end{proof}

The next theorem states the chain rule for our $p$-energy measures.
\begin{thm}[Chain rule]\label{thm.KSpem-chain}
    Let $n \in \mathbb{N}$, $u \in B_{p,\infty}^{\bm{k}} \cap \contfunc_{b}(K)$, $\bm{v} = (v_{1},\dots,v_{n}) \in \bigl(B_{p,\infty}^{\bm{k}} \cap \contfunc_{b}(K)\bigr)^{n}$, $\Phi \in C^{1}(\mathbb{R})$, $\Psi \in C^{1}(\mathbb{R}^{n})$ and suppose that $\Phi(0) = \Psi(0) = 0$.
    Then $\Phi(u),\Psi(\bm{v}) \in \KS \cap \contfunc_{b}(K)$ and 
    \begin{equation}\label{KSpem-chain}
        d\Gamma_{p}^{\bm{k}}\langle \Phi(u); \Psi(\bm{v}) \rangle
        = \sum_{k = 1}^{n}\gamma_{p}\bigl(\Phi'(u)\bigr)\partial_{k}\Psi(\bm{v})\,d\Gamma_{p}^{\bm{k}}\langle u; v_{k} \rangle.
    \end{equation}
\end{thm}
\begin{proof}
	It is immediate from Theorem \ref{thm.KS-energy}-\ref{KS.GC} (see also Proposition \ref{prop.GClist}-\ref{GC.lip}) that $\Phi(u),\Psi(\bm{v}) \in \KS \cap \contfunc_{b}(K)$.
	Note that 
	\[
	d\mu \coloneqq d\abs{\Gamma_{p}^{\bm{k}}\langle \Phi(u); \Psi(\bm{v}) \rangle} + \sum_{k = 1}^{n}\abs{\gamma_{p}\bigl(\Phi'(u)\bigr)\partial_{k}\Psi(\bm{v})}\,d\abs{\Gamma_{p}^{\bm{k}}\langle u; v_{k} \rangle}  	
	\]
	defines a finite Borel measure $\mu$ on $K$ by \eqref{KSpem.holder}. 
	Since $\contfunc_{c}(K)$ is dense in $L^{1}(K,\mu)$, it suffices to prove that for any $\varphi \in \contfunc_{c}(K)$,
	\[
	\int_{K}\varphi\,d\Gamma_{p}^{\bm{k}}\langle \Phi(u); \Psi(v) \rangle
    = \sum_{k = 1}^{n}\int_{K}\varphi\gamma_{p}\bigl(\Phi'(u)\bigr)\partial_{k}\Psi(v)\,d\Gamma_{p}^{\bm{k}}\langle u; v_{k} \rangle. 
	\]
	Let $\varphi \in \contfunc_{c}(K)$ and define $F_{n} \in \mathcal{B}(K \times K)$, $n \in \mathbb{N}$, by
    \[
    F_{n} \coloneqq \{ (x,y) \in K \times K \mid d(x,y) < \delta(r_n), \varphi(x) \neq 0 \}. 
    \]
    Note that $\closure{F_{n}}^{K \times K}$ is a compact subset of $K \times K$ for sufficiently large $n \in \mathbb{N}$ since $\varphi \in \contfunc_{c}(K)$, $\lim_{n \to 0}\delta(r_{n}) = 0$ and $(K,d)$ is locally compact. 
	Set 
    \[
    a_{n} \coloneqq \int_{F_{n}}\gamma_{p}\bigl(\Phi(u(x)) - \Phi(u(v))\bigr)(\Psi(v(x)) - \Psi(v(y)))\varphi(x)\,m_{n}(dxdy)
    \]
    and
    \[
    b_{n} \coloneqq \sum_{k = 1}^{n}\int_{F_{n}}\gamma_{p}\bigl(\Phi'(u(x))\bigr)\partial_{k}\Psi(v(x))\cdot\gamma_{p}\bigl(u(x) - u(y)\bigr)(v(x) - v(y))\varphi(x)\,m_{n}(dxdy).
    \]
    By Theorem \ref{thm.KSpem-two} and \eqref{e:asylocal}, it suffices to show $\lim_{n \to \infty}\abs{a_{n} - b_{n}} = 0$.
    To estimate $\abs{a_{n} - b_{n}}$, we introduce
    \[
    c_{n} \coloneqq \int_{F_{n}}\gamma_{p}\bigl(\Phi'(u(x))\bigr)\cdot\gamma_{p}\bigl(u(x) - u(y)\bigr)(v(x) - v(y))\varphi(x)\,m_{n}(dxdy). 
    \]
    We will show that $\lim_{n \to \infty}\abs{a_{n} - c_{n}} = \lim_{n \to \infty}\abs{b_{n} - c_{n}} = 0$.
    Note that
    \[
    \Phi(u(y)) - \Phi(u(x))
    = \bigl[u(y) - u(x)\bigr]\bigl(\Phi'(u(x)) + e_{\Phi,u}(x,y)\bigr),
    \]
    where we set $e_{\Phi,u}(x,y) \coloneqq \int_{0}^{1}\bigl[\Phi'\bigl(u(x) + t(u(y) - u(x))\bigr) - \Phi'(u(x))\bigr]\,dt$.
    Let $\varepsilon > 0$.
    Since $\Phi'$ is continuous, $\norm{u}_{\sup} < \infty$ and $u$ is uniformly continuous on $F_{n}$ for large enough $n \in \mathbb{N}$, we can find $N_{1} \in \mathbb{N}$ so that $\abs{e_{\Phi,u}(x,y)} < \varepsilon$ for any $(x,y) \in \bigcup_{n \ge N_{1}}F_{n}$. 
    By Lemma \ref{lem.p-1}, there exists $C_{p} \in (0,\infty)$ depending only on $p$ such that for any $n \ge N_{1}$ and $(x,y) \in \bigcup_{n \ge N_{1}}F_{n}$,
    \begin{align*}
        &\abs{\gamma_{p}\bigl(\Phi(u(x)) - \Phi(u(y))\bigr) - \gamma_{p}\bigl(\Phi'(u(x))\bigr)\cdot\gamma_{p}\bigl(u(x) - u(y)\bigr)} \\
        &\le C_{p}\varepsilon^{(p - 1) \wedge 1}A_{u,\Phi}(x,y)^{(p - 2)^{+}}\abs{u(x) - u(y)}^{(p - 1) \wedge 1},
    \end{align*}
    where $A_{u,\Phi}(x,y) \coloneqq \abs{\Phi(u(y)) - \Phi(u(x))} \vee \abs{\Phi'(u(x))(u(y) - u(x))}$. 
    By H\"{o}lder's inequality, we have 
    \begin{align*}
        &\sup_{n \ge N_{1}}\abs{a_{n} - c_{n}} \\
        &\le C_{p}\varepsilon^{(p - 1) \wedge 1}\Bigl[C_{\Phi,u}\bigl(\norm{\Phi(u)}_{B_{p,\infty}^{\bm{k}}} + \norm{u}_{B_{p,\infty}^{\bm{k}}}\bigr)\Bigr]^{(p - 2)^{+}}\norm{u}_{B_{p,\infty}^{\bm{k}}}^{(p - 1) \wedge 1}\norm{v}_{B_{p,\infty}^{\bm{k}}}, 
    \end{align*}
    where $C_{\Phi,u} \coloneqq 1 + \norm{\Phi'}_{\sup,[-\norm{u}_{\sup},\norm{u}_{\sup}]}$. In particular, we get $\lim_{n \to \infty}\abs{a_{n} - c_{n}} = 0$.
    
    Similarly, we can find $N_{2} \in \mathbb{N}$ so that for any $(x,y) \in \bigcup_{n \ge N_{2}}F_{n}$, 
    \[
    \abs{\bigl(\Psi(v(x)) - \Psi(v(y))\bigr) - \sum_{k = 1}^{n}\partial_{k}\Psi(v(x))(v(x) - v(y))}
    \le \varepsilon\abs{v(x) - v(y)}.
    \]
    Then we easily see that
    \[
    \sup_{n \ge N_{2}}\abs{b_{n} - c_{n}} \\
    \le \varepsilon\norm{\Phi'}_{\sup,[-\norm{u}_{\sup},\norm{u}_{\sup}]}^{p - 1}\norm{u}_{B_{p,\infty}^{\bm{k}}}^{p - 1}\norm{v}_{B_{p,\infty}^{\bm{k}}},
    \]
    whence $\lim_{n \to \infty}\abs{b_{n} - c_{n}} = 0$.
\end{proof}

The following \emph{image density property} of $p$-energy measures is a consequence of the chain rule. We note that the proof below does not rely on specific representations of $\KSem$ like \eqref{KSpem.limext} and \eqref{KSpem.twoexp}.
\begin{thm}[Image density property]\label{thm.EIDP}
	For any $u \in B_{p,\infty}^{\bm{k}} \cap \contfunc_{b}(K)$, the Borel measure $\Gamma_{p}^{\bm{k}}\langle u \rangle \circ u^{-1}$ on $\mathbb{R}$ defined by $\Gamma_{p}^{\bm{k}}\langle u \rangle \circ u^{-1}(A) \coloneqq \Gamma_{p}^{\bm{k}}\langle u \rangle(u^{-1}(A))$, $A \in \mathcal{B}(\mathbb{R})$, is absolutely continuous with respect to the $1$-dimensional Lebesgue measure on $\mathbb{R}$.
\end{thm}
\begin{proof}
	This is proved, on the basis of Theorem \ref{thm.KSpem-chain}, in exactly the same way as \cite[Proposition 7.6]{Shi24}, which is a simple adaptation of \cite[Theorem 4.3.8]{CF}, but we present the details because in \cite{Shi24} the underlying topological space $K$ is assumed to be a generalized Sierpi\'{n}ski carpet, a self-similar compact set in the Euclidean space. 
    It suffices to prove that $\KSem\langle u \rangle \circ u^{-1}(F) = 0$ for any $u \in \KS \cap \contfunc_{b}(K)$ and any compact subset $F$ of $\mathbb{R}$ such that $\mathscr{L}^{1}(F) = 0$, where $\mathscr{L}^{1}$ denotes the $1$-dimensional Lebesgue measure on $\mathbb{R}$.
    Let $\{ \varphi_{n} \}_{n \in \mathbb{N}} \subseteq \contfunc_{c}(\mathbb{R})$ satisfy $\abs{\varphi_{n}} \le 1$, $\lim_{n \to \infty}\varphi_{n}(x) = \indicator{F}(x)$ for any $x \in \mathbb{R}$ and
    \[
    \int_{0}^{\infty}\varphi_{n}(t)\,dt = \int_{-\infty}^{0}\varphi_{n}(t)\,dt = 0 \quad \text{for any $n \in \mathbb{N}$.}
    \]
    We define $\Phi_{n}(x) \coloneqq \int_{0}^{x}\varphi_{n}(t)\,dt$, $x \in \mathbb{R}$, and $u_{n} \coloneqq \Phi_{n} \circ u$ for any $n \in \mathbb{N}$.
    Then we easily see that $\Phi_{n} \in \contfunc^{1}(\mathbb{R}) \cap \contfunc_{c}(\mathbb{R})$, $\Phi_{n}(0) = 0$, and $\Phi_{n}' = \varphi_{n}$ for any $n \in \mathbb{N}$.
    Also, $u_{n}$ converges to $0$ in $L^{p}(K,m)$ as $n \to \infty$ by the dominated convergence theorem. 
    By Proposition \ref{prop.GClist}-\ref{GC.lip}, we deduce that $\{ u_{n} \}_{n \in \mathbb{N}}$ is bounded in $\KS$.  
    Since $\KS$ is reflexive by Theorem \ref{thm.banach} and $\KS$ is continuously embedded in $L^{p}(K,m)$, there exists a subsequence $\{ u_{n_{k}} \}_{k \in \mathbb{N}}$ weakly converging to $0$ in $\KS$. 
    By Mazur's lemma, there exist $N(l) \in \mathbb{N}$ and $\{ a_{l,k} \}_{k = l}^{N(l)} \subseteq [0,1]$ with $N(l) > l$ and $\sum_{k = l}^{N(l)}a_{l,k} = 1$ for each $l \in \mathbb{N}$ such that $\sum_{k = l}^{N(l)}a_{l,k}u_{n_{k}}$ converges to $0$ in $\KS$ as $l \to \infty$. 
    Let us define $\Psi_{l} \in \contfunc^{1}(\mathbb{R})$ by $\Psi_{l} \coloneqq \sum_{k = l}^{N(l)}a_{l,k}\Phi_{n_k}$. 
    Then $\Psi_{l}(0) = 0$, $\Psi_{l}' \to \indicator{F}$ and, by Fatou's lemma, Theorem \ref{thm.KSpem-chain} and Proposition \ref{prop.totalmass}, 
    \begin{align*}
        \KSem\langle u \rangle \circ u^{-1}(F)
        &= \int_{\mathbb{R}}\lim_{l \to \infty}\abs{\Psi_{l}'(t)}^{p}\,\bigl(\KSem\langle u \rangle \circ u^{-1}\bigr)(dt) \\
        &\le \liminf_{l \to \infty}\int_{K}\abs{\Psi_{l}'(u(x))}^{p}\,\KSem\langle u \rangle(dx) \\
        &= \liminf_{l \to \infty}\KSem\langle \Psi_{l}(u) \rangle(K)
        = \liminf_{l \to \infty}\KSform\bigl(\Psi_{l}(u)\bigr) = 0, 
    \end{align*}
    which completes the proof. 
\end{proof}

Now we can obtain the strongest possible forms of the strong locality of $\Gamma_{p}^{\bm{k}}\langle \,\cdot\, ; \,\cdot\, \rangle$ as in the following theorem, which is an easy consequence of Theorem \ref{thm.EIDP}, the triangle inequality for $\Gamma_{p}^{\bm{k}}\langle \,\cdot\, \rangle^{1/p}$ and \eqref{e:defn-KSpemtwo}; see \cite[Theorem 4.17]{KS.gc} for a proof.

\begin{thm}[Strong locality of $p$-energy measures]\label{thm.KSpem-sl}
	Let $u,u_{1},u_{2},v \in \KS \cap \contfunc_{b}(K)$, $a,a_{1},a_{2},b \in \mathbb{R}$ and $A \in \mathcal{B}(K)$. 
    \begin{enumerate}[label=\textup{(\alph*)},align=left,leftmargin=*,topsep=2pt,parsep=0pt,itemsep=2pt]
        \item\label{it:KSpem.slconst} If $A \subseteq u^{-1}(a)$, then $\KSem\langle u \rangle(A) = 0$.
        \item\label{it:KSpem.slbasic} If $A \subseteq (u - v)^{-1}(a)$, then $\KSem\langle u \rangle(A) = \KSem\langle v \rangle(A)$.
        \item\label{it:KSpem.sl1} If $A \subseteq u_{1}^{-1}(a_{1}) \cup u_{2}^{-1}(a_{2})$, then
			\begin{align}\label{e:KSpem-sl1}
			\Gamma_{p}^{\bm{k}}\langle u_1 + u_2 + v \rangle(A) + \Gamma_{p}^{\bm{k}}\langle v \rangle(A) &= \Gamma_{p}^{\bm{k}}\langle u_1 + v \rangle(A) + \Gamma_{p}^{\bm{k}}\langle u_2 + v \rangle(A), \\
			\KSem\langle u_{1} + u_{2}; v \rangle(A) &= \KSem\langle u_{1}; v \rangle(A) + \KSem\langle u_{2}; v \rangle(A). 
			\label{e:KSpem-sl1-cor}
			\end{align}
        \item\label{it:KSpem.sl2} If $A \subseteq (u_1 - u_2)^{-1}(a) \cup v^{-1}(b)$, then
			\begin{equation}\label{e:KSpem-sl2}
			\Gamma_{p}^{\bm{k}}\langle u_1; v \rangle(A) = \Gamma_{p}^{\bm{k}}\langle u_2; v \rangle(A)
			\quad\text{and}\quad
			\Gamma_{p}^{\bm{k}}\langle v; u_1 \rangle(A) = \Gamma_{p}^{\bm{k}}\langle v; u_2 \rangle(A).
			\end{equation}
    \end{enumerate}
\end{thm}

Using Theorem \ref{thm.KSpem-sl}, we can extend Proposition \ref{prop.totalmass} as follows. 
\begin{cor}\label{cor.totalmass-ext}
	Let $u,v \in \bclosureKS$. 
	If $\{ u,v \} \cap \cclosureKS \neq \emptyset$, then $\KSem\langle u; v \rangle(K) = \KSform(u; v)$. 
\end{cor}

\begin{proof}
	Similar to the proof of Proposotion \ref{prop.totalmass}, it suffices to consider the case $u,v \in \KS \cap \contfunc_{b}(K)$ with $\{ u,v \} \cap \KS \cap \contfunc_{c}(K) \neq \emptyset$.
	Let $f,g \in \{ u,v \}$ satisfy $\{ f,g \} = \{ u,v \}$ and $f \in \KS \cap \contfunc_{c}(K)$. 
	Similar to the proof of Proposition \ref{prop.totalmass}, we can find an open neighborhood $U$ of the compact subset $\supp_{K}[f]$  of $K$ and $\varphi \in \KS \cap \contfunc_{c}(K)$ so that $0 \le \varphi \le 1$ and $\varphi(x) = 1$ for any $x \in U$. 
	Then $\supp_{K}[f] \cap \supp_{K}[g(\varphi - \indicator{K})] = \emptyset$, so we have 
	\[
	\KSform(u; v) = 
	\begin{cases}
		\KSform(f; g\varphi) \quad &\text{if $f = u$,} \\
		\KSform(g\varphi; f) \quad &\text{if $f = v$,}
	\end{cases}
	\]
	by Theorem \ref{thm.KS-energy}-\ref{KS.slbdd} and 
	\[
	\KSem\langle u; v\rangle(K) = 
	\begin{cases}
		\KSem\langle f; g\varphi \rangle(K) \quad &\text{if $f = u$,} \\
		\KSem\langle g\varphi; f\rangle(K) \quad &\text{if $f = v$,}
	\end{cases}
	\]
	by Theorem \ref{thm.KSpem-sl}-\ref{it:KSpem.sl1},\ref{it:KSpem.sl2}.
	Since $f, g\varphi \in \KS \cap \contfunc_{c}(K)$, we obtain $\KSem\langle u; v\rangle(K) = \KSform(u; v)$ by Proposition \ref{prop.totalmass}. 
\end{proof}

\section{$p$-Energy forms on $p$-conductively homogeneous spaces}\label{sec.Kig}
\setcounter{equation}{0}
In this section, we verify \ref{KSwm} for the family of kernels $\bm{k}=\bm{k}^{s_{p}}$ defined by \eqref{KSkernel} and \eqref{Lp-Besov} on \emph{$p$-conductively homogeneous} compact metric spaces equipped with Ahlfors regular measures.
We also show some estimates on localized versions of Korevaar--Schoen $p$-energy forms, and construct, on the basis of Korevaar--Schoen $p$-energy forms, self-similar $p$-energy forms on $p$-conductively homogeneous self-similar sets as well.
We refer to \cite[Sections 4.3--4.6]{Kig23} for many concrete examples covered by this framework. 

\subsection{$p$-Conductively homogeneous spaces}
Les us recall the notation and terminology in \cite{Kig20,Kig23} by following \cite[Section 8.1]{KS.gc}. 
We fix a locally finite (non-directed) infinite tree $(T,E_{T})$ in the usual sense (see \cite[Definition 2.1]{Kig23} for example), and fix a \emph{root} $\phi \in T$ of $T$.
(Here $T$ is the set of vertices and $E_{T}$ is the set of edges.)
For any $w \in T \setminus \{ \phi \}$, we use $\overline{\phi w}$ to denote the unique simple path in $T$ from $\phi$ to $w$. 
\begin{defn}[{\cite[Definition 2.2]{Kig23}}]\label{defn.tree-notation}
	\begin{enumerate}[label=\textup{(\arabic*)},align=left,leftmargin=*,topsep=2pt,parsep=0pt,itemsep=2pt]
		\item For $w \in T$, define $\pi \colon T \to T$ by
		\begin{equation*}\label{defn.pi}
			\pi(w) \coloneqq 
			\begin{cases}
				w_{n - 1} \quad &\text{if $w \neq \phi$ and $\overline{\phi w} = (w_{0}, \dots, w_{n})$,} \\
				\phi \quad &\text{if $w = \phi$.}
			\end{cases}
		\end{equation*}
		Set $S(w) \coloneqq \{ v \in T \mid \pi(v) = w \} \setminus \{ w \}$.
		Moreover, for $k \in \mathbb{N}$, we define $S^{k}(w)$ inductively as
		\[
		S^{k + 1}(w) = \bigcup_{v \in S(w)}S^{k}(v).
		\]
		For $A \subseteq T$, define $S^{k}(A) \coloneqq \bigcup_{w \in A}S^{k}(A)$.
		\item For $w \in T$ and $n \in \mathbb{N} \cup \{ 0 \}$, define $\abs{w} \coloneqq \min\{ n \ge 0 \mid \pi^{n}(w) = \phi \}$ and $T_{n} \coloneqq \{ w \in T \mid \abs{w} = n \}$.
		\item Define $\Sigma \coloneqq \{ (\omega_{n})_{n \ge 0} \mid \text{$\omega_{n} \in T_{n}$ and $\omega_{n} = \pi(\omega_{n + 1})$ for all $n \in \mathbb{N} \cup \{ 0 \}$} \}$. 
		For $\omega = (\omega_{n})_{n \ge 0} \in \Sigma$, we write $[\omega]_{n}$ for $\omega_{n} \in T_{n}$.
		For $w \in T$, define $\Sigma_{w} \coloneqq \{ (\omega_{n})_{n \ge 0} \in \Sigma \mid \text{$\omega_{\abs{w}} = w$} \}$. 
		For $A \subseteq T$, define $\Sigma_{A} \coloneqq \bigcup_{w \in A}\Sigma_{w}$.
	\end{enumerate}
\end{defn}

We introduce a partition parametrized by a rooted tree (see \cite[Definition 2.2.1]{Kig20} and \cite[Lemma 3.6]{Sas23}).

\begin{defn}[Partition parametrized by a tree]\label{defn.partition}
	Let $K$ be a compact metrizable topological space without isolated points. 
	A family of non-empty compact subsets $\{ K_{w} \}_{w \in T}$ of $K$ is called a \emph{partition of $K$ parametrized by the rooted tree $(T, E_{T}, \phi)$} if and only if it satisfies the following conditions:
	\begin{enumerate}[label=\textup{(P\arabic*)},align=left,leftmargin=*,topsep=2pt,parsep=0pt,itemsep=2pt]
		\item $K_{\phi} = K$ and for any $w \in T$, $\#K_{w} \ge 2$ and $K_{w} = \bigcup_{v \in S(w)}K_{v}$.
		\item For any $w \in \Sigma$, $\bigcap_{n \ge 0}K_{[\omega]_{n}}$ is a single point.
	\end{enumerate}
\end{defn}

In the rest of this section, we fix a compact metrizable topological space without isolated points $K$, a locally finite rooted tree $(T, E_{T}, \phi)$ satisfying $\#\{ v \in T \mid \{v,w\} \in E_{T} \} \ge 2$ for any $w \in T$, a partition $\{ K_{w} \}_{w \in T}$ parametrized by $(T,E_{T},\phi)$, a metric $d$ on $K$ with $\diam(K, d) = 1$, and a Borel probability measure $m$ on $K$.
In the following definition, we collect some basic pieces of the notation used in \cite{Kig20,Kig23}.
\begin{defn}
	For $n \in \mathbb{N} \cup \{ 0 \}$ and $A \subseteq T_{n}$, define
	\begin{equation*}\label{defn.h-edge}
		E_{n}^{\ast} \coloneqq \bigl\{ \{ v, w \} \bigm| v, w \in T_{n}, v \neq w, K_{v} \cap K_{w} \neq \emptyset \bigr\}, 
	\end{equation*}
	and $E_{n}^{\ast}(A) = \bigl\{ \{ v, w \} \in E_{n}^{\ast} \bigm| v, w \in A \bigr\}$. 
	Let $d_{n}$ be the graph distance of $(T_{n}, E_{n}^{\ast})$.
	For $M \in \mathbb{N} \cup \{ 0 \}$, $w \in T_{n}$ and $x \in K$, define
	\begin{equation*}\label{defn.h-nbd}
		\Gamma_{M}(w) \coloneqq \{ v \in T_{n} \mid d_{n}(v, w) \le M \} \quad \text{and} \quad U_{M}(x; n) \coloneqq \bigcup_{w \in T_{n}; x \in K_{w}}\bigcup_{v \in \Gamma_{M}(w)}K_{v}.
	\end{equation*}
\end{defn}

To state geometric assumptions in \cite{Kig23}, we need the following definition (see \cite[Definitions 2.2.1 and 3.1.15]{Kig20}.) 
\begin{defn}
	\begin{enumerate}[label=\textup{(\arabic*)},align=left,leftmargin=*,topsep=2pt,parsep=0pt,itemsep=2pt]
        \item The partition $\{ K_{w} \}_{w \in T}$ is said to be \emph{minimal} if and only if $K_{w} \setminus \bigcup_{v \in T_{\abs{w}} \setminus \{ w \}} \neq \emptyset$ for any $w \in T$.
        \item The partition $\{ K_{w} \}_{w \in T}$ is said to be \emph{uniformly finite} if and only if $\sup_{w \in T}\#\Gamma_{1}(w) < \infty$. 
    \end{enumerate}
\end{defn}

We also use the following notation for simplicity.
\begin{defn}\label{d:nbd}
	For $n \in \mathbb{N} \cup \{ 0 \}$ and $U \subseteq K$, define $T_{n}[U] \coloneqq \{ w \in T_{n} \mid K_{w} \cap U \neq \emptyset \}$.
\end{defn}

Now we describe basic geometric conditions in \cite{Kig23}. 
The conditions \ref{BF1}, \ref{BF2} and \eqref{BF-nooverlap} in \ref{BF3} below are important to follow the rest of this paper. 
\begin{assum}[{\cite[Assumption 2.15]{Kig23}}]\label{assum.BF}
	Let $(K, \mathcal{O})$ be a connected compact metrizable space, $\{ K_{w} \}_{w \in T}$ a partition parametrized by the rooted tree $(T, \phi)$, $d$ a metric on $K$ that is compatible with the topology $\mathcal{O}$ and $\diam(K, d) = 1$ and $\measure$ a Borel probability measure on $K$. 
	There exist $M_{\ast} \in \mathbb{N}$ and $r_{\ast} \in (0, 1)$ such that the following conditions (1)--(5) hold.
	\begin{enumerate}[label=\textup{(\arabic*)},align=left,leftmargin=*,topsep=2pt,parsep=0pt,itemsep=2pt]
		\item\label{BF1} $K_{w}$ is connected for any $w \in T$, $\{ K_{w} \}_{w \in T}$ is minimal and uniformly finite, and $\inf_{m \ge 0}\min_{w \in T_{m}}\#S(w) \ge 2$.
		\item\label{BF2} There exist $c_{i} > 0$, $i \in \{ 1,\dots,5 \}$, such that the following conditions (2A)--(2C) are true.
		\begin{itemize}
			\item [(2A)]\label{BF2A} For any $w \in T$,
			\begin{equation}\label{BF.bi-Lip}
				c_{1}r_{\ast}^{\abs{w}} \le \diam(K_{w}, \metric) \le c_{2}r_{\ast}^{\abs{w}}.
			\end{equation}
			\item [(2B)]\label{BF2B} For any $n \in \mathbb{N}$ and $x \in K$,
			\begin{equation}\label{BF.adapted}
				B_{d}(x, c_{3}r_{\ast}^{n}) \subseteq U_{M_{\ast}}(x; n) \subseteq B_{d}(x, c_{4}r_{\ast}^{n}).
			\end{equation}
			\item [(2C)]\label{BF2C} For any $n \in \mathbb{N}$ and $w \in T_{n}$, there exists $x \in K_{w}$ satisfying
			\begin{equation}\label{BF.thick}
				K_{w} \supseteq B_{d}(x, c_{5}r_{\ast}^{n}).
			\end{equation}
		\end{itemize}
		\item\label{BF3} There exist $m_{1} \in \mathbb{N}$, $\gamma_{1} \in (0, 1)$ and $\gamma \in (0, 1)$ such that
		\begin{equation}\label{BF.super-exp}
			\measure(K_{w}) \ge \gamma\measure(K_{\pi(w)}) \quad \text{for any $w \in T$,}
		\end{equation}
		and
		\begin{equation}\label{BF.sub-exp}
			\measure(K_{v}) \le \gamma_{1}\measure(K_{w}) \quad \text{for any $w \in T$ and $v \in S^{m_{1}}(w)$.}
		\end{equation}
		Furthermore, $\measure$ is volume doubling with respect to $\metric$ and  
		\begin{equation}\label{BF-nooverlap}
			\measure(K_{w}) = \sum_{v \in S(w)}\measure(K_{v}) \quad \text{for any $w \in T$.}
		\end{equation}
		\item\label{BF4} There exists $M_{0} \ge M_{\ast}$ such that for any $w \in T$, $k \ge 1$ and any $v \in S^{k}(w)$,
		\[
		\Gamma_{M_{\ast}}(v) \cap S^{k}(w) \subseteq \biggl\{ v' \in T_{\abs{v}} \biggm| 
		\begin{minipage}{190pt}
            there exist $l \le M_{0}$ and $(v_{0},\dots,v_{l}) \in S^{k}(w)^{l + 1}$
            such that $(v_{j - 1},v_{j}) \in E_{\abs{v}}^{\ast}$ for any $j \in \{ 1,\dots,l \}$
        \end{minipage}
		\biggr\}. 
		\]
		\item\label{BF5} For any $w \in T$, $\pi(\Gamma_{M_{\ast} + 1}(w)) \subseteq \Gamma_{M_{\ast}}(\pi(w))$.
	\end{enumerate}
\end{assum}

Note that if a Borel probability measure $m$ on $K$ satisfies \eqref{BF-nooverlap}, then we have 
\begin{equation}\label{BF.nomass-intersection}
	m(K_{v} \cap K_{w}) = 0 \quad \text{for any $v,w \in T$ with $v \neq w$ and $\abs{v} = \abs{w}$;} 
\end{equation}
see \cite[Proposition 8.7]{KS.gc} for a proof of this fact. 

Next we introduce conductance, neighbor disparity constants and the notion of $p$-conductive homogeneity in Definitions \ref{defn.nei-const}, \ref{defn.con-const} and \ref{defn.pCH}.
We also recall the notion of a covering system in Definition \ref{defn.covering}, which is used in the definition of neighbor disparity constants. 
See \cite[Sections 2.2, 2.3 and 3.3]{Kig23} for further details on these topics.
In the rest of this section, we fix $p \in (1,\infty)$ unless otherwise stated. We will state some definitions and statements below for any $p \in [1,\infty)$, but on each such occasion we will explicitly declare that we let $p \in [1,\infty)$. 
\begin{defn}[{\cite[Definitions 2.17 and 3.4]{Kig23}}]\label{defn.con-const}
	Let $p \in [1,\infty)$,  $n \in \mathbb{N} \cup \{ 0 \}$ and $A \subseteq T_{n}$.
    \begin{enumerate}[label=\textup{(\arabic*)},align=left,leftmargin=*,topsep=2pt,parsep=0pt,itemsep=2pt]
        \item Define $\mathcal{E}_{p,A}^{n} \colon \mathbb{R}^{A} \to [0,\infty)$ by
        \[
        \mathcal{E}_{p,A}^{n}(f) \coloneqq \sum_{\{u,v\} \in E_{n}^{\ast}(A)}\abs{f(u) - f(v)}^{p}, \quad f \in \mathbb{R}^{A}.
        \]
        We write $\mathcal{E}_{p}^{n}(f)$ for $\mathcal{E}_{p,T_{n}}^{n}(f)$.
        \item For $A_{0},A_{1} \subseteq A$, define $\mathrm{cap}_{p}^{n}(A_{0},A_{1};A)$ by
        \[
        \mathrm{cap}_{p}^{n}(A_{0},A_{1};A) \coloneqq \inf\bigl\{ \mathcal{E}_{p,A}^{n}(f) \bigm| f \in \mathbb{R}^{A}, \text{$f|_{A_{i}} = i$ for $i \in \{ 0,1 \}$} \bigr\}.
        \]
        \item (Conductance constant) For $A_{1},A_{2} \subseteq A$ and $k \in \mathbb{N} \cup \{ 0 \}$, define 
        \[
        \mathcal{E}_{p,k}(A_{1},A_{2},A) \coloneqq \mathrm{cap}_{p}^{n + k}\bigl(S^{k}(A_{1}), S^{k}(A_{2}); S^{k}(A)\bigr).
        \]
    	For $M \in \mathbb{N}$, define $\mathcal{E}_{M, p, k} \coloneqq \sup_{w \in T}\mathcal{E}_{p,k}(\{w\},T_{\abs{w}} \setminus \Gamma_{M}(w),T_{\abs{w}})$.
    \end{enumerate}
\end{defn}

\begin{defn}[{\cite[Definitions 2.26-(3) and 2.29]{Kig23}}]\label{defn.covering}
	Let $N_{T},N_{E} \in \mathbb{N}$. 
	\begin{enumerate}[label=\textup{(\arabic*)},align=left,leftmargin=*,topsep=2pt,parsep=0pt,itemsep=2pt]
	\item Let $n \in \mathbb{N} \cup \{ 0 \}$ and $A \subseteq T_{n}$. A collection $\{ G_{i} \}_{i = 1}^{k}$ with $G_{i} \subseteq T_{n}$ is called a \emph{covering of $(A,E_{n}^{\ast}(A))$ with covering numbers $(N_T,N_E)$} if and only if $A = \bigcup_{i = 1}^{k}G_{k}$, $\max_{x \in A}\#\{ i \mid x \in G_{i} \} \le N_T$ and for any $(u,v) \in E_{n}^{\ast}(A)$, there exists $l \le N_E$ and $\{ w(1),\dots,w(l + 1)\} \subseteq A$ such that $w(1) = u$, $w(l + 1) = v$ and $(w(i),w(i + 1)) \in \bigcup_{j = 1}^{k}E_{n}^{\ast}(G_{j})$ for any $i \in \{ 1,\dots,l \}$. 
	\item Let $\mathscr{J} \subseteq \bigcup_{n \in \mathbb{N} \cup \{ 0 \}}\{ A \mid A \subseteq T_{n} \}$. The collection $\mathscr{J}$ is called a \emph{covering system with covering number $(N_T,N_E)$} if and only if the following conditions are satisfied: 
		\begin{enumerate}[label=\textup{(\roman*)},align=left,leftmargin=*,topsep=2pt,parsep=0pt,itemsep=2pt]
       		\item $\sup_{A \in \mathscr{J}}\#A < \infty$. 
        	\item For any $w \in T$ and $k \in \mathbb{N}$, there exists a finite subset $\mathscr{N} \subseteq \mathscr{J} \cap T_{\abs{w} + k}$ such that $\mathscr{N}$ is a covering of $\bigl(S^{k}(w),E_{\abs{w} + k}^{\ast}(S^{k}(w))\bigr)$ with covering numbers $(N_T,N_E)$. 
        	\item For any $G \in \mathscr{J}$ and $k \in \mathbb{N} \cup \{ 0 \}$, if $G \subseteq T_{n}$, then there exists a finite subset $\mathscr{N} \subseteq \mathscr{J} \cap T_{n + k}$ such that $\mathscr{N}$ is a covering of $\bigl(S^{k}(G),E_{n + k}^{\ast}(S^{k}(G))\bigr)$ with covering numbers $(N_T,N_E)$.
    	\end{enumerate}
    	The collection $\mathscr{J}$ is simply said to be a \emph{covering system} if $\mathscr{J}$ is a covering system with covering numbers $(N_T,N_E)$ for some $(N_T, N_E) \in \mathbb{N}^{2}$. 
    \end{enumerate}
\end{defn}

\begin{defn}[{\cite[Definitions 2.26 and 2.29]{Kig23}}]\label{defn.nei-const}
    Let $p \in [1,\infty)$, $n \in \mathbb{N}$ and $A \subseteq T_{n}$.
    \begin{enumerate}[label=\textup{(\arabic*)},align=left,leftmargin=*,topsep=2pt,parsep=0pt,itemsep=2pt]
        \item For $k \in \mathbb{N} \cup \{ 0 \}$ and $f \colon T_{n + k} \to \mathbb{R}$, define $P_{n,k}f \colon T_{n} \to \mathbb{R}$ by
        	\[
        	(P_{n,k}f)(w) \coloneqq \frac{1}{\sum_{v \in S^{k}(w)}\measure(K_{v})}\sum_{v \in S^{k}(w)}f(v)\measure(K_{v}), \quad w \in T_{n}.
        	\]
        (Note that $P_{n,k}f$ depends on the measure $\measure$.)
        \item (Neighbor disparity constant) For $k \in \mathbb{N} \cup \{ 0 \}$, define
        \[
        \sigma_{p,k}(A) \coloneqq \sup_{f \colon S^{k}(A) \to \mathbb{R}}\frac{\mathcal{E}_{p,A}^{n}(P_{n,k}f)}{\mathcal{E}_{p,S^{k}(A)}^{n + k}(f)}. 
        \]
        \item Let $\mathscr{J} \subseteq \bigcup_{n \ge 0}\{ A \mid A \subseteq T_{n}\}$ be a covering system. Define
        \[
        \sigma_{p,k,n}^{\mathscr{J}} \coloneqq \max\{ \sigma_{p,k}(A) \mid A \in \mathscr{J}, A \subseteq T_{n} \} \quad \text{and} \quad \sigma_{p,k}^{\mathscr{J}} \coloneqq \sup_{n \in \mathbb{N} \cup \{ 0 \}}\sigma_{p,k,n}^{\mathscr{J}}.
        \]
    \end{enumerate}
\end{defn}

\begin{defn}[{\cite[Definition 3.4]{Kig23}}]\label{defn.pCH}
    Let $p \in [1,\infty)$.
    The compact metric space $K$ (with a partition $\{ K_{w} \}_{w \in T}$ and a measure $\measure$) is said to be \emph{$p$-conductively homogeneous} if and only if there exists a covering system $\mathscr{J}$ such that
    \begin{equation}\label{d:pCH}
        \sup_{k \in \mathbb{N} \cup \{ 0 \}}\sigma_{p,k}^{\mathscr{J}}\mathcal{E}_{M_{\ast},p,k} < \infty.
    \end{equation}
\end{defn}

\begin{thm}[A part of {\cite[Theorem 3.30]{Kig23}}]\label{t:pCH}
    Let $p \in [1,\infty)$ and suppose that Assumption \ref{assum.BF} holds.
    If $K$ is $p$-conductively homogeneous, then there exist $c_{1},c_{2},\sigma_{p} \in (0,\infty)$ and a covering system $\mathscr{J}$ such that for any $k \in \mathbb{N} \cup \{ 0 \}$, 
    \begin{equation}\label{pCH.1}
        c_{1}\sigma_{p}^{-k} \le \mathcal{E}_{M_{\ast},p,k} \le c_{2}\sigma_{p}^{-k} \quad \text{and} \quad c_{1}\sigma_{p}^{k} \le \sigma^{\mathscr{J}}_{p,k} \le c_{2}\sigma_{p}^{k}. 
    \end{equation}
\end{thm}

The following weak monotonicity is a key consequence of the $p$-conductive homogeneity.

\begin{lem}[Weak monotonicity]\label{lem.Kig-wm}
	Let $p \in [1,\infty)$ and suppose that Assumption \ref{assum.BF} holds.
    If $K$ is $p$-conductively homogeneous, then there exists $C \in (0,\infty)$ such that for any $k,l \in \mathbb{N}$, any $A \subseteq T_{k}$ and any $f \in L^{1}(K,m)$, 
    \begin{equation}\label{e:Kig-wm}
    	\sigma_{p}^{k}\mathcal{E}_{p,A}^{k}(P_{k}f) \le C\sigma_{p}^{k + l}\mathcal{E}_{p,S^{l}(A)}^{k + l}(P_{k + l}f), 
    \end{equation}
    where $\sigma_{p}$ is the constant in \eqref{pCH.1}. 
\end{lem}
\begin{proof}
This follows immediately by combining \cite[Lemma 2.27]{Kig23} and \eqref{pCH.1}.
\end{proof}

We also recall the ``Sobolev space'' $\mathcal{W}^{p}$ introduced in \cite[Lemma 3.13]{Kig23}.
\begin{defn}\label{d:Kig-sob}
	Let $p \in [1,\infty)$.
	Suppose that Assumption \ref{assum.BF} holds and that $K$ is $p$-conductively homogeneous.
    Let $\sigma_{p}$ be the constant in \eqref{pCH.1}.
    \begin{enumerate}[label=\textup{(\arabic*)},align=left,leftmargin=*,topsep=2pt,parsep=0pt,itemsep=2pt]
    	\item For $n \in \mathbb{N} \cup \{ 0 \}$, $w \in T_{n}$, $E \in \mathcal{B}(K)$ with $E \supseteq K_{w}$ and $f \in L^{1}(E, m|_{E})$, define $P_{n}f(w) \coloneqq \fint_{K_{w}}f\,dm$.  
        \item We define $\mathcal{N}_{p} \colon L^{p}(K,m) \to [0,\infty]$ and $\mathcal{W}^{p} \subseteq L^{p}(K,m)$ by
        \begin{align*}
            \mathcal{N}_{p}(f) &\coloneqq \left(\sup_{n \in \mathbb{N} \cup \{ 0 \}}\sigma_{p}^{n}\mathcal{E}_{p}^{n}(P_{n}f)\right)^{1/p}, \quad f \in L^{p}(K,m), \\
            \mathcal{W}^{p} &\coloneqq \bigl\{ f \in L^{p}(K,\measure) \bigm| \mathcal{N}_{p}(f) < \infty \bigr\}.
        \end{align*}
        Note that $\mathcal{N}_{p}(f) = 0$ if and only if $f$ is constant on $K$ (see \cite[Section 8.1]{KS.gc} for details).
        We also equip $\mathcal{W}^{p}$ with the norm $\norm{\,\cdot\,}_{\mathcal{W}^{p}}$ defined by 
        \[
        \norm{f}_{\mathcal{W}^{p}} \coloneqq \left(\norm{f}_{L^{p}(K,m)}^{p} + \mathcal{N}_{p}(f)^{p}\right)^{1/p}, \quad f \in \mathcal{W}^{p}. 
        \]
        \item For $n \in \mathbb{N} \cup \{ 0 \}$, $A \subseteq T_{n}$, $E \in \mathcal{B}(K)$ with $E \supseteq \bigcup_{w \in A}K_{w}$ and $f \in L^{1}(E,m|_{E})$, we define
        \begin{equation*}
            \widetilde{\mathcal{E}}_{p,A}^{n}(f) \coloneqq \sigma_{p}^{n}\mathcal{E}_{p,A}^{n}(P_{n}f). 
        \end{equation*}
        We also set $\widetilde{\mathcal{E}}_{p}^{n}(f) \coloneqq \widetilde{\mathcal{E}}_{p,T_{n}}^{n}(f)$ for $f \in L^{1}(K,m)$.
    \end{enumerate}
\end{defn}

Now we can introduce a framework to construct a $p$-resistance form on $K$. 
\begin{assum}\label{assum.CH}
	Let $(K,d,\{ K_{w} \}_{w \in T},m)$ satisfy Assumption \ref{assum.BF}. 
    In addition, $(K,d,\{ K_{w} \}_{w \in T},m,p)$ satisfies the following conditions:
    \begin{enumerate}[label=\textup{(\arabic*)},align=left,leftmargin=*,topsep=2pt,parsep=0pt,itemsep=2pt]
        \item \label{BA-AR} The measure $m$ is Ahlfors regular with respect to $d$. (Recall \eqref{AR}.)
        \item \label{BA-pch} $K$ is $p$-conductively homogeneous.
        \item \label{BA-ARC} $\sigma_{p} > 1$, where $\sigma_{p}$ is the constant in \eqref{pCH.1}. 
    \end{enumerate}
\end{assum}
\begin{rmk}\label{rmk.assm-pch}
	\begin{enumerate}[label=\textup{(\arabic*)},align=left,leftmargin=*,topsep=2pt,parsep=0pt,itemsep=2pt]
		\item\label{it:rmk-ARC} By \cite[Theorem 4.6.9]{Kig20}, Assumption \ref{assum.CH}-\ref{BA-ARC} is equivalent to $p > \dim_{\mathrm{ARC}}(K,d)$, where $\dim_{\mathrm{ARC}}(K,d)$ denotes the Ahlfors regular conformal dimension of $(K,d)$. (See, e.g., \cite[(1.1)]{Kig23} for the definition of $\dim_{\mathrm{ARC}}(K,d)$.)
		\item\label{it:rmk-pCH} It is highly non-trivial in general to verify that a given compact metric space $K$ is $p$-conductively homogeneous. In \cite[Sections 4.3--4.6]{Kig23} and \cite{KO+}, the $p$-conductive homogeneity for $p > \dim_{\mathrm{ARC}}(K,d)$ has been proved for various large classes of self-similar sets $K$ in $\mathbb{R}^{n}$ equipped with the Euclidean metric $d$.
	\end{enumerate}
\end{rmk}

In the following theorem, we recall a fundamental result on $\mathcal{W}^{p}$. 
\begin{thm}[{\cite[Lemmas 3.16, 3.19, 3.24 and Theorem 3.22]{Kig23}}, {\cite[Theorem 8.16]{KS.gc}}]\label{thm.Wp}
    Let $p \in [1,\infty)$.
    Suppose that $(K,d,\{ K_{w} \}_{w \in T},m)$ satisfies Assumption \ref{assum.BF} and that $K$ is $p$-conductively homogeneous.
    Then $\mathcal{W}^{p}$ equipped with the norm $\norm{\,\cdot\,}_{\mathcal{W}^{p}}$ is a Banach space.
    If $p > 1$, then $\mathcal{W}^{p}$ is reflexive and separable. 
    If $p > \dim_{\mathrm{ARC}}(K,d)$, or equivalently $\sigma_{p} > 1$, then $\mathcal{W}^{p} \subseteq \contfunc(K)$ and $\mathcal{W}^{p}$ is dense in $(\contfunc(K),\norm{\,\cdot\,}_{\sup})$.
\end{thm}

Let us introduce an important exponent, which we call the $p$-walk dimension, to describe the main result in this section. 
\begin{defn}\label{d:values}
    Suppose that Assumption \ref{assum.BF} holds, that $m$ is Ahlfors regular and that $K$ is $p$-conductively homogeneous.
    Let $r_{\ast} \in (0,1)$ be the constant in \eqref{BF.bi-Lip}, let $\sigma_{p}$ be the constant in \eqref{pCH.1} and let $\hdim$ be the Hausdorff dimension of $(K,d)$.
    Define 
    \begin{equation}
        \pwalk \coloneqq \hdim + \frac{\log{\sigma_{p}}}{\log{r_{\ast}^{-1}}}.
    \end{equation}
    We call $\pwalk$ the \emph{$p$-walk dimension} of $(K,d,\{ K_{w} \}_{w \in T},m)$.
\end{defn}

The next proposition states a suitable capacity upper bound in this framework. 

\begin{prop}[{\cite[Proposition 8.21]{KS.gc}}]\label{prop.capu-CH}
    Suppose that $(K,d,\{ K_{w} \}_{w \in T},m,p)$ satisfies Assumption \ref{assum.CH}.
    Then there exists $C \in (0,\infty)$ such that for any $(x,s) \in K \times (0,1]$,
    \begin{equation}\label{Kig-capu}
        \inf\Bigl\{ \mathcal{N}_{p}(f)^{p} \Bigm| f \in \mathcal{W}^{p}, f|_{B_{d}(x,s)} = 1, \supp_{K}[f] \subseteq B_{d}(x,2s) \Bigr\} \le Cs^{\hdim - \pwalk}.
    \end{equation}
\end{prop}
%

We also consider the following setting to deal with the case $p \le \dim_{\mathrm{ARC}}(K,d)$. 
\begin{assum}\label{assum.CH2}
    Let $(K,d,\{ K_{w} \}_{w \in T},m)$ satisfy Assumption \ref{assum.BF}. 
    In addition, $(K,d,\{ K_{w} \}_{w \in T},m,p)$ satisfies the following conditions:
    \begin{enumerate}[label=\textup{(\arabic*)},align=left,leftmargin=*,topsep=2pt,parsep=0pt,itemsep=2pt]
        \item \label{BA2-AR} The measure $m$ is Ahlfors regular with respect to $d$.
        \item \label{BA2-pch} $K$ is $p$-conductively homogeneous.
        \item \label{BA2-ARC} There exists $C \in (0,\infty)$ such that for any $(x,s) \in K \times (0,1]$,
   		\begin{align}\label{BF.capu}
        	&\inf\Bigl\{ \mathcal{N}_{p}(f)^{p} \Bigm| f \in \mathcal{W}^{p} \cap \contfunc(K), f|_{B_{d}(x,s)} = 1, \supp_{K}[f] \subseteq B_{d}(x,2s) \Bigr\} \nonumber \\
        	&\le Cs^{\hdim - \pwalk}.
    	\end{align}
    \end{enumerate}
\end{assum}
Note that Assumption \ref{assum.CH} implies Assumption \ref{assum.CH2} by Proposition \ref{prop.capu-CH}. 

The same argument as in \cite[Lemma 6.26]{MS+} yields a good partition of unity under Assumption \ref{assum.CH2} as given in Lemma \ref{lem.unity-Kig} and thus we obtain the regularity of $\mathcal{W}^{p}$ in Corollary \ref{cor.reg-Kig}.
\begin{lem}\label{lem.unity-Kig}
    Suppose that Assumption \ref{assum.CH2} holds.
    Let $\varepsilon \in (0, 1)$ and let $V$ be a maximal $\varepsilon$-net of $(K, d)$.
    Then there exists a family of functions $\{ \psi_z \}_{z \in V}$ that satisfies the following properties:
    \begin{enumerate}[label=\textup{(\roman*)},align=left,leftmargin=*,topsep=2pt,parsep=0pt,itemsep=2pt]
        \item $\sum_{z \in V}\psi_z \equiv 1$. 
        \item $\psi_z \in \mathcal{W}^{p} \cap \contfunc(K)$, $0 \le \psi_z \le 1$, $\psi_{z}|_{B_{\metric}(z, \varepsilon/4)} \equiv 1$ and $\supp_{K}[\psi_z] \subseteq B_{\metric}(z, 5\varepsilon/4)$ for any $z \in V$. 
        \item If $z \in V$ and $z' \in V \setminus \{ z \}$, then $\psi_{z'}|_{B_{\metric}(z, \varepsilon/4)} \equiv 0$.
        \item There exists $C \in (0,\infty)$ such that $\mathcal{N}_{p}(\psi_z)^{p} \le C\varepsilon^{\hdim - \pwalk}$ for any $z \in V$.
    \end{enumerate}
\end{lem}

\begin{cor}\label{cor.reg-Kig}
	Suppose that Assumption \ref{assum.CH2} holds.
	Then $\mathcal{W}^{p} \cap \contfunc(K)$ is dense in $(\contfunc(K),\norm{\,\cdot\,}_{\sup})$.
\end{cor}

\subsection{Localized energy estimates}
In this subsection, we show localized energy estimates on Korevaar--Schoen $p$-energy forms, which will imply \ref{KSwm} with the family of kernels $\bm{k}^{s_p}$ (recall \eqref{KSkernel}) and the equality $s_{p} = \pwalk/p$.
Estimates in this subsection are very similar to \cite[Section 7]{MS+} although the setting of ``partitions'' in \cite{MS+} is slightly different from ours.

We start with the following lemma giving a Poincar\'{e}-type estimate.

\begin{lem}[{\cite[Lemma 8.22]{KS.gc}}]\label{lem.prePI}
	Suppose that Assumption \ref{assum.CH2} holds.
	Then there exists a constant $C \in (0,\infty)$ such that for any $f \in L^{p}(K,m)$ and any $w \in T$, 
	\begin{equation}\label{e:prePI}
		\int_{K_{w}}\abs{f(x) - f_{K_{w}}}^{p}\,m(dx)
        \le Cr_{\ast}^{\abs{w}\pwalk}\liminf_{n \to \infty}\widetilde{\mathcal{E}}_{p,S^{n}(w)}^{n + \abs{w}}(f).
	\end{equation}
\end{lem}
%

The next proposition shows an upper bound on localized Korevaar--Schoen energy functionals.

\begin{prop}\label{prop.upper}
    Suppose that Assumption \ref{assum.CH2} holds.
    Then there exists $C \in (0,\infty)$ such that for any $E \in \mathcal{B}(K)$, any open neighborhood $E'$ of $\closure{E}^{K}$ and any $f \in L^{p}(E', \measure|_{E'})$,
	\begin{align}\label{KSu.local}
	    \limsup_{r \downarrow 0}&\int_{E}\fint_{B_{\metric}(x, r)}\frac{\abs{f(x) - f(y)}^{p}}{r^{\pwalk}}\,\measure(dy)\measure(dx) \nonumber \\
        &\mspace{150mu}\le C\limsup_{r \downarrow 0}\liminf_{n \to \infty}\widetilde{\mathcal{E}}_{p, T_{n}[(E)_{d,r}]}^{n}(f),
	\end{align} 
	Furthermore, with $C \in (0,\infty)$ the same as in \eqref{KSu.local}, for any $f \in L^{p}(K,m)$, 
	\begin{equation}\label{KSu.global}
		\sup_{r > 0}\int_{K}\fint_{B_{\metric}(x, r)}\frac{\abs{f(x) - f(y)}^{p}}{r^{\pwalk}}\,\measure(dy)\measure(dx)
        \le C\mathcal{N}_{p}(f)^{p}.
	\end{equation}
\end{prop}
\begin{proof}
    Let $r_{\ast} \in (0,1)$ and $M_{\ast} \in \mathbb{N}$ be the constants in Assumption \ref{assum.BF}.
    Let $r > 0$ and choose $n(r) \in \mathbb{N}$ satisfying $c_{3}r_{\ast}^{n(r) + 1} < r \le c_{3}r_{\ast}^{n(r)}$, where $c_{3}$ is the constant in \eqref{BF.adapted}.
    Then, for any $w \in T_{n(r)}$ and $x \in K_{w}$, we have $B_{d}(x,r) \subseteq U_{M_{\ast}}(x;n(r)) \subseteq U_{M_{\ast} + 1}(w)$.  
    Let $f \in L^{p}(E', m|_{E'})$, where $E'$ is an open neighborhood of $\closure{E}^{K}$. 
    Set $c \coloneqq (M_{\ast} + 2)c_{2}(c_{3}r_{\ast})^{-1} \in (0,\infty)$, where $c_{2}$ is the constant in \eqref{BF.bi-Lip}. 
    Then, by \eqref{BF.bi-Lip}, $\bigcup_{w \in T_{n(r)}[E]}S^{k}(\Gamma_{M_{\ast} + 1}(w)) \subseteq T_{k + n(r)}[(E)_{d,cr}]$ for any $k \in \mathbb{N}$ and there exists $r_{0} \in (0,\infty)$ such that $(E)_{d,cr} \subseteq E'$ for any $r \in (0,r_0)$.
    By using $\abs{f(x) - f(y)}^{p} \lesssim \abs{f(x) - f_{K_{w}}}^{p} + \abs{f(y) - f_{K_{v}}}^{p} + \abs{f_{K_{v}} - f_{K_{w}}}^{p}$ and Lemma \ref{lem.prePI}, we see that for any $r \in (0,r_0)$,
    \begin{align}\label{KSu.main}
        &\int_{E}\fint_{B_{\metric}(x, r)}\frac{\abs{f(x) - f(y)}^{p}}{r^{\pwalk}}\,\measure(dy)\measure(dx) \nonumber \\
        &\le r_{\ast}^{-n(r)\hdim}\sum_{w \in T_{n(r)}[E], v \in \Gamma_{M_{\ast} + 1}(w)}\int_{K_{w}}\int_{K_{v}}\frac{\abs{f(x) - f(y)}^{p}}{r^{\pwalk}}\,\measure(dy)\measure(dx) \nonumber \\
        &\lesssim \sum_{w \in T_{n(r)}[E], v \in \Gamma_{M_{\ast} + 1}(w)}\biggl(\liminf_{k \to \infty}\widetilde{\mathcal{E}}_{p,S^{k}(v)}^{k + n(r)}(f) + \liminf_{k \to \infty}\widetilde{\mathcal{E}}_{p,S^{k}(w)}^{k + n(r)}(f)\biggr) \nonumber \\
        &\hspace*{130pt}+ \sum_{w \in T_{n(r)}[E], v \in \Gamma_{M_{\ast} + 1}(w)}\sigma_{p}^{n(r)}\abs{f_{K_{v}} - f_{K_{w}}}^{p} \nonumber \\
        &\lesssim \sum_{w \in T_{n(r)}[E]}\liminf_{k \to \infty}\widetilde{\mathcal{E}}_{p,S^{k}(\Gamma_{M_{\ast} + 1}(w))}^{k + n(r)}(f) + \sum_{w \in T_{n(r)}[E]}\widetilde{\mathcal{E}}_{p,\Gamma_{M_{\ast} + 1}(w)}^{n(r)}(f).
    \end{align}
    Since the partition $\{ K_{w} \}_{w \in T}$ is uniformly finite, we have
    \begin{align}\label{KSu.1}
        \sum_{w \in T_{n(r)}[E]}\liminf_{k \to \infty}\widetilde{\mathcal{E}}_{p,S^{k}(\Gamma_{M_{\ast} + 1}(w))}^{k + n(r)}(f)
        &\le \liminf_{k \to \infty}\sum_{w \in T_{n(r)}[E]}\widetilde{\mathcal{E}}_{p,S^{k}(\Gamma_{M_{\ast} + 1}(w))}^{n + n(r)}(f) \nonumber \\
        &\lesssim \liminf_{k \to \infty}\widetilde{\mathcal{E}}_{p,T_{k}[(E)_{d,cr}]}^{k}(f).
    \end{align}
    We also have from Lemma \ref{lem.Kig-wm} that
    \begin{equation}\label{KSu.2}
        \sum_{w \in T_{n(r)}[E]}\widetilde{\mathcal{E}}_{p,\Gamma_{M_{\ast} + 1}(w)}^{n(r)}(f)
        \lesssim \widetilde{\mathcal{E}}_{p,T_{n(r)}[(E)_{d,cr}]}^{n(r)}(f)
        \lesssim \liminf_{n \to \infty}\widetilde{\mathcal{E}}_{p,T_{n}[(E)_{d,cr}]}^{n}(f).
    \end{equation}
    By \eqref{KSu.main}, \eqref{KSu.1} and \eqref{KSu.2}, there exists $C \in (0,\infty)$ (depending only on the constants associated with Assumption \ref{assum.BF}) such that for any $r \in (0,r_0)$,
    \begin{equation}\label{e:KSu.local.pre}
    \int_{E}\fint_{B_{\metric}(x, r)}\frac{\abs{f(x) - f(y)}^{p}}{r^{\pwalk}}\,\measure(dy)\measure(dx)
    \le C\liminf_{n \to \infty}\widetilde{\mathcal{E}}_{p,T_{n}[(E)_{d,cr}]}^{n}(f), 
    \end{equation}
    whence we obtain \eqref{KSu.local} by letting $r\downarrow 0$ in \eqref{e:KSu.local.pre}.
    If $f \in L^{p}(K,m)$, then we have \eqref{KSu.global} by letting $E \coloneqq K$ in \eqref{e:KSu.local.pre}. 
\end{proof}

Before proving inequalities in the converse direction matching \eqref{KSu.local} and \eqref{KSu.global}, let us introduce a localized version of $\mathcal{W}^{p}$.
\begin{defn}\label{defn.localKig}
    Let $U$ be a non-empty open subset of $K$.
    We define a linear subspace $\mathcal{W}^{p}_{\mathrm{loc}}(U)$ of $L^{0}(U,m|_{U})$ by 
    \begin{equation*}
        \mathcal{W}^{p}_{\mathrm{loc}}(U) \coloneqq
        \biggl\{ f \in L^{0}(U,m|_{U}) \biggm|
        \begin{minipage}{170pt}
		$f = f^{\#}$ $m$-a.e.\ on $V$ for some $f^{\#} \in \mathcal{W}^{p}$ for each relatively compact open subset $V$ of $U$
        \end{minipage}
        \biggr\}.
    \end{equation*}
\end{defn}

The following lower bound on localized Korevaar--Schoen energy functionals can be shown in a similar way as \cite[Theorem 5.2]{Bau22+}.
\begin{prop}\label{prop.lower}
    Suppose that Assumption \ref{assum.CH2} holds.
	Then there exists $C \in (0,\infty)$ such that for any $E \subseteq K$, any open neighborhood $E'$ of $\overline{E}^{K}$ and any $u \in \mathcal{W}^{p}_{\mathrm{loc}}(E')$,
	\begin{equation}\label{KSl.local}
		\limsup_{n \to \infty}\widetilde{\mathcal{E}}_{p, T_{n}[E]}^{n}(u) \le C\lim_{\delta \downarrow 0}\liminf_{r \downarrow 0}\int_{(E)_{d,\delta}}\fint_{B_{\metric}(x, r)}\frac{\abs{u(x) - u(y)}^{p}}{r^{\pwalk}}\,\measure(dy)\measure(dx).
	\end{equation} 
	Furthermore, with $C \in (0,\infty)$ the same as in \eqref{KSl.local}, for any $f \in L^{p}(K, \measure)$,
	\begin{equation}\label{KSl.global}
		\mathcal{N}_{p}(f)^{p} \le C\liminf_{r \downarrow 0}\int_{K}\fint_{B_{\metric}(x, r)}\frac{\abs{f(x) - f(y)}^{p}}{r^{\pwalk}}\,\measure(dy)\measure(dx).
	\end{equation}
\end{prop}
\begin{proof}
    Let $r \in (0, 1)$, let $N_{r}$ be a maximal $r$-net of $(K, d)$, and let $\{ \psi_{z,r} \}_{z \in N_{r}}$ be a partition of unity as given in Lemma \ref{lem.unity-Kig}.
    Define $A_{r} \colon L^{p}(K,m) \to \mathcal{W}^{p} \cap \contfunc(K)$ by $A_{r}f \coloneqq \sum_{z \in N_{r}}f_{B_{d}(z, r/4)}\psi_{z,r}$ for $f \in L^{p}(K,m)$.
    Then we can easily see that $\lim_{r \to 0}\norm{A_{r}f - f}_{L^{p}(K,m)} = 0$ and $\sup_{r > 0}\norm{A_{r}}_{L^{p}(K,m) \to L^{p}(K,m)} < \infty$.
    For any large $n \in \mathbb{N}$ so that $4c_{2}r_{\ast}^{n} < r$, where $c_{2}$ is the constant in \eqref{BF.bi-Lip}, a similar argument as in \cite[Lemma 7.4]{MS+} shows that there exists $C_{1} > 0$ depending only on the constants associated with Assumption \ref{assum.BF} such that
    \begin{align}\label{KSl.1}
        &\widetilde{\mathcal{E}}_{p, T_{n}[B_{\metric}(z, 5r/4)]}^{n}(A_{r}f) \nonumber \nonumber \\
        &\le C_{1}\sum_{w \in N_{r} \cap B_{\metric}(z, 11r/4)}\int_{B_{\metric}(w, 3r)}\fint_{B_{\metric}(x, 9r)}\frac{\abs{f(x) - f(y)}^{p}}{r^{\pwalk}}\,\measure(dy)\measure(dx).
    \end{align}
    Let us fix $\delta > 0$ and define $N_{r}(E) \coloneqq \{ z \in N_{r} \mid E \cap B_{d}(z,r) \neq \emptyset \}$.
    Then, for any small enough $r > 0$ so that $r < \delta/7$, we have $E \subseteq \bigcup_{z \in N_{r}(E)}B_{d}(z,5r/4)$ and
    \[
    \bigcup_{z \in N_{r}(E)}\bigcup_{w \in N_{r} \cap B_{d}(z,11r/4)}B_{d}(w,3r) \subseteq (E)_{d,\delta},
    \]
    whence we see that for any $f \in L^{p}(K,m)$,
    \begin{align}\label{KSl.2}
        &\widetilde{\mathcal{E}}_{p, T_{n}[E]}^{n}(A_{r}f) \nonumber \\
        &\le \sum_{z \in N_{r}(E)}\widetilde{\mathcal{E}}_{p, T_{n}[B_{\metric}(z, 5r/4)]}^{n}(A_{r}f) \nonumber \\
        &\overset{\eqref{KSl.1}}{\le} C_{1}\sum_{z \in N_{r}(E)}\sum_{w \in N_{r} \cap B_{\metric}(z, 11r/4)}\int_{B_{\metric}(w, 3r)}\fint_{B_{\metric}(x, 9r)}\frac{\abs{f(x) - f(y)}^{p}}{r^{\pwalk}}\,\measure(dy)\measure(dx) \nonumber \\
        &\lesssim \int_{(E)_{d,\delta}}\fint_{B_{\metric}(x, 9r)}\frac{\abs{f(x) - f(y)}^{p}}{r^{\pwalk}}\,\measure(dy)\measure(dx),
    \end{align}
    where we used the metric doubling property of $(K,d)$ in the last inequality.
    (Here, we consider small enough $r > 0$ so that $r < \delta/7$ and large enough $n \in \mathbb{N}$ so that $4c_{2}r_{\ast}^{n} < r$.)

    To prove the desired estimate \eqref{KSl.local} for $u \in \mathcal{W}^{p}_{\mathrm{loc}}(E')$, we fix a relatively compact open subset $V$ of $E'$ and $u^{\#} \in \mathcal{W}^{p}$ satisfying $V \supseteq \closure{E}^{K}$, $u^{\#} \in \mathcal{W}^{p}$ and $u = u^{\#}$ $m$-a.e.\ on $V$.
    Also, fix a sequence $\{ r_{k} \}_{k \in \mathbb{N}} \subseteq (0,\infty)$ such that $r_{k} \downarrow 0$ as $k \to \infty$ and
    \[
    \lim_{k \to \infty}r_{k}^{-\pwalk}J_{p,r_{k}}(u^{\#}\:\vert\:(E)_{d,\delta})
    = \liminf_{r \downarrow 0}r^{-\pwalk}J_{p,r}(u^{\#}\:\vert\:(E)_{d,\delta}) \le \mathcal{N}_{p}(u^{\#}) < \infty, 
    \]
    where $J_{p,r}(g\:\vert\:A) \coloneqq \int_{A}\fint_{B_{d}(x,r)}\frac{\abs{g(x) - g(y)}^{p}}{r^{\pwalk}}\,m(dy)m(dx)$ for $g \in L^{p}(K,m)$ and $A \in \mathcal{B}(K)$. 
    Set $u_{k} \coloneqq A_{r_{k}/9}u^{\#}$ for each $k \in \mathbb{N}$.
    By combining \eqref{KSl.2} with $E = K$ and \eqref{KSu.global}, for all large $k \in \mathbb{N}$, we have 
    \begin{equation}\label{mazur.approx}
        \mathcal{N}_{p}(u_{k})^{p}
        \lesssim \int_{K}\fint_{B_{\metric}(x, r_{k})}\frac{\abs{u^{\#}(x) - u^{\#}(y)}^{p}}{r_{k}^{\pwalk}}\,\measure(dy)\measure(dx)
        \lesssim \mathcal{N}_{p}(u^{\#})^{p} < \infty, 
    \end{equation}
    which implies that $\{ u_{k} \}_{k \in \mathbb{N}}$ is bounded in $\mathcal{W}^{p}$.
    Since $\mathcal{W}^{p}$ is reflexive by Theorem \ref{thm.Wp}, we can assume that $u_{k}$ converges weakly in $\mathcal{W}^{p}$ to some function $u_{\infty} \in \mathcal{W}^{p}$ as $k \to \infty$.
    Since $\mathcal{W}^{p}$ is continuously embedded in $L^{p}(K, \measure)$, we have $u_{\infty} = u^{\#}$.
    Hence, by Mazur's lemma and \eqref{KSl.2}, we obtain 
    \begin{equation}\label{e:mazur.after}
    	\limsup_{n \to \infty}\widetilde{\mathcal{E}}_{p, T_{n}[E]}^{n}(u^{\#}) 
    	\le \lim_{\delta \downarrow 0}\liminf_{r \downarrow 0}r^{-\pwalk}J_{p,r}(u^{\#}\:\vert\:(E)_{d,\delta}). 
    \end{equation}
    Note that, by \eqref{BF.bi-Lip}, $\bigcup_{w \in T_{n}[E]}K_{w} \subseteq V$ for all large enough $n \in \mathbb{N}$ and $(E)_{d,r + \delta} \subseteq V$ for all small enough $\delta,r \in (0,\infty)$. 
    For such $n$, $\delta$ and $r$, we have $\widetilde{\mathcal{E}}_{p, T_{n}[E]}^{n}(u^{\#}) = \widetilde{\mathcal{E}}_{p, T_{n}[E]}^{n}(u)$ and $J_{p,r}(u^{\#}\:\vert\:(E)_{d,\delta}) = J_{p,r}(u\:\vert\:(E)_{d,\delta})$, whence we obtain \eqref{KSl.local}.

    We next consider the case $E = K$. 
    Let $f \in L^{p}(K,m)$ and set $J_{p,r}(f) \coloneqq J_{p,r}(f\:\vert\:K)$ for $r > 0$.
    Similar to the previous case, we assume that $\{ r_{k} \}_{k \in \mathbb{N}}$ is a sequence of positive numbers such that $r_{k} \downarrow 0$ as $k \to \infty$ and
    \[
    \lim_{k \to \infty}r_{k}^{-\pwalk}J_{p,r_{k}}(f)
    = \liminf_{r \downarrow 0}r^{-\pwalk}J_{p,r}(f) < \infty,
    \]
    which together with \eqref{KSl.2} implies that $\{ A_{r_{k}/9}f \}_{k \in \mathbb{N}}$ is a bounded sequence in $\mathcal{W}^{p}$.
    A similar argument using Mazur's lemma as in the previous paragraph yields \eqref{KSl.global}.
\end{proof}

\subsection{Weak monotonicity and Poincar\'{e} inequality}
Now we can prove the main theorem of this section, which verifies \ref{KSwm} for the family of kernels $\bm{k}=\bm{k}^{s_{p}}$ defined by \eqref{KSkernel} and \eqref{Lp-Besov} for the first time in the setting of a $p$-conductively homogeneous compact metric space equipped with an Ahlfors regular measure.
This also solves a part of \cite[Section 6.3, Problem 4]{Kig23}.

\begin{thm}\label{thm.Kig-KS}
    Suppose that $(K,d,\{ K_{w} \}_{w \in T},m,p)$ satisfies Assumption \ref{assum.CH2}, and let $s_{p}$, $\bm{k} \coloneqq \bm{k}^{s_p}$ and $\mathrm{KS}^{1,p}$ be as defined in Example \ref{ex.KS}. 
    Then $s_{p} = \pwalk/p$, $\mathcal{W}^{p} = \mathrm{KS}^{1,p}$, $\mathcal{W}^{p} \cap \contfunc(K)$ is dense in $\mathcal{W}^{p}$, and \ref{KSwm} holds.
    Moreover, there exists $C \in [1,\infty)$ such that 
    \begin{equation}\label{KSfull}
    	C^{-1}\sup_{r > 0}J_{p,r}^{\bm{k}}(f) \le \mathcal{N}_{p}(f)^{p} \le C\liminf_{r \downarrow 0}J_{p,r}^{\bm{k}}(f) \quad \text{for any $f \in L^{p}(K,m)$.}
    \end{equation}
\end{thm}
\begin{proof}
    By \eqref{KSu.global} and \eqref{KSl.global}, we have $\mathcal{W}^{p} = B_{p,\infty}^{\pwalk/p}$ and \eqref{KSfull}.
    (Recall Example \ref{ex.KS} for the definition of $B_{p,\infty}^{s}$.)
    In particular, $s_{p} \ge \pwalk/p$.
    To show the opposite inequality, let $s > \pwalk/p$ and let $f \in \mathcal{W}^{p} \setminus \mathbb{R}\indicator{K}$.
    (Note that $\mathcal{W}^{p}$ contains a non-constant function by \eqref{Kig-capu}.)
    Let $A_{r} \colon L^{p}(K,m) \to \mathcal{W}^{p} \cap \contfunc(K)$ be the same operator as in the proof of Proposition \ref{prop.lower} for each $r \in (0,1)$.
    Then, by \eqref{KSl.2} with $E = K$, 
    \begin{equation}\label{diverge}
        \frac{r^{\pwalk}}{r^{sp}}\widetilde{\mathcal{E}}_{p}^{n}(A_{r}f) \le C\int_{K}\fint_{B_{d}(x,9r)}\frac{\abs{f(x) - f(y)}^{p}}{r^{sp}}\,m(dy)m(dx)
    \end{equation}
    for any $n \in \mathbb{N}$ and $r \in(0,1)$ with $4c_{2}r_{\ast}^{n} < r$, where $c_{2}$ is the constant in \eqref{BF.bi-Lip} and $C > 0$ is a constant independent of $f$, $r$, and $n$. 
    As in the proof of \cite[Theorem 3.21]{Kig23}, let $\{ \widetilde{\mathcal{E}}_{p}^{n_k} \}_{k \in \mathbb{N}}$ be a $\Gamma$-converging subsequence of $\{ \widetilde{\mathcal{E}}_{p}^{n} \}_{n \in \mathbb{N}}$ and define $\widehat{\mathcal{E}}_{p}$ as its $\Gamma$-limit.
    Since $\widehat{\mathcal{E}}_{p}$ is lower semicontinuous with respect to the $L^{p}(K,m)$-topology (see \cite[Proposition 6.8]{Dal}) and $\widehat{\mathcal{E}}_{p} \asymp \mathcal{N}_{p}(\,\cdot\,)^{p}$ (see \cite[pp.\ 45--46]{Kig23}), we see that
    \[
    0 < \mathcal{N}_{p}(f)^{p} \lesssim \widehat{\mathcal{E}}_{p}(f) \le \liminf_{r \downarrow 0}\widehat{\mathcal{E}}_{p}(A_{r}f) \le \liminf_{r \downarrow 0}\liminf_{k \to \infty}\widetilde{\mathcal{E}}_{p}^{n_k}(A_{r}f),
    \]
    which together with \eqref{diverge} and $\lim_{r \to 0}r^{\pwalk - sp} = \infty$ implies that $f \not\in B_{p,\infty}^{s}$.
    Since $s > \pwalk/p$ is arbitrary, we conclude that $\pwalk/p \ge s_{p}$.
    In particular, we obtain $\mathcal{W}^{p} = \mathrm{KS}^{1,p}$ and \ref{KSwm}.
    The inclusion $\mathcal{W}^{p} \subseteq \closure{\mathcal{W}^{p} \cap \contfunc(K)}^{\mathcal{W}^{p}}$ follows from \eqref{mazur.approx} and Mazur's lemma, so we complete the proof.
\end{proof}

\begin{cor}
	Suppose that $(K,d,\{ K_{w} \}_{w \in T},m,p)$ satisfies Assumption \ref{assum.CH}.
	Then any Korevaar--Schoen $p$-energy form $(\mathcal{E}_{p}^{\mathrm{KS}},\mathcal{W}^{p})$ on $(K,d,m)$, which exists by Theorems \ref{thm.Kig-KS} and \ref{thm.KS-energy} (recall Example \ref{ex.KS}), is a $p$-resistance form on $K$, and there exist $\alpha_{0},\alpha_{1} \in (0,\infty)$ such that for any such $(\mathcal{E}_{p}^{\mathrm{KS}},\mathcal{W}^{p})$,
    \begin{equation}\label{e:pRM.diam.Kig}
    	\alpha_{0}d(x,y)^{\pwalk - \hdim} \le R_{\mathcal{E}_{p}^{\mathrm{KS}}}(x,y) \le \alpha_{1}d(x,y)^{\pwalk - \hdim} \quad \text{for any $x,y \in K$.}
    \end{equation}
\end{cor}
\begin{proof}
	Define $\bm{k} \coloneqq \bm{k}^{s}$ by \eqref{KSkernel} with $s = \pwalk/p$.
	Then by Theorem \ref{thm.Kig-KS}, Proposition \ref{prop.capu-CH} and \cite[Theorem 3.2]{Bau22+}, the assumptions of Proposition \ref{prop.KS-RF} with $\hdim,\pwalk$ in place of $Q,\beta_{p}$ hold under Assumption \ref{assum.CH}, so $(\mathcal{E}_{p}^{\mathrm{KS}}, \mathcal{W}^{p})$ is a $p$-resistance form on $K$. 
    The estimate \eqref{e:pRM.diam.Kig} follows from the $\hdim$-Ahlfors regularity of $m$ and Proposition \ref{prop.KS-RF}. 
\end{proof}

We also have a Poincar\'{e}-type inequality in terms of the localized versions of $(\KSform,\mathcal{W}^{p})$.
(For the Vicsek set, such a Poincar\'{e}-type inequality was proved in \cite[Corollary 4.2]{BC23}.)
\begin{prop}\label{prop.PI-Kig}
    Suppose that $(K,d,\{ K_{w} \}_{w \in T},m,p)$ satisfies Assumption \ref{assum.CH2}.
    Then there exist $C \in (0,\infty)$ and $A \in [1,\infty)$ such that for any $(z,s) \in K \times (0,1]$ and any $f \in \mathcal{W}_{\mathrm{loc}}^{p}(B_{d}(z,As))$,
    \begin{align}\label{PI-Kig}
        &\int_{B_{d}(z,s)}\abs{f(y) - f_{B_{d}(z,s)}}^{p}\,m(dy) \nonumber \\
        &\le Cs^{\pwalk}\liminf_{r \downarrow 0}\int_{B_{d}(z,As)}\fint_{B_{d}(x,r)}\frac{\abs{f(x) - f(y)}^{p}}{r^{\pwalk}}\,m(dy)m(dx). 
    \end{align}
\end{prop}
\begin{proof}
    Throughout this proof, $M_{\ast} \in \mathbb{N}$ and $r_{\ast} \in (0,1)$ are the same constants as in Assumption \ref{assum.BF}.
    We assume that $f \in \mathcal{W}^{p}$ for simplicity.
    Let $(z,s) \in K \times (0,1]$ and choose $n \in \mathbb{N}$ satisfying $c_{3}r_{\ast}^{n} \ge s > c_{3}r_{\ast}^{n + 1}$, where $c_{3}$ is the constant in \eqref{BF.adapted}.
    Let $f \in L^{p}(K,m)$ and set $\Gamma_{M_{\ast}}(z; n) \coloneqq \{ v \in T \mid \text{$v \in \Gamma_{M_{\ast}}(w)$ for some $w \in T_{n}$ with $z \in K_{w}$} \}$.
    Then we see that 
    \begin{align}\label{Kig.var}
        &\int_{U_{M_{\ast}}(z; n)}\abs{f(y) - f_{U_{M_{\ast}}(x; n)}}^{p}\,m(dy) \nonumber \\
        &\le \sum_{w \in \Gamma_{M_{\ast}}(z; n)}\int_{K_{w}}\abs{f(y) - f_{U_{M_{\ast}}(x; n)}}^{p}\,m(dy) \nonumber\\
        &\le 2^{p - 1}\sum_{w \in \Gamma_{M_{\ast}}(z; n)}\left(\int_{K_{w}}\abs{f(y) - f_{K_{w}}}^{p}\,m(dy) + m(K_{w})\abs{f_{K_{w}} - f_{U_{M_{\ast}}(x; n)}}^{p}\right)  \nonumber \\
        &\lesssim \sum_{w \in \Gamma_{M_{\ast}}(z; n)}\left(s^{\pwalk}\liminf_{k \to \infty}\widetilde{\mathcal{E}}_{p,S^{k}(w)}^{n + k}(f) + s^{\hdim}\abs{f_{K_{w}} - f_{U_{M_{\ast}}(z; n)}}^{p}\right).
    \end{align}
    Since $\min_{v \in \Gamma_{M_{\ast}}(z; n)}f_{K_{v}} \le f_{U_{M_{\ast}}(z; n)} \le \max_{v \in \Gamma_{M_{\ast}}(z; n)}f_{K_{v}}$, for any $w \in \Gamma_{M_{\ast}}(z; n)$ there exists $w' \in \Gamma_{M_{\ast}}(z; n) \setminus \{ w \}$ such that $\abs{f_{K_{w}} - f_{U_{M_{\ast}}(z; n)}} \le \abs{f_{K_{w}} - f_{K_{w'}}}$, which together with H\"{o}lder's inequality yields that
    \begin{equation}\label{avediff.2}
        \abs{f_{K_{w}} - f_{U_{M_{\ast}}(z; n)}}^{p}
        \lesssim \mathcal{E}_{p,\Gamma_{2M_{\ast}}(w)}^{n}(f)
        \lesssim s^{\pwalk - \hdim}\liminf_{k \to \infty}\widetilde{\mathcal{E}}_{p,S^{k}(\Gamma_{2M_{\ast}}(w))}^{n + k}(f),
    \end{equation}
    where we used \eqref{pCH.1} and \cite[(2.17)]{Kig23} in the last inequality.
    Note that $\sup_{v \in T}\#\Gamma_{M}(w) \le L_{\ast}^{M}$ by \eqref{BF.adapted} and the volume doubling property of $m$. 
    This observation together with \eqref{Kig.var} and \eqref{avediff.2} implies that  
    \begin{align*}
        &\int_{U_{M_{\ast}}(z; n)}\abs{f(y) - f_{U_{M_{\ast}}(x; n)}}^{p}\,m(dy)  \\
        &\lesssim s^{\pwalk}\liminf_{k \to \infty}\sum_{w \in \Gamma_{M_{\ast}}(z; n)}\widetilde{\mathcal{E}}_{p,S^{k}(\Gamma_{2M_{\ast}}(w))}^{n + k}(f) 
        \lesssim s^{\pwalk}\liminf_{k \to \infty}\widetilde{\mathcal{E}}_{p,T_{k}[B_{d}(z,As/2)]}^{k}(f) \\
        &\overset{\eqref{KSl.local}}{\lesssim} s^{\pwalk}\liminf_{r \downarrow 0}\int_{B_{d}(z,As)}\fint_{B_{d}(x,r)}\frac{\abs{f(y) - f(x)}^{p}}{r^{\pwalk}}\,m(dy)m(dx),
    \end{align*}
    which yields \eqref{PI-Kig} in the case $f \in \mathcal{W}^{p}$ since 
    \[
    \int_{U_{M_{\ast}}(z; n)}\abs{f(y) - f_{U_{M_{\ast}}(x; n)}}^{p}\,m(dy) \gtrsim \int_{B_{d}(z,s)}\abs{f(y) - f_{B_{d}(z,s)}}^{p}\,m(dy). 
    \]
    The case $f \in \mathcal{W}^{p}_{\mathrm{loc}}(B_{d}(x,A's))$, where $A' > A$ (set, e.g., $A' = 2A$), is similar. 
\end{proof}

\subsection{Self-similar $p$-energy forms based on Korevaar--Schoen $p$-energy forms}
In this subsection, we construct a self-similar $p$-energy form by improving \cite[Theorem 4.6]{Kig23}.
We need some preparations before constructing such a good self-similar $p$-energy form. 
We first review basic notation and terminology on self-similar structures. 
In particular, we recall the notion of a post-critically finite self-similar structure introduced by Kigami \cite{Kig93}, which is mainly dealt with in the next section. 
See \cite[Section 1]{Kig01} and \cite[Chapter 1]{Kig09} for further details.
Throughout this section, we fix a compact metrizable space $K$, a finite set $S$ with $\#S\geq 2$ and a continuous injective map $F_{i}\colon K\to K$ for each $i\in S$. 
We set $\mathcal{L} \coloneqq (K,S,\{F_{i}\}_{i\in S})$.
\begin{defn}\label{d:shift}
\begin{enumerate}[label=\textup{(\arabic*)},align=left,leftmargin=*,topsep=5pt,parsep=0pt,itemsep=2pt]
\item Let $W_{0} \coloneqq \{\emptyset\}$, where $\emptyset$ is an element
	called the \emph{empty word}, let
	$W_{n} \coloneqq S^{n}=\{w_{1}\dots w_{n}\mid w_{i}\in S\text{ for }i\in\{1,\dots,n\}\}$
	for $n\in\mathbb{N}$ and let $W_{*} \coloneqq \bigcup_{n\in\mathbb{N}\cup\{0\}}W_{n}$.
	For $w\in W_{*}$, the unique $n\in\mathbb{N}\cup\{0\}$ with $w\in W_{n}$
	is denoted by $\lvert w\rvert$ and called the \emph{length of $w$}.
\item We set
	$\Sigma \coloneqq S^{\mathbb{N}}=\{\omega_{1}\omega_{2}\omega_{3}\ldots\mid \omega_{i}\in S\text{ for }i\in\mathbb{N}\}$,
	which is always equipped with the product topology of the discrete topology on $S$,
	and define the \emph{shift map} $\sigma\colon\Sigma\to\Sigma$ by
	$\sigma(\omega_{1}\omega_{2}\omega_{3}\dots)\coloneqq\omega_{2}\omega_{3}\omega_{4}\dots$.
	For $i\in S$ we define $\sigma_{i}\colon\Sigma\to\Sigma$  by
	$\sigma_{i}(\omega_{1}\omega_{2}\omega_{3}\dots) \coloneqq i\omega_{1}\omega_{2}\omega_{3}\dots$.
	For $\omega=\omega_{1}\omega_{2}\omega_{3}\ldots\in\Sigma$ and
	$n\in\mathbb{N}\cup\{0\}$, we write $[\omega]_{n} \coloneqq \omega_{1}\dots\omega_{n}\in W_{n}$.
\item For $w=w_{1}\dots w_{n}\in W_{*}$, we set
	$F_{w} \coloneqq F_{w_{1}}\circ\dots\circ F_{w_{n}}$ ($F_{\emptyset} \coloneqq \id_{K}$),
	$K_{w} \coloneqq F_{w}(K)$, $\sigma_{w} \coloneqq \sigma_{w_{1}}\circ\dots\circ \sigma_{w_{n}}$
	($\sigma_{\emptyset} \coloneqq \id_{\Sigma}$) and $\Sigma_{w} \coloneqq \sigma_{w}(\Sigma)$. 
\item Let $w,v \in W_{\ast}$, $w = w_{1} \dots w_{n_{1}}$, $v = v_{1} \dots v_{n_{2}}$. We define $wv \in W_{\ast}$ by $wv \coloneqq w_{1} \dots w_{n_{1}}v_{1} \dots v_{n_{2}} \, (w\emptyset \coloneqq w, \emptyset v \coloneqq v)$. We write $w \le v$ if and only if $w = v\tau$ for some $\tau \in W_{\ast}$. 
        \item A finite subset $\Lambda$ of $W_{\ast}$ is called a \emph{partition} of $\Sigma$ if and only if $\Sigma_{w} \cap \Sigma_{v} = \emptyset$ for any $w,v \in \Lambda$ with $w \neq v$ and $\Sigma = \bigcup_{w \in \Lambda}\Sigma_{w}$.
        \item Let $\Lambda_{1},\Lambda_{2}$ be partitions of $\Sigma$. We say that $\Lambda_{1}$ is a \emph{refinement} of $\Lambda_{2}$, and write $\Lambda_{1} \le \Lambda_{2}$, if and only if for each $w^{1} \in \Lambda_{1}$ there exists $w^{2} \in \Lambda_{2}$ such that $w^{1} \le w^{2}$.
\end{enumerate}
\end{defn}
\begin{defn}\label{d:sss}
$\mathcal{L}=(K,S,\{F_{i}\}_{i\in S})$ is called a \emph{self-similar structure}
if and only if there exists a continuous surjective map $\chi\colon\Sigma\to K$ such that
$F_{i}\circ\chi=\chi\circ\sigma_{i}$ for any $i\in S$.
Note that such $\chi$, if it exists, is unique and satisfies
$\{\chi(\omega)\}=\bigcap_{n\in\mathbb{N}}K_{[\omega]_{n}}$ for any $\omega\in\Sigma$.
\end{defn}

\begin{defn}\label{d:V0Vstar}
Let $\mathcal{L} = (K,S,\{ F_{i} \}_{i \in S})$ be a self-similar structure.
\begin{enumerate}[label=\textup{(\arabic*)},align=left,leftmargin=*,topsep=5pt,parsep=0pt,itemsep=2pt]
\item We define the \emph{critical set} $\mathcal{C}_{\mathcal{L}}$ and the
	\emph{post-critical set} $\mathcal{P}_{\mathcal{L}}$ of $\mathcal{L}$ by
	\begin{equation*}\label{e:C-P}\textstyle
		\mathcal{C}_{\mathcal{L}}\coloneqq\chi^{-1}\bigl(\bigcup_{i,j\in S,\,i\not=j}K_{i}\cap K_{j}\bigr)
		\qquad\text{and}\qquad
		\mathcal{P}_{\mathcal{L}}\coloneqq\bigcup_{n\in\mathbb{N}}\sigma^{n}(\mathcal{C}_{\mathcal{L}}).
	\end{equation*}
	$\mathcal{L}$ is called \emph{post-critically finite}, \emph{p.-c.f.}\ for short,
	if and only if $\mathcal{P}_{\mathcal{L}}$ is a finite set.
\item We set $V_{0}\coloneqq\chi(\mathcal{P}_{\mathcal{L}})$, $V_{n}\coloneqq\bigcup_{w\in W_{n}}F_{w}(V_{0})$
	for $n\in\mathbb{N}$ and $V_{*}\coloneqq\bigcup_{n\in\mathbb{N} \cup \{0\}}V_{n}$.
\end{enumerate}
\end{defn}
The set $V_{0}$ should be considered as the \emph{``boundary"} of the self-similar set $K$;
indeed, $K_{w}\cap K_{v}=F_{w}(V_{0})\cap F_{v}(V_{0})$ for any $w,v\in W_{*}$ with
$\Sigma_{w}\cap\Sigma_{v}=\emptyset$ by \cite[Proposition 1.3.5-(2)]{Kig01}.
According to \cite[Lemma 1.3.11]{Kig01}, $V_{n-1}\subseteq V_{n}$ for any $n\in\mathbb{N}$, and $V_{*}$ is dense in $K$ if $V_{0} \neq \emptyset$.

\begin{defn}[Self-similar measure]
    Let $\mathcal{L} = (K,S,\{ F_{i} \}_{i \in S})$ be a self-similar structure and let $(\theta_{i})_{i \in S} \in (0,1)^{S}$ satisfy $\sum_{i \in S}\theta_{i} = 1$.
    A Borel probability measure $m$ on $K$ is said to be a \emph{self-similar measure on $\mathcal{L}$ with weight $(\theta_{i})_{i \in S}$} if and only if the following equality (of Borel measures on $K$) holds:
    \begin{equation}\label{e:ss-meas}
        m = \sum_{i \in S}\theta_{i}(F_{i})_{\ast}m.
    \end{equation}
\end{defn}

Next we introduce the notion of self-similarity for $p$-energy forms and $p$-resistance forms. 
\begin{defn}[Self-similar $p$-energy/$p$-resistance form]\label{defn.ssform}
	Let $\mathcal{L} = (K,S,\{ F_{i} \}_{i \in S})$ be a self-similar structure and let $m$ be a Radon measure on $K$ with full topological support.
	Let $(\rweight_{p,s})_{s \in S} \in (0,\infty)^{S}$ and define $\rweight_{p,w} \coloneqq \rweight_{p,w_1} \cdots \rweight_{p,w_{n}}$ for each $w = w_{1} \dots w_{n} \in W_{\ast}$. 
	A $p$-energy form $(\mathcal{E}_{p},\mathcal{F}_{p})$ on $(K,m)$ (with $\mathcal{F}_{p} \subseteq L^{p}(K,m)$) is called a \emph{self-similar $p$-energy form on $(\mathcal{L},m)$ with weight $(\rweight_{p,s})_{s \in S}$} if and only if the following hold: 
	\begin{gather}
		\mathcal{F}_{p} \cap \contfunc(K) = \{ u \in \contfunc(K) \mid \text{$u \circ F_{s} \in \mathcal{F}_{p}$ for any $s \in S$} \}, \label{domain.ss}\\
		\mathcal{E}_{p}(u) = \sum_{s \in S}\rweight_{p,s}\mathcal{E}_{p}(u \circ F_{s}) \quad \text{for any $u \in \mathcal{F}_{p} \cap \contfunc(K)$.} \label{form.ss}
	\end{gather}
	If $\mathcal{F}_{p} \subseteq \contfunc(K)$ and $(\mathcal{E}_{p},\mathcal{F}_{p})$ is a $p$-resistance form on $K$ satisfying \eqref{domain.ss} and \eqref{form.ss}, then $(\mathcal{E}_{p},\mathcal{F}_{p})$ is called a \emph{self-similar $p$-resistance form on $\mathcal{L}$ with weight $(\rweight_{p,s})_{s \in S}$}. 
\end{defn}

We will focus on self-similar structures having \emph{rationally related contraction ratios} as in \cite{Kig23}. 
In the next definition, we introduce a good partition parametrized by a rooted tree. 
\begin{defn}[{\cite[Definition 4.2]{Kig23}}]\label{d:scale}
    Let $\mathcal{L} = (K,S,\{ F_{i} \}_{i \in S})$ be a self-similar structure, let $r \in (0,1)$ and let $(j_{s})_{s \in S} \in \mathbb{R}^{S}$.
    Define
    \[
    j(w) \coloneqq \sum_{i = 1}^{n}j_{w_{i}} \quad \text{and} \quad g(w) \coloneqq r^{j(w)} \quad \text{for $w = w_{1}\dots w_{n} \in W_{n}$.}
    \]
    Define $\widetilde{\pi}(w_{1} \cdots w_{n}) \coloneqq w_{1} \cdots w_{n-1}$ for $w = w_{1}\dots w_{n} \in W_{n}$ and
    \[
    \Lambda_{r^{k}}^{g} \coloneqq \{ w = w_{1} \cdots w_{n} \in W_{\ast} \mid g(\widetilde{\pi}(w)) > r^{k} \ge g(w) \}.
    \]
    Set $T_{k}^{(r)} \coloneqq \{ (k,w) \mid w \in \Lambda_{r^{k}}^{g} \}$ and $T^{(r)} \coloneqq \bigcup_{k \in \mathbb{N} \cup \{ 0 \}}T_{k}^{(r)}$.
    Moreover, define $E_{T^{(r)}} \subseteq T^{(r)} \times T^{(r)}$ by
    \[
    E_{T^{(r)}} \coloneqq \Bigl\{ ((k,v), (k + 1,w)) \in T_{k}^{(r)} \times T_{k + 1}^{(r)} \Bigm| k \in \mathbb{N} \cup \{ 0 \}, \text{$v = w$ or $v = \widetilde{\pi}(w)$} \Bigr\}.
    \]
\end{defn}

We introduce the following assumption in order to construct a self-similar $p$-energy form on $(\mathcal{L},m)$.
(Recall that we have fixed $p \in (1,\infty)$.)
\begin{assum}\label{assum.CH-ss}
    Let $\mathcal{L} = (K,S,\{ F_{i} \}_{i \in S})$ be a self-similar structure such that $\#S \ge 2$ and $K$ is connected.
    There exist $r_{\ast} \in (0,1)$, $(j_s)_{s \in S} \in \mathbb{N}^{S}$ and a metric $d$ giving the original topology of $K$ with $\diam(K,d) = 1$ such that $(K,d,\{ K_{w} \}_{w \in T^{(r_{\ast})}},m,p)$ satisfies Assumption \ref{assum.CH2}, where $\hdim \in (0,\infty)$ is such that $\sum_{s \in S}r_{\ast}^{j_{s}\hdim} = 1$ and $m$ is the self-similar measure on $K$ with weight $(r_{\ast}^{j_{s}\hdim})_{s \in S}$.
    (The collection $\{ F_{i} \}_{i \in S}$ is said to have rationally related contraction ratios $(r_{\ast}^{j_s})_{s \in S}$.)
\end{assum} 
 
Under Assumption \ref{assum.CH-ss}, we have $V_{0} \neq \emptyset$ since $K$ is connected and $\#S \ge 2$ (see \cite[Proposition 1.3.5-(3)]{Kig01} or \cite[Theorem 1.6.2]{Kig01}). 
Also, we can easily show that $m$ is $\hdim$-Ahlfors regular as stated in the following proposition (see \cite[Proposition 4.5]{Kig23}). 
\begin{prop}\label{prop.ARss}
	Suppose that $\mathcal{L}$ is a self-similar structure and that there exist $r_{\ast} \in (0,1)$, $(j_s)_{s \in S} \in \mathbb{N}^{S}$ and a metric $d$ giving the original topology of $K$ with $\diam(K,d) = 1$ such that $(K,d,\{ K_{w} \}_{w \in T^{(r_{\ast})}},m)$ satisfies Assumption \ref{assum.BF}. 
	Let $\hdim \in (0,\infty)$ be such that $\sum_{s \in S}r_{\ast}^{j_{s}\hdim} = 1$ and let $m$ be the self-similar measure on $K$ with weight $(r_{\ast}^{j_{s}\hdim})_{s \in S}$.
	Then $\hdim$ is the Hausdorff dimension of $(K,d)$ and $m$ is $\hdim$-Ahlfors regular with respect to $d$. 
\end{prop}

To construct a self-similar $p$-energy form, we need to take care of the pre-self-similarity condition (see \cite[Theorem 8.12]{MS+}).
We can easily verify this condition in the case $\sigma_{p} > 1$ by modifying \cite[Proof of Theorem 4.6]{Kig23}; see \cite[Section 8.2]{KS.gc} for details. 
\begin{prop}
	Suppose that Assumption \ref{assum.CH-ss} holds and that $\sigma_{p} > 1$. 
	Then \eqref{domain.ss} with $\mathcal{W}^{p}$ in place of $\mathcal{F}_{p}$ holds and there exists $C \in [1,\infty)$ such that for any $n \in \mathbb{N}$ and any $u \in \mathcal{W}^{p} \subseteq \contfunc(K)$, 
	\[
	C^{-1}\sum_{w \in W_{n}}\sigma_{p}^{j(w)}\mathcal{N}_{p}(u \circ F_{w})^{p} \le \mathcal{N}_{p}(u)^{p} \le C\sum_{w \in W_{n}}\sigma_{p}^{j(w)}\mathcal{N}_{p}(u \circ F_{w})^{p}.
	\]
\end{prop}

Now we can present an improvement of \cite[Theorem 4.6]{Kig23} in the following formulation. 
\begin{thm}\label{thm.KSss-Kig}
	Suppose that Assumption \ref{assum.CH-ss} holds, that \eqref{domain.ss} with $\mathcal{W}^{p}$ in place of $\mathcal{F}_{p}$ holds, and that there exists $C_{0} \in [1,\infty)$ such that for any $n \in \mathbb{N}$ and any $u \in \mathcal{W}^{p} \cap \contfunc(K)$,
	\begin{equation}\label{form.pss}
		C_{0}^{-1}\sum_{w \in W_{n}}\sigma_{p}^{j(w)}\mathcal{N}_{p}(u \circ F_{w})^{p} \le \mathcal{N}_{p}(u)^{p} \le C_{0}\sum_{w \in W_{n}}\sigma_{p}^{j(w)}\mathcal{N}_{p}(u \circ F_{w})^{p}. 
	\end{equation}
	For each $n \in \mathbb{N}$, define $\bm{k}^{(n)} = \{ k_{r}^{(n)} \}_{r > 0}$ by 
	\begin{equation*}
		k_{r}^{(n)}(x,y) \coloneqq \frac{1}{n + 1}\sum_{l = 0}^{n}\sum_{w \in W_{l}}r_{\ast}^{-j(w)\cdot(\pwalk + \hdim)}\frac{\indicator{A_{w,r}}(x,y)}{r^{\pwalk + \hdim}}, \quad x,y \in K, 
	\end{equation*}
	where $A_{w,r} \coloneqq \bigl\{ (x,y) \in K_{w} \times K_{w} \bigm| d(F_{w}^{-1}(x),F_{w}^{-1}(y)) < r \bigr\}$. 
	Then $\bm{k}^{(n)}$ is asymptotically local, \hyperref[KSwm]{\textup{(WM)$_{p,\bm{k}^{(n)}}$}} holds, $B_{p,\infty}^{\bm{k}^{(n)}} = \mathcal{W}^{p}$,
	and for any sequence $\{(\mathcal{E}_{p}^{\bm{k}^{(n)}},\mathcal{W}^{p})\}_{n\in\mathbb{N}}$ with $(\mathcal{E}_{p}^{\bm{k}^{(n)}},\mathcal{W}^{p})$ a $\bm{k}^{(n)}$-Korevaar--Schoen $p$-energy form on $(K,m)$ for each $n \in \mathbb{N}$,
	there exists a sequence $\{ n_{j} \}_{j \in \mathbb{N}} \subseteq \mathbb{N}$ with $n_{j} < n_{j + 1}$ for any $j \in \mathbb{N}$ such that the following limit exists in $[0,\infty)$ for any $u \in \mathcal{W}^{p}$: 
	\begin{equation}\label{KSss.constr}
         \breve{\mathcal{E}}_{p}^{\mathrm{KS}}(u) \coloneqq \lim_{j \to \infty}\mathcal{E}_{p}^{\bm{k}^{(n_j)}}(u). 
    \end{equation}
	Moreover, for any such $\{ \mathcal{E}_{p}^{\bm{k}^{(n)}} \}_{n \in \mathbb{N}}$ and $\{ n_{j} \}_{j \in \mathbb{N}}$, the functional $\breve{\mathcal{E}}_{p}^{\mathrm{KS}} \colon \mathcal{W}^{p} \to [0,\infty)$ defined by \eqref{KSss.constr} satisfies the following properties: 
    \begin{enumerate}[label=\textup{(\alph*)},align=left,leftmargin=*,topsep=2pt,parsep=0pt,itemsep=2pt]
    	\item\label{it:KS-Kig.ss} $(\breve{\mathcal{E}}_{p}^{\mathrm{KS}},\mathcal{W}^{p})$ is a self-similar $p$-energy form on $(\mathcal{L},m)$ with weight $(\sigma_{p}^{j_s})_{s \in S}$.  
    	\item\label{it:KS-Kig.comparable} For any $u \in \mathcal{W}^{p}$, 
    		\begin{equation}\label{e:KS-Kig.comparable}
    			(CC_{0})^{-1}\mathcal{N}_{p}(u)^{p} \le \breve{\mathcal{E}}_{p}^{\mathrm{KS}}(u) \le CC_{0}\mathcal{N}_{p}(u)^{p}, 
    		\end{equation}
    		where $C,C_{0} \in [1,\infty)$ are the constants in \eqref{KSfull} and in \eqref{form.pss} respectively.
    	\item\label{it:KS-Kig.GC} $(\breve{\mathcal{E}}_{p}^{\mathrm{KS}},\mathcal{W}^{p})$ satisfies \ref{GC}. Furthermore, for any $u,v \in \mathcal{W}^{p}$, $\{ \mathcal{E}_{p}^{\bm{k}^{(n_{j})}}(u; v) \}_{j \in \mathbb{N}}$ is convergent in $\mathbb{R}$ and 
			\begin{equation}\label{KS-Kig.compatible}
				\breve{\mathcal{E}}_{p}^{\mathrm{KS}}(u; v) = \lim_{j \to \infty}\mathcal{E}_{p}^{\bm{k}^{(n_{j})}}(u; v).  
			\end{equation}
		\item\label{it:KS-Kig.basic} Theorem \ref{thm.KS-energy}-\ref{it:difffunc},\ref{KS.conti},\ref{KS.slbdd} with $(\breve{\mathcal{E}}_{p}^{\mathrm{KS}},\mathcal{W}^{p})$ in place of $(\mathcal{E}_{p}^{\bm{k}},B_{p,\infty}^{\bm{k}})$ hold. 
		\item\label{it:KS-Kig.inv} For any isometric map $T \colon (K,d) \to (K,d)$ preserving $m$, $u \circ T \in \mathcal{W}^{p}$ and $\breve{\mathcal{E}}_{p}^{\mathrm{KS}}(u \circ T) = \breve{\mathcal{E}}_{p}^{\mathrm{KS}}(u)$ for any $u \in \mathcal{W}^{p}$.
		\item\label{it:KS-Kig.pRF} If in addition $\sigma_{p} > 1$, then $(\breve{\mathcal{E}}_{p}^{\mathrm{KS}},\mathcal{W}^{p})$ is a $p$-resistance form on $K$, and there exist $\alpha_{0},\alpha_{1} \in (0,\infty)$ independent of particular choices of $\{ \mathcal{E}_{p}^{\bm{k}^{(n)}} \}_{n \in \mathbb{N}}$ and $\{ n_j \}_{j \in \mathbb{N}}$ such that 
			\begin{equation}\label{e:KS-Kig;Rp-compl}
        		\alpha_{0}d(x,y)^{\pwalk - \hdim} \le R_{\breve{\mathcal{E}}_{p}^{\mathrm{KS}}}(x,y) \le \alpha_{1}d(x,y)^{\pwalk - \hdim} \quad \text{for any $x,y \in K$.}
    		\end{equation}
    \end{enumerate}
\end{thm}
\begin{proof}
	Set $\bm{k} = \{ k_{r} \}_{r > 0}$ by $k_{r}(x,y) \coloneqq r^{-\pwalk - \hdim}\indicator{B_{d}(x,r)}(y)$. 
	Recall that $B_{p,\infty}^{\bm{k}} = \mathrm{KS}^{1,p}$ since $ps_{p} = \pwalk$ and $m$ is $\hdim$-Ahlfors regular (see Example \ref{ex.KS}, Theorem \ref{thm.Kig-KS} and Proposition \ref{prop.ARss}).
	By using \eqref{BF.adapted}, we can easily see that $\bm{k}^{(n)}$ is asymptotically local. 
	Let us show \hyperref[KSwm]{\textup{(WM)$_{p,\bm{k}^{(n)}}$}}. 
	Note that for any $r > 0$ and any $u \in L^{p}(K,m)$, we have 
	\begin{equation}\label{kernel.ss}
		J_{p,r}^{\bm{k}^{(n)}}(u)
		= \frac{1}{n + 1}\sum_{l = 0}^{n}\sum_{w \in W_{l}}\sigma_{p}^{j(w)}J_{p,r}^{\bm{k}}(u \circ F_{w}), 
	\end{equation}
	where we used $(F_{w} \times F_{w})^{-1}(A_{w,r}) = \{ (x,y) \in K \times K \mid d(x,y) < r \}$ and $m = r_{\ast}^{j(w)\hdim}(F_{w})_{\ast}m$. 
	By combining \eqref{kernel.ss}, Theorem \ref{thm.Kig-KS} and \eqref{form.pss}, we obtain \hyperref[KSwm]{\textup{(WM)$_{p,\bm{k}^{(n)}}$}}. 
	Moreover, for any $n \in \mathbb{N}$ and any $u \in \mathcal{W}^{p}$,
	\begin{equation}\label{kernel.pss}
		(CC_{0})^{-1}\sup_{r > 0}J_{p,r}^{\bm{k}^{(n)}}(u) \le \mathcal{N}_{p}(u)^{p} \le CC_{0}\liminf_{r \to 0}J_{p,r}^{\bm{k}^{(n)}}(u),  
	\end{equation}
	where $C,C_{0} \in [1,\infty)$ are the constants in \eqref{KSfull} and in \eqref{form.pss} respectively.
	In particular, $B_{p,\infty}^{\bm{k}^{(n)}} = \mathcal{W}^{p}$ and $\{ \mathcal{E}_{p}^{\bm{k}^{(n)}}(u) \}_{n \in \mathbb{N}}$ is bounded for each $u \in \mathcal{W}^{p}$. 
	Since $\mathcal{W}^{p}$ is separable and $\mathcal{E}_{p}^{\bm{k}^{(n)}} \asymp \mathcal{N}_{p}(\,\cdot\,)^{p}$ by \eqref{kernel.pss}, a standard diagonal argument implies that there exists $\{ n_j \}_{j \in \mathbb{N}} \subseteq \mathbb{N}$ with $n_{j} < n_{j + 1}$ such that the limit $\lim_{j \to \infty}\mathcal{E}_{p}^{\bm{k}^{(n_j)}}(u) \eqqcolon \breve{\mathcal{E}}_{p}^{\mathrm{KS}}(u)$ exists for any $u \in \mathcal{W}^{p}$. 
	From this definition, \eqref{kernel.pss} and Theorem \ref{thm.KS-energy}-\ref{KS.GC}, we immediately see that \eqref{e:KS-Kig.comparable} holds and that $(\breve{\mathcal{E}}_{p}^{\mathrm{KS}},\mathcal{W}^{p})$ satisfies \ref{GC}.
	
	\ref{it:KS-Kig.ss}: 
	Since we assume that $\mathcal{W}^{p}$ satisfies \eqref{domain.ss}, it suffices to show the following equality for any $u \in \mathcal{W}^{p}$:
	\begin{equation}\label{KS-Kig.ss}
		\breve{\mathcal{E}}_{p}^{\mathrm{KS}}(u)
			= \sum_{s \in S}\sigma_{p}^{j_{s}}\breve{\mathcal{E}}_{p}^{\mathrm{KS}}(u \circ F_{s}).
	\end{equation}
	From Theorem \ref{thm.KS-energy} together with a diagonal argument, we can choose a sequence $\{ r_{l} \}_{l \in \mathbb{N}} \subseteq (0,\infty)$ with $\lim_{l \to \infty}r_{l} = 0$ such that $\mathcal{E}_{p}^{\bm{k}^{(n_j)}}(u) = \lim_{l \to \infty}J_{p,r_{l}}^{\bm{k}^{(n_j)}}(u)$ for any $j \in \mathbb{N}$ and any $u \in \mathcal{W}^{p}$.
	Using \eqref{kernel.ss}, we easily see that for any $(j,l) \in \mathbb{N}^{2}$ and any $u \in L^{p}(K,m)$, 
	\begin{align*}
		&\sum_{s \in S}\sigma_{p}^{j_{s}}J_{p,r_l}^{\bm{k}^{(n_j)}}(u \circ F_{s}) + \frac{1}{n_j + 1}J_{p,r_l}^{\bm{k}^{(n_j)}}(u) \\
		&= J_{p,r_l}^{\bm{k}^{(n_j)}}(u) + \frac{1}{n_j + 1}\sum_{w \in W_{n_j + 1}}\sigma_{p}^{j(w)}J_{p,r_l}^{\bm{k}^{(n_j)}}(u \circ F_{w}). 
	\end{align*}
	Letting $l \to \infty$ and $j \to \infty$, we obtain \eqref{KS-Kig.ss} by \eqref{kernel.pss} and \eqref{form.pss}.
	
	\ref{it:KS-Kig.GC}: 
	Similar to the proof of \eqref{KS-compatible}, by using Proposition \ref{prop.c-diff} and the convexity of $t \mapsto \mathcal{E}_{p}^{\bm{k}^{(n_{j})}}(u + tv)$, we can prove \eqref{KS-Kig.compatible}. 
	
	\ref{it:KS-Kig.basic}: 
	This is clear from Theorem \ref{thm.KS-energy}-\ref{it:difffunc},\ref{KS.conti},\ref{KS.slbdd} for $(\breve{\mathcal{E}}_{p}^{\bm{k}^{(n)}},\mathcal{W}^{p})$ and \eqref{KS-Kig.compatible}.
	
	\ref{it:KS-Kig.inv}: 
	If $T \colon (K,d) \to (K,d)$ is an isometric map preserving $m$, then for any $n\in\mathbb{N}$, $\bm{k}^{(n)}$ is clearly $T$-invariant, and hence by Theorem \ref{thm.KS-energy}-\ref{KS.inv} and $B_{p,\infty}^{\bm{k}^{(n)}} = \mathcal{W}^{p}$ we have $u \circ T \in \mathcal{W}^{p}$ and $\mathcal{E}_{p}^{\bm{k}^{(n)}}(u \circ T) = \mathcal{E}_{p}^{\bm{k}^{(n)}}(u)$ for any $u \in \mathcal{W}^{p}$, which together with \eqref{KSss.constr} implies that $\breve{\mathcal{E}}_{p}^{\mathrm{KS}}(u \circ T) = \breve{\mathcal{E}}_{p}^{\mathrm{KS}}(u)$ for any $u \in \mathcal{W}^{p}$. 
	
	\ref{it:KS-Kig.pRF}: 
	In the case $\sigma_{p} > 1$, we easily see that $(\breve{\mathcal{E}}_{p}^{\mathrm{KS}},\mathcal{W}^{p})$ is a $p$-resistance form on $K$ satisfying \eqref{e:KS-Kig;Rp-compl} by combining Proposition \ref{prop.KS-RF}, $\hdim$-Ahlfors regularity of $m$, $\pwalk > \hdim$ by $\sigma_{p} > 1$, Theorem \ref{thm.Wp}, Proposition \ref{prop.capu-CH} and \cite[Lemma 3.34]{Kig23}. 
\end{proof}



We collect properties of the $p$-energy measures associated with $(\breve{\mathcal{E}}_{p}^{\mathrm{KS}},\mathcal{W}^{p})$ in the following theorem. 
See also \cite[Sections 4 and 5]{KS.gc} for other basic properties. 
Let us emphasize that we do not know whether Theorem \ref{thm.Kigsspem}-\ref{it:Kigss.basic} below holds in a more general setting of self-similar $p$-energy forms like that of \cite{KS.gc}. 
\begin{thm}\label{thm.Kigsspem}
	Suppose the same assumptions as in Theorem \ref{thm.KSss-Kig}, let $(\mathcal{E}_{p}^{\bm{k}^{(n)}},\mathcal{W}^{p})$ be any $\bm{k}^{(n)}$-Korevaar--Schoen $p$-energy form on $(K,m)$ for each $n \in \mathbb{N}$, let $\{ n_{j} \}_{j \in \mathbb{N}} \subseteq \mathbb{N}$ be any sequence as in Theorem \ref{thm.KSss-Kig}, and let $(\breve{\mathcal{E}}_{p}^{\mathrm{KS}},\mathcal{W}^{p})$ be the $p$-energy form on $(K,m)$ defined by \eqref{KSss.constr}.
	Then for any $u \in \mathcal{W}^{p} \cap \contfunc(K)$, there exists a unique positive Radon measure $\breve{\Gamma}_{p}^{\mathrm{KS}}\langle u \rangle$ on $K$ such that 
	\begin{align}\label{Kigsspem.defn}
		&\int_{K}\varphi\,d\breve{\Gamma}_{p}^{\mathrm{KS}}\langle u \rangle \nonumber \\
		&= \breve{\mathcal{E}}_{p}^{\mathrm{KS}}(u; u\varphi) - \left(\frac{p - 1}{p}\right)^{p - 1}\breve{\mathcal{E}}_{p}^{\mathrm{KS}}\bigl(\abs{u}^{\frac{p}{p - 1}}; \varphi\bigr) \quad \text{for any $\varphi \in \mathcal{W}^{p} \cap \contfunc(K)$.}
	\end{align}
	Moreover, the following hold: 
	\begin{enumerate}[label=\textup{(\alph*)},align=left,leftmargin=*,topsep=2pt,parsep=0pt,itemsep=2pt]
		\item\label{it:Kigss.ext} For any $u \in \mathcal{W}^{p}$, there exists a unique positive Radon measure $\breve{\Gamma}_{p}^{\mathrm{KS}}\langle u \rangle$ on $K$ such that for any $\{ u_{n} \}_{n \in \mathbb{N}} \subseteq \mathcal{W}^{p} \cap \contfunc(K)$ with $\lim_{n \to \infty}\mathcal{N}_{p}(u - u_{n}) = 0$ and any Borel measurable function $\varphi \colon K \to [0,\infty)$ with $\norm{\varphi}_{\sup} < \infty$, 
			\begin{equation}\label{KSpem-Kig.extension}
				\int_{K}\varphi\,d\breve{\Gamma}_{p}^{\mathrm{KS}}\langle u \rangle
				= \lim_{n \to \infty}\int_{K}\varphi\,d\breve{\Gamma}_{p}^{\mathrm{KS}}\langle u_{n} \rangle, 
			\end{equation}
			and $\breve{\Gamma}_{p}^{\mathrm{KS}}\langle u \rangle$ further satisfies $\breve{\Gamma}_{p}^{\mathrm{KS}}\langle u \rangle(K) = \breve{\mathcal{E}}_{p}^{\mathrm{KS}}(u)$.  
			Moreover, for each such $\varphi$, $(\int_{K}\varphi\,d\breve{\Gamma}_{p}^{\mathrm{KS}}\langle \,\cdot\, \rangle,\mathcal{W}^{p})$ is a $p$-energy form on $K$ satisfying \ref{GC}.
		\item\label{it:Kigss.limit} Theorem \ref{thm.KSpem-diffble}, with $\mathcal{W}^{p}$ and $\breve{\Gamma}_{p}^{\bm{k}}$ in place of $\bclosureKS$ and $\KSem$ respectively, holds. In particular, for any $u,v \in \mathcal{W}^{p}$, 
			\begin{equation}\label{e:defn.newKSpem.Kigss}
				\breve{\Gamma}_{p}^{\mathrm{KS}}\langle u;v \rangle(A) \coloneqq \frac{1}{p}\left.\frac{d}{dt}\breve{\Gamma}_{p}^{\mathrm{KS}}\langle u + tv \rangle(A)\right|_{t = 0}, \quad A \in \mathcal{B}(K), 
			\end{equation}
			defines a signed Borel measure on $K$ such that $\breve{\Gamma}_{p}^{\mathrm{KS}}\langle u;v \rangle(K) = \breve{\mathcal{E}}_{p}^{\mathrm{KS}}(u; v)$ and $\breve{\Gamma}_{p}^{\mathrm{KS}}\langle u;u \rangle = \breve{\Gamma}_{p}^{\mathrm{KS}}\langle u \rangle$.
			Furthermore, for any $u,v \in \mathcal{W}^{p}$ and any $\varphi \in \contfunc(K)$,
			\begin{equation}\label{Kigsspem.compatible}
				\int_{K}\varphi\,d\breve{\Gamma}_{p}^{\mathrm{KS}}\langle u;v \rangle 
				= \lim_{j \to \infty}\int_{K}\varphi\,d\Gamma_{p}^{\bm{k}^{(n_{j})}}\langle u;v \rangle.
			\end{equation}
		\item\label{it:Kigss.basic} Theorem \ref{thm.KSpem-two}-\ref{it:KSpem.difffunc},\ref{it:KSpem.conti}, with $\mathcal{W}^{p}$ and $\breve{\Gamma}_{p}^{\mathrm{KS}}$ in place of $\bclosureKS$ and $\Gamma_{p}^{\bm{\bm{k}}}$ respectively, hold. 
		\item\label{it:Kigss.chain} Theorems \ref{thm.KSpem-chain}, \ref{thm.EIDP} and \ref{thm.KSpem-sl}, with $\mathcal{W}^{p} \cap \contfunc(K)$ and $\breve{\Gamma}_{p}^{\mathrm{KS}}$ in place of $\KS \cap \contfunc_{b}(K)$ and $\Gamma_{p}^{\bm{\bm{k}}}$ respectively, hold.
	\end{enumerate}
\end{thm}
\begin{proof}
	Fix $\{ n_{j} \}_{j \in \mathbb{N}} \subseteq \mathbb{N}$ so that $\breve{\mathcal{E}}_{p}^{\mathrm{KS}} = \lim_{j \to \infty}\mathcal{E}_{p}^{\bm{k}^{(n_{j})}}$. 
	Let $u \in \mathcal{W}^{p} \cap \contfunc(K)$. 
	Letting $j \to \infty$ in \eqref{KSpem.characterize} with $\mathcal{E}_{p}^{\bm{k}^{(n_{j})}}$ in place of $\KSform$ and using \eqref{KS-Kig.compatible}, we have 
	\[
	0 \le \Psi_{p}(u; \varphi) \coloneqq \breve{\mathcal{E}}_{p}^{\mathrm{KS}}(u; u\varphi) - \left(\frac{p - 1}{p}\right)^{p - 1}\breve{\mathcal{E}}_{p}^{\mathrm{KS}}\bigl(\abs{u}^{\frac{p}{p - 1}}; \varphi\bigr) \le \norm{\varphi}_{\sup}\breve{\mathcal{E}}_{p}^{\mathrm{KS}}(u)
	\]
	for any $\varphi \in \mathcal{W}^{p} \cap \contfunc(K)$ with $\varphi \ge 0$. 
	Since $\mathcal{W}^{p} \cap \contfunc(K)$ is dense in $\contfunc(K)$, we can get the desired positive Radon measure $\breve{\Gamma}_{p}^{\mathrm{KS}}\langle u \rangle$ (in the case $u \in \mathcal{W}^{p} \cap \contfunc(K)$) by using the Riesz--Markov--Kakutani representation theorem as done in the proof of Theorem \ref{thm.KSpem-exist}. 
	Also, we easily see that 
	\begin{equation}\label{Kitsspem.compatible.pre}
		\int_{K}\psi\,d\breve{\Gamma}_{p}^{\mathrm{KS}}\langle u \rangle 
		= \lim_{j \to \infty}\int_{K}\psi\,d\Gamma_{p}^{\bm{k}^{(n_j)}}\langle u \rangle \quad \text{for any $\psi \in \contfunc(K)$,}
	\end{equation}
	whence $\bigl(\int_{K}\psi\,d\breve{\Gamma}_{p}^{\mathrm{KS}}\langle \,\cdot\, \rangle, \mathcal{W}^{p} \cap \contfunc(K)\bigr)$ is a $p$-energy form on $(K,m)$ satisfying \ref{GC}. 
	Then we can prove \ref{it:Kigss.ext} by following the same argument as in the proof of Theorem \ref{thm.KSpem-GC}. 
	
	The property \ref{it:Kigss.limit} except for $\breve{\Gamma}_{p}^{\mathrm{KS}}\langle u;v \rangle(K) = \breve{\mathcal{E}}_{p}^{\mathrm{KS}}(u; v)$ and for \eqref{Kigsspem.compatible} follow from \cite[Theorem 4.5 and Proposition 4.6]{KS.gc}.
	The equality \eqref{Kigsspem.compatible} can be shown in the same way as the proof of \eqref{KS-compatible} by using \eqref{Kitsspem.compatible.pre}, Proposition \ref{prop.c-diff} and the convexity of $t \mapsto \int_{K}\varphi\,d\breve{\Gamma}_{p}^{\mathrm{KS}}\langle u + tv \rangle$. 
	We have $\breve{\Gamma}_{p}^{\mathrm{KS}}\langle u;v \rangle(K) = \breve{\mathcal{E}}_{p}^{\mathrm{KS}}(u;v)$ from Proposition \ref{prop.totalmass}, \eqref{KSss.constr} and \eqref{Kigsspem.compatible}.
	
	The statement \ref{it:Kigss.basic} and the chain rule \eqref{KSpem-chain} with $\breve{\Gamma}_{p}^{\mathrm{KS}}$ in place of $\KSem$ are immediate from \eqref{Kigsspem.compatible} and the corresponding  properties of $\Gamma_{p}^{\bm{k}^{(n_{j})}}$. 
	Since we can follow the proofs of Theorems \ref{thm.EIDP} and \ref{thm.KSpem-sl} by using the chain rule of $\breve{\Gamma}_{p}^{\mathrm{KS}}$, we complete the proof of \ref{it:Kigss.chain}. 
\end{proof}

\begin{rmk}
	There is another way to construct the $p$-energy measures associated with $(\breve{\mathcal{E}}_{p}^{\mathrm{KS}},\mathcal{W}^{p})$, which is based on the self-similarity \eqref{KS-Kig.ss}; see \cite[Section 5.2]{KS.gc} for the details of this construction (see also Proposition \ref{prop.pEMss} below).
	The resulting $p$-energy measures turn out to satisfy \eqref{Kigsspem.defn} and therefore coincide with the ones $\{\breve{\Gamma}_{p}^{\mathrm{KS}}\langle u \rangle\}_{u \in \mathcal{W}^{p}}$ constructed in Theorem \ref{thm.Kigsspem} (see \cite[Proposition 5.12]{KS.gc}). 
\end{rmk}

\section{$p$-Resistance forms on p.-c.f. self-similar structures}\label{sec.CGQ}
\setcounter{equation}{0}
In this section, we verify \ref{KSwm} for a family of kernels $\bm{k}$ corresponding to the $(1,p)$-Korevaar--Schoen--Sobolev space under the assumption of the existence of a good $p$-resistance form on a post-critically finite self-similar structure.
(See \cite[Theorem 4.2]{CGQ22} or \cite[Section 8.3]{KS.gc} for the existing construction of self-similar $p$-resistance forms in this setting.)

\subsection{Geometry under the $p$-resistance metric}\label{sec.ss}
We first present the setting of this section.
Throughout this section, we presume the following assumption. 
\begin{assum}\label{assum.pRes}
	Let $p \in (1,\infty)$ and $\mathcal{L} = (K,S,\{ F_{i} \}_{i \in S})$ be a p.-c.f. self-similar structure with $\#S \ge 2$ and $K$ connected. 
	Let $(\mathcal{E}_{p},\mathcal{F}_{p})$ be a self-similar $p$-resistance form on $\mathcal{L}$ with weight $(\rweight_{p,i})_{i \in S} \in (0,\infty)^{S}$ such that 
	\begin{equation}\label{e:cond.R}
		\min_{i \in S}\rweight_{p,i} > 1. 
	\end{equation}
	Let $\phdim \in (0,\infty)$ be such that $\sum_{i \in S}\rweight_{p,i}^{-\phdim/(p - 1)} = 1$, and let $m$ be the self-similar measure on $\mathcal{L}$ with weight $\bigl(\rweight_{p,i}^{-\phdim/(p - 1)})_{i \in S}$. 
\end{assum}

\begin{rmk}\label{rmk.CGQ-Kig}
    \begin{enumerate}[label=\textup{(\arabic*)},align=left,leftmargin=*,topsep=2pt,parsep=0pt,itemsep=2pt]
        \item The condition \eqref{e:cond.R} corresponds to the condition (R) in \cite[p. 18]{CGQ22}.
        \item Assumption \ref{assum.pRes} is equivalent to the existence of a \emph{$p$-eigenform} on $V_{0}$ with respect to the renormalization operator with weight $(\rweight_{p,i})_{i \in S} \in (1,\infty)^{S}$, i.e., a $p$-resistance form $\mathcal{E}_{p}^{(0)}$ on $V_{0}$ such that 
        \[
        \inf\Biggl\{ \sum_{i \in S}\rweight_{p,i}\mathcal{E}_{p}^{(0)}(v \circ F_{i}) \Biggm| v \in \mathbb{R}^{V_{1}}, v|_{V_{0}} = u \Biggr\} = \mathcal{E}_{p}^{(0)}(u) \quad \text{for any $u \in \mathbb{R}^{V_{0}}$;}
        \]
        see \cite[Proposition 6.19 and Theorem 8.42]{KS.gc} for a detailed proof of this equivalence.
		In the case $p = 2$, this is nothing but the existence of a regular harmonic structure on $\mathcal{L}$ as defined in \cite[Definition 3.1.2]{Kig01}.
        \item\label{it.difference-CGQ-Kig} Any self-similar $p$-resistance form constructed in \cite[Theorem 4.6]{Kig23} must satisfy $\rweight_{p,i} = \sigma_{p}^{n_{i}}$ for some $n_{i} \in \mathbb{N}$, where $\sigma_{p}$ is the constant in \eqref{pCH.1}.
        This restriction excludes the self-similar $p$-resistance forms with weight $(\rweight_{p,i})_{i \in S} \in (1,\infty)^{S}$ satisfying $(\log \rweight_{p,i})/\log \rweight_{p,j}\not\in\mathbb{Q}$ for some $i,j\in S$, whereas they are covered by \cite{CGQ22};
		as proved in \cite[Proposition B.2]{KS.gc}, they do exist abundantly on plenty of typical affine nested fractals. 
        \item It is easy to see that $\phdim \ge 1$ by using \eqref{e:cond.R} and \eqref{Rp-ss} below.  
    \end{enumerate}
\end{rmk}

\begin{figure}[tb]\centering
   \includegraphics[height=110pt]{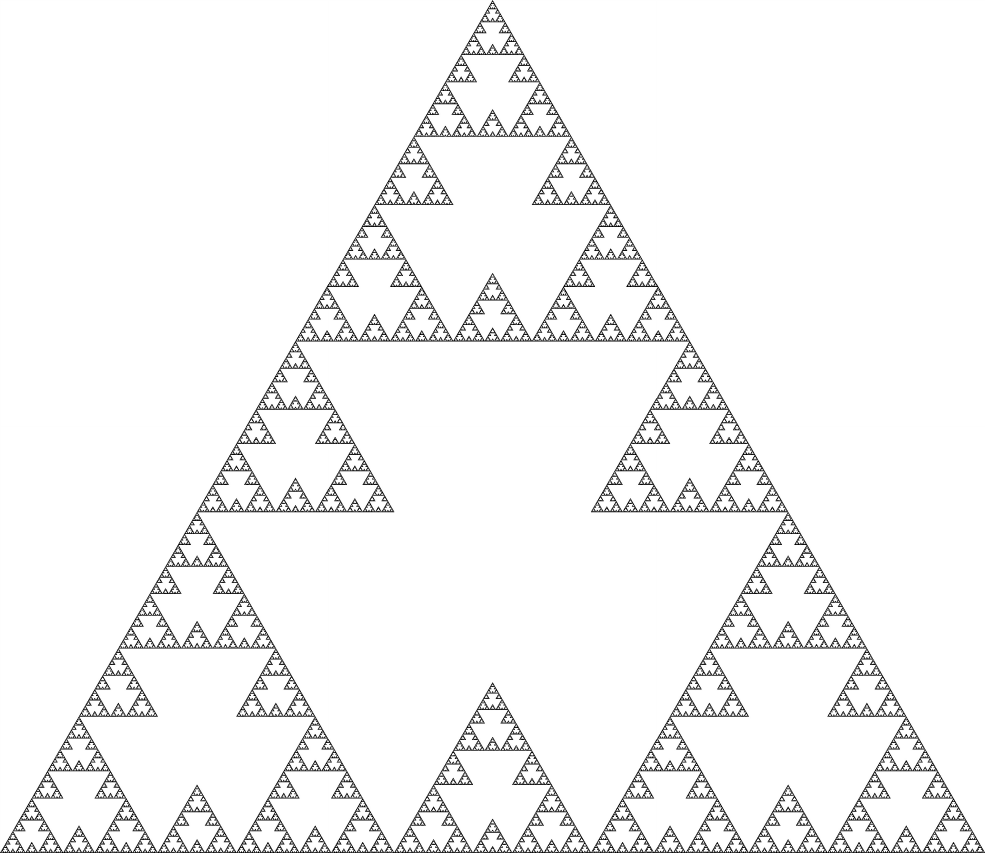}\hspace*{40pt}
   \includegraphics[height=110pt]{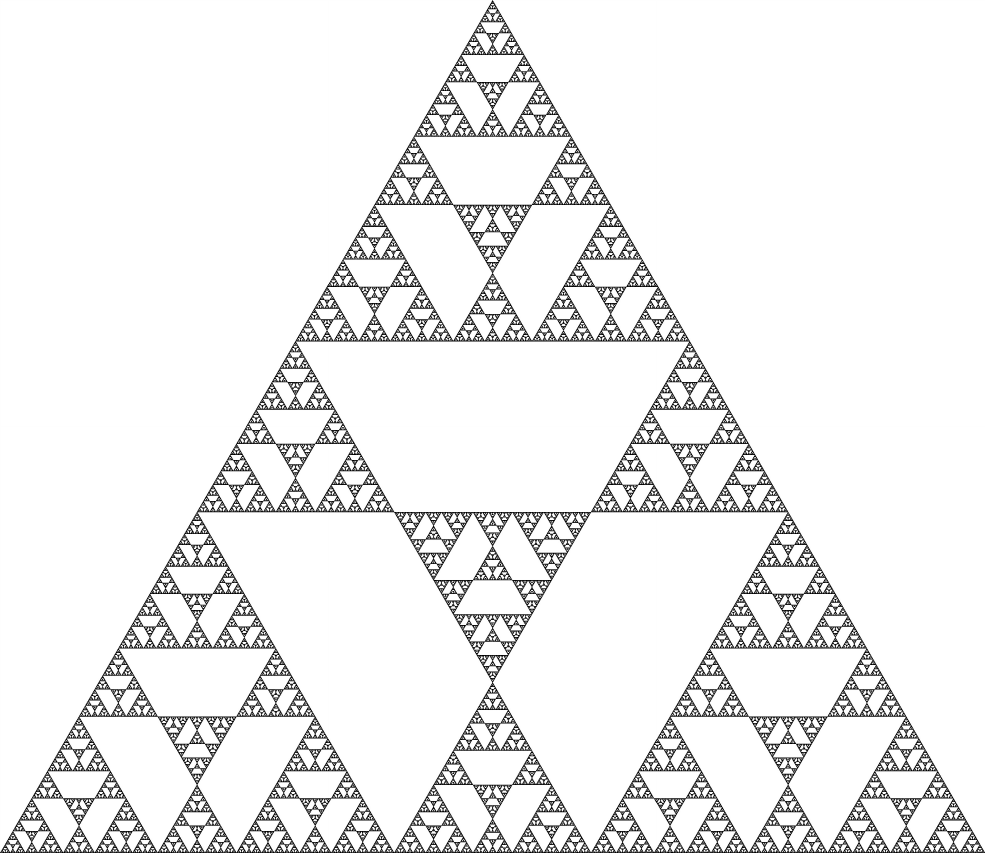}
   \caption{Some examples of affine nested fractals}\label{fig.AFN}
\end{figure}



In this subsection, we will show the Ahlfors regularity of $m$, the capacity upper bound and the Poincar\'{e} inequality in terms of the \emph{$p$-resistance metric} of $(\mathcal{E}_{p},\mathcal{F}_{p})$, which is defined as follows.

\begin{defn}[$p$-Resistance metric; {\cite[Definition 6.33]{KS.gc}}]
	We define the \emph{$p$-resistance metric} $\pmetric_{p,\mathcal{E}_{p}} \colon K \times K \to [0,\infty)$ of $(\mathcal{E}_{p},\mathcal{F}_{p})$ by 
	\begin{equation}\label{pmetric}
    	\pmetric_{p,\mathcal{E}_{p}}(x,y) \coloneqq R_{\mathcal{E}_{p}}(x,y)^{\frac{1}{p - 1}}, \quad x,y \in K
	\end{equation}
	(recall \eqref{R-def}). For simplicity, we write $\pmetric_{p} \coloneqq \pmetric_{p,\mathcal{E}_{p}}$. 
\end{defn}
  
We record some properties of $R_{\mathcal{E}_{p}}$ and $\pmetric_{p}$. 
\begin{prop}[{\cite[Proposition 7.2 and Corollary 6.32]{KS.gc}}]\label{prop.pmetric}
	\begin{enumerate}[label=\textup{(\arabic*)},align=left,leftmargin=*,topsep=2pt,parsep=0pt,itemsep=2pt]
		\item For any $w \in W_{\ast}$ and any $x,y \in K$,
    	\begin{equation}\label{Rp-ss}
        	R_{\mathcal{E}_{p}}(F_{w}(x),F_{w}(y)) \le \rweight_{p,w}^{-1}R_{\mathcal{E}_{p}}(x,y). 
    	\end{equation}
    	\item $\pmetric_{p}$ is a metric on $K$ giving the original topology of $K$.
    In particular, $\closure{V_{\ast}}^{\pmetric_{p}} = K$.
    	\item\label{it.pmetric-Hoelder} For any $u \in \mathcal{F}_{p}$ and $x,y \in K$, 
		\[
		\abs{u(x) - u(y)}^{p} \le R_{\mathcal{E}_{p}}(x,y)\mathcal{E}_{p}(u). 
		\]
		In particular, $\mathcal{F}_{p} \subseteq \contfunc(K)$.
	\end{enumerate} 
\end{prop}


In the next definition, we introduce the \emph{symmetry} on $K$ with respect to $(\mathcal{E}_{p},\mathcal{F}_{p})$. 
\begin{defn}
	We define 
	\begin{equation}\label{e:defn-symmetry}
		\mathcal{G} \coloneqq \biggl\{ T \biggm| 
		\begin{minipage}{229pt}
			$T \colon K \to K$, $T$ is a homeomorphism preserving $m$, and
			$u \circ T, u \circ T^{-1} \in \mathcal{F}_{p}$ and $\mathcal{E}_{p}(u \circ T) = \mathcal{E}_{p}(u)$ for any $u \in \mathcal{F}_{p}$
		\end{minipage}
		\biggr\},
	\end{equation}
	which forms a subgroup of the group of surjective isometries of $(K,\pmetric_{p})$ by \eqref{R-def} and \eqref{pmetric}.
\end{defn}

Let us introduce natural \emph{scales} $\{ \Lambda_{s} \}_{s \in (0,1]}$ with respect to $\pmetric_{p}$. 
(See \cite[Definitions 6.12 and 6.13]{KM23} for the case $p = 2$.)
\begin{defn}
    \begin{enumerate}[label=\textup{(\arabic*)},align=left,leftmargin=*,topsep=2pt,parsep=0pt,itemsep=2pt]
        \item We define $\Lambda_{1} \coloneqq \{ \emptyset \}$,
        \begin{equation*}
            \Lambda_{s} \coloneqq \Bigl\{ w \Bigm| w = w_{1} \dots w_{n} \in W_{\ast} \setminus \{ \emptyset \}, (\rweight_{p,w_{1} \dots w_{n - 1}})^{-1/(p - 1)} > s \ge \rweight_{p,w}^{-1/(p - 1)} \Bigr\}
        \end{equation*}
        for each $s \in (0,1)$.
        (Note that $\{ \Lambda_{s} \}_{s \in (0,1]}$ is the scale associated with the weight function $g(w) \coloneqq \rweight_{p,w}^{-1/(p - 1)}$; see \cite[Definition 2.3.1]{Kig20}.)
        \item For each $(s,x) \in (0,1] \times K$, we define $\Lambda_{s,0}(x) \coloneqq \{ w \in \Lambda_{s} \mid x \in K_{w} \}$, $U_{0}(x,s) \coloneqq \bigcup_{w \in \Lambda_{s,0}(x)}K_{w}$, $\Lambda_{s,1}(x) \coloneqq \{ w \in \Lambda_{s} \mid K_{w} \cap U_{0}(x,s) \neq \emptyset \}$ and $U_{1}(x,s) \coloneqq \bigcup_{w \in \Lambda_{s,1}(x)}K_{w}$.
    \end{enumerate}
\end{defn}

Similar to the case $p = 2$ in \cite[Section 6.1]{KM23}, it is easy to see that $\lim_{s \downarrow 0}\min\{ \abs{w} \mid w \in \Lambda_{s} \} = \infty$, that $\Lambda_{s}$ is a partition of $\Sigma$ for any $s \in (0,1]$, and that $\Lambda_{s_{1}} \le \Lambda_{s_{2}}$ for any $s_1, s_2 \in (0,1]$ with $s_{1} \le s_{2}$.
By \cite[Proposition 2.3.7]{Kig20}, for any $x \in K$, each of $\{ U_{0}(x,s) \}_{s \in (0,1]}$ and $\{ U_{1}(x,s) \}_{s \in (0,1]}$ is non-decreasing in $s$ and forms a fundamental system of neighborhoods of $x$ in $K$.
Moreover, $\{ U_{1}(x,s) \}_{s \in (0,1]}$ can be used as a replacement for the metric balls $\{ B_{\pmetric_{p}}(x,s) \}_{(x,s) \in K \times (0,\diam(K,\pmetric_{p})]}$ in $(K,\pmetric_{p})$ by virtue of the following lemma, which was obtained in \cite[Lemma 6.14]{KM23} in the case $p = 2$.
\begin{lem}\label{lem.Rp-adapted}
    There exist $\alpha_{1},\alpha_{2} \in (0,\infty)$ such that for any $(s,x) \in (0,1] \times K$,
    \begin{equation}\label{Rp-adapted}
        B_{\pmetric_{p}}(x,\alpha_{1}s) \subseteq U_{1}(x,s) \subseteq B_{\pmetric_{p}}(x,\alpha_{2}s). 
    \end{equation}
\end{lem}
\begin{proof}
    By \eqref{form.ss}, we have $\diam(K_{w},\pmetric_{p}) \le \rweight_{p,w}^{-1/(p - 1)}\diam(K,\pmetric_{p})$ for any $w \in W_{\ast}$, which implies the latter inclusion in \eqref{Rp-adapted} with $\alpha_{2} \in (2\diam(K,\pmetric_{p}),\infty)$ arbitrary.
    (In particular, $\diam(K_{w},\pmetric_{p}) < \alpha_{2}s$ for any $w \in \Lambda_{s}$.)
    We will show the former inclusion in \eqref{Rp-adapted} in the rest of this proof.
    To this end, it suffices to prove that there exists $\alpha_{1} \in (0,\infty)$ such that $\pmetric_{p}(x,y) \ge \alpha_{1}s$ for any $s \in(0,1]$, any $w,v \in \Lambda_{s}$ with $K_{w} \cap K_{v} = \emptyset$ and any $(x,y) \in K_{w} \times K_{v}$.
    Let $\psi_{q} \coloneqq h_{V_{0}}^{\mathcal{E}_{p}}\bigl[\indicator{q}^{V_{0}}\bigr]$ for any $q \in V_{0}$, where $h_{V_{0}}^{\mathcal{E}_{p}}$ denotes the $\mathcal{E}_{p}$-harmonic extension operator from $V_{0}$, that is, $\psi_{q}$ is the unique function in $\mathcal{F}_{p}$ such that $\psi_{q}|_{V_{0}} = \indicator{q}^{V_{0}}$ and $\mathcal{E}_{p}(\psi_{q}) = \min\bigl\{ \mathcal{E}_{p}(v) \bigm| v \in \mathcal{F}_{p}, v|_{V_{0}} = \indicator{q}^{V_{0}} \bigr\}$ (see \cite[Theorem 6.13]{KS.gc}).
    Fix $w \in \Lambda_{s}$ and let $u_{w} \in \contfunc(K)$ be such that, for $\tau \in \Lambda_{s}$,
    \begin{align}\label{d:uw}
        u_{w} \circ F_{\tau} =
        \begin{cases}
            1 \quad &\text{if $\tau = w$,} \\
            \sum_{q \in V_{0}; F_{\tau}(q) \in F_{w}(V_{0})}\psi_{q} \quad &\text{if $\tau \neq w$ and $K_{\tau} \cap K_{w} \neq \emptyset$,} \\
            0 \quad &\text{if $K_{\tau} \cap K_{w} = \emptyset$.}
        \end{cases}
    \end{align}
    Since $\Lambda_{s}$ is a partition of $\Sigma$, we have $u_{w} \in \mathcal{F}_{p}$ by \eqref{domain.ss}, and
    \begin{align}\label{uw.ss}
        \mathcal{E}_{p}(u_{w})
        &= \sum_{\tau \in \Lambda_{s}}\rweight_{p,\tau}\mathcal{E}_{p}(u_{w} \circ F_{\tau}) \nonumber \\
        &= \sum_{\tau \in \Lambda_{s} \setminus \{ w \}; K_{\tau} \cap K_{w} \neq \emptyset}\rweight_{p,\tau}\mathcal{E}_{p}\left(\sum_{q \in V_{0}; F_{\tau}(q) \in F_{w}(V_{0})}\psi_{q}\right)
    \end{align}
    by \eqref{form.ss}.
    Set $\overline{\rweight}_{p} \coloneqq \max_{i \in S}\rweight_{p,i} \in (1,\infty)$ and $c_{1} \coloneqq \max_{q \in V_{0}}\mathcal{E}_{p}(\psi_{q}) \in (0,\infty)$.
    Then $\rweight_{p,\tau}^{-1} \ge (\overline{\rweight}_{p})^{-1}s^{p - 1}$ for any $\tau \in \Lambda_{s}$.
    Since $\#\{ \tau \in \Lambda_{s} \mid K_{\tau} \cap K_{w} \neq \emptyset \} \le (\#\mathcal{C}_{\mathcal{L}})(\#V_{0})$ by \cite[Lemma 4.2.3]{Kig01}, \eqref{uw.ss} together with H\"{o}lder's inequality implies that
    \begin{equation}\label{uw.upper}
        \mathcal{E}_{p}(u_{w}) \le (\#\mathcal{C}_{\mathcal{L}})(\#V_{0})\overline{\rweight}_{p}s^{-p + 1}(\#V_{0})^{p - 1}c_{1} \eqqcolon (\alpha_{1}s)^{-(p - 1)}.
    \end{equation}
    For any $v \in \Lambda_{s}$ with $K_{w} \cap K_{v} = \emptyset$ and any $(x,y) \in K_{w} \times K_{v}$, we clearly have $u_{w}(x) = 1$ and $u_{w}(y) = 0$.
    Hence
    \[
    \pmetric_{p}(x,y) \ge \mathcal{E}_{p}(u)^{-1/(p - 1)} \ge \alpha_{1}s,
    \]
    which proves the desired result.
\end{proof}

Now we can show that $m$ is $\phdim$-Ahlfors regular (see \cite[Lemma 6.8]{KM23} for the case $p = 2$). 
\begin{lem}\label{lem.Rp-AR}
    There exist $c_{1},c_{2} \in (0,\infty)$ such that for any $x \in K$ and any $s \in (0,2\diam(K,\pmetric_{p})]$, 
    \begin{equation}\label{Rp-AR}
        c_{1}s^{\phdim} \le m(B_{\pmetric_{p}}(x,s)) \le c_{2}s^{\phdim}.
    \end{equation}
\end{lem}
\begin{proof}
    This is immediate from \eqref{Rp-adapted}, $\#\{ \tau \in \Lambda_{s} \mid K_{\tau} \cap K_{w} \neq \emptyset \} \le (\#\mathcal{C}_{\mathcal{L}})(\#V_{0})$ (see \cite[Lemma 4.2.3]{Kig01}) and $m(K_{w}) = \rweight_{p,w}^{-1/(p - 1)}$ (see \cite[Corollary 1.4.8]{Kig01}).
\end{proof}

The proof of Lemma \ref{lem.Rp-adapted} includes the following capacity upper bound in terms of the $p$-resistance metric $\pmetric_{p}$.
\begin{prop}\label{prop.Rp-capu}
    There exists $C \in (0,\infty)$ such that for any $x \in K$ and any $s \in (0,2\diam(K,\pmetric_{p})]$, 
    \begin{equation}\label{Rp-capu}
        \inf\bigl\{ \mathcal{E}_{p}(u) \bigm| u \in \mathcal{F}_{p}, u|_{B_{\pmetric_{p}}(x,\alpha_{1}s)} = 1, \supp_{K}[u] \subseteq B_{\pmetric_{p}}(x,2\alpha_{2}s) \bigr\} \le Cs^{-(p - 1)},
    \end{equation}
    where $\alpha_{1},\alpha_{2}$ are the constants in \eqref{Rp-adapted}.
\end{prop}
\begin{proof}
    Let $u_{w} \in \mathcal{F}_{p}$ be the same function as in the proof of Lemma \ref{lem.Rp-adapted} for each $w \in \Lambda_{s}$.
    Then $\varphi \coloneqq \max_{w \in \Lambda_{s,1}(x)}u_{w}$ satisfies $\varphi|_{U_{1}(x,s)} = 1$.
    Since $\diam(K_{w},\pmetric_{p}) < \alpha_{2}s$, we see from \eqref{Rp-adapted} that $\supp_{K}[\varphi] \subseteq B_{\pmetric_{p}}(x,2\alpha_{2}s)$.
    By \eqref{sadd} for $(\mathcal{E}_{p},\mathcal{F}_{p})$, \eqref{uw.upper} and \cite[Lemma 4.2.3]{Kig01}, we have $\varphi \in \mathcal{F}_{p}$ and
    \[
    \mathcal{E}_{p}(\varphi) \le \sum_{w \in \Lambda_{s,1}(x)}\mathcal{E}_{p}(u_{w}) \le (\alpha_{1}s)^{-(p - 1)}(\#\mathcal{C}_{\mathcal{L}})(\#V_{0}) \eqqcolon Cs^{-(p - 1)}. 
    \qedhere\]
\end{proof}

Similar to Lemma \ref{lem.unity-Kig} and Corollary \ref{cor.reg-Kig}, we can easily show the next lemma as a consequence of \eqref{Rp-capu}, and obtain the regularity of $\mathcal{F}_{p}$.
\begin{lem}\label{lem.Rp-unity}
    Let $\varepsilon \in (0, 1)$ and let $V$ be a maximal $\varepsilon$-net of $(K, \pmetric_{p})$.
    Then there exists a family of functions $\{ \psi_z \}_{z \in V}$ that satisfies the following properties:
    \begin{enumerate}[label=\textup{(\roman*)},align=left,leftmargin=*,topsep=2pt,parsep=0pt,itemsep=2pt]
        \item $\sum_{z \in V}\psi_z \equiv 1$.
        \item $\psi_z \in \mathcal{F}_{p}$, $0 \le \psi_z \le 1$, $\psi_{z}|_{B_{\pmetric_{p}}(z, \varepsilon/4)} \equiv 1$ and $\supp_{K}[\psi_z] \subseteq B_{\pmetric_{p}}(z, 5\varepsilon/4)$ for any $z \in V$;
        \item If $z \in V$ and $z' \in V \setminus \{ z \}$, then $\psi_{z'}|_{B_{\pmetric_{p}}(z, \varepsilon/4)} \equiv 0$.
        \item There exists $C \in (0,\infty)$ such that $\mathcal{E}_{p}(\psi_z) \le C\varepsilon^{-(p - 1)}$ for any $z \in V$.
    \end{enumerate}
\end{lem}

\begin{cor}\label{cor.regular}
    $(\mathcal{E}_{p},\mathcal{F}_{p})$ is regular, i.e., $\mathcal{F}_{p}$ is dense in $(\contfunc(K),\norm{\,\cdot\,}_{\sup})$.
\end{cor}

Next, in order to state a Poincar\'{e}-type inequality in this context, we introduce the associated self-similar $p$-energy measures in Proposition \ref{prop.pEMss} and a localized version of $\mathcal{F}_{p}$ in Definition \ref{defn.local}. 
Thanks to \eqref{form.ss}, we can define the $p$-energy measures associated with $(\mathcal{E}_{p},\mathcal{F}_{p})$ by using Kolmogorov's extension theorem.
We recall fundamental results on the $p$-energy measures constructed in this way in the following proposition.
See \cite[Section 9]{MS+} and \cite[Section 5.2]{KS.gc} for further details and properties of them.
\begin{prop}[Self-similar $p$-energy measures]\label{prop.pEMss}
	For each $u \in \mathcal{F}_{p}$, there exists a unique positive Radon measure $\Gamma_{\mathcal{E}_{p}}\langle u \rangle$ on $K$ satisfying
    	\begin{equation}\label{eq.functional}
    	\int_{K}\varphi\,d\Gamma_{\mathcal{E}_{p}}\langle u \rangle 
    	= \mathcal{E}_{p}(u; u\varphi) - \left(\frac{p - 1}{p}\right)^{p - 1}\mathcal{E}_{p}\bigl(\abs{u}^{\frac{p}{p - 1}}; \varphi\bigr) \quad \text{for any $\varphi \in \mathcal{F}_{p}$}.
    	\end{equation}
	Moreover, the following hold:
    \begin{enumerate}[label=\textup{(\roman*)},align=left,leftmargin=*,topsep=2pt,parsep=0pt,itemsep=2pt]
        \item \label{em.eachcell} $\Gamma_{\mathcal{E}_{p}}\langle u \rangle(K_{w}) = \rweight_{p,w}\mathcal{E}_{p}(u \circ F_{w})$ for any $u \in \mathcal{F}_{p}$ and any $w \in W_{\ast}$.
        \item \label{em.seminorm} $\Gamma_{\mathcal{E}_{p}}\langle \,\cdot\, \rangle(A)^{1/p}$ is a seminorm on $\mathcal{F}_{p}$ for any $A \in \mathcal{B}(K)$.
        \item \label{em.nomeas} $\Gamma_{\mathcal{E}_{p}}\langle u \rangle(K_{w} \cap K_{\tau}) = 0$ for any $u \in \mathcal{F}_{p}$ and any $w,\tau \in W_{\ast}$ with $\Sigma_{w} \cap \Sigma_{\tau} = \emptyset$.
        \item \label{em.sl} $\Gamma_{\mathcal{E}_{p}}\langle u \rangle(A) = \Gamma_{\mathcal{E}_{p}}\langle v \rangle(A)$ for any $u,v \in \mathcal{F}_{p}$ and any $A \in \mathcal{B}(K)$ with $(u - v)|_{A} \in \mathbb{R}\indicator{A}$.
    \end{enumerate}
\end{prop}
\begin{proof}
    For the construction of a candidate for $\Gamma_{\mathcal{E}_{p}}\langle u \rangle$, see \cite[Section 9]{MS+} or \cite[Section 5.2]{KS.gc}.
    Then the properties \ref{em.seminorm}, \ref{em.nomeas} and \ref{em.sl} follow from \cite[Proposition 9.3, Corollaries 9.8 and 9.9]{MS+} since $\#(K_{w} \cap K_{\tau}) < \infty$ by $\#V_{0} < \infty$ and \cite[Proposition 1.3.5-(2)]{Kig01}.
    We obtain \ref{em.eachcell} by combining \ref{em.nomeas} and \cite[Proposition 9.4]{MS+}. 
    The equality \eqref{eq.functional} is proved in \cite[Proposition 5.12]{KS.gc}, and the uniqueness of a positive Radon measure on $K$ satisfying \eqref{eq.functional} follows from Corollary \ref{cor.regular} and the uniqueness part of the Riesz--Markov--Kakutani representation theorem (see, e.g., \cite[Theorems 2.14 and 2.18]{Rud}).
\end{proof}


\begin{defn}\label{defn.local}
    Let $U$ be a non-empty open subset of $K$.
    We define a linear subspace $\mathcal{F}_{p,\mathrm{loc}}(U)$ of $\contfunc(U)$ by
    \begin{equation}\label{d:Floc}
        \mathcal{F}_{p,\mathrm{loc}}(U) \coloneqq
        \biggl\{ f \in \contfunc(U) \biggm|
            \begin{minipage}{150pt}
			$f|_{A} = f^{\#}|_{A}$ for some $f^{\#} \in \mathcal{F}$
			for each relatively compact open subset $A$ of $U$
        \end{minipage}
        \biggr\}.
    \end{equation}
	For each $f \in \mathcal{F}_{p,\mathrm{loc}}(U)$, we further define a positive Radon measure $\Gamma_{\mathcal{E}_{p}}\langle f \rangle$ on $U$ as follows.
	We first define $\Gamma_{\mathcal{E}_{p}}\langle f \rangle(E) \coloneqq \Gamma_{\mathcal{E}_{p}}\langle f^{\#} \rangle(E)$ for each relatively compact Borel subset $E$ of $U$, with $A \subseteq U$ and $f^{\#} \in \mathcal{F}_{p}$ as in \eqref{d:Floc} chosen so that $E \subseteq A$;
	this definition of $\Gamma_{\mathcal{E}_{p}}\langle f \rangle(E)$ is independent of a particular choice of such $A$ and $f^{\#}$ by Proposition \ref{prop.pEMss}-\ref{em.sl}.
	We then define $\Gamma_{\mathcal{E}_{p}}\langle f \rangle(E) \coloneqq \lim_{n \to \infty}\Gamma_{\mathcal{E}_{p}}\langle f \rangle(E \cap A_{n})$ for each $E \in \mathcal{B}(U)$, where $\{ A_{n} \}_{n \in \mathbb{N}}$ is a non-decreasing sequence of relatively compact open subsets of $U$ such that $\bigcup_{n \in \mathbb{N}}A_{n} = U$; it is clear that this definition of $\Gamma_{\mathcal{E}_{p}}\langle f \rangle(E)$ is independent of a particular choice of $\{ A_{n} \}_{n \in \mathbb{N}}$, coincides with the previous one when $E$ is relatively compact in $U$, and gives a Radon measure on $U$.
\end{defn}

Now we can prove a Poincar\'{e}-type inequality in terms of the $p$-resistance metric.%
\begin{prop}[$(p,p)$-Poincar\'{e} inequality]\label{prop.Rp-PI}
    There exist $C,A \in (0,\infty)$ with $A \geq 1$ such that for any $(x,s) \in K \times (0,\diam(K,\pmetric_{p})]$ and any $u \in \mathcal{F}_{p,\mathrm{loc}}(B_{\pmetric_{p}}(x,As))$, 
    \begin{equation}\label{Rp-PI}
        \int_{B_{\pmetric_{p}}(x,s)}\abs{u(y) - u_{B_{\pmetric_{p}}(x,s)}}^{p}\,m(dy) \le Cs^{\phdim + p - 1}\Gamma_{\mathcal{E}_{p}}\langle u \rangle(B_{\pmetric_{p}}(x,As)).
    \end{equation}
\end{prop}
\begin{proof}
    For simplicity, we consider the case $u \in \mathcal{F}_{p}$.
    Note that, since $m(K_{v} \cap K_{v'}) = 0$ for any $v,v' \in W_{\ast}$ with $\Sigma_{v} \cap \Sigma_{v'} = \emptyset$ (see \cite[Corollary 1.4.8]{Kig01}), 
    \begin{equation*}
        \int_{B_{\pmetric_{p}}(x,\alpha_{1}s)}\abs{u - u_{B_{\pmetric_{p}}(x,\alpha_{1}s)}}^{p}\,dm 
        \le \sum_{w \in \Lambda_{s,1}(x)}\int_{K_{w}}\abs{u - u_{U_{1}(x,s)}}^{p}\,dm. 
    \end{equation*}
    Let $w \in \Lambda_{s,1}(x)$.
    For any $(y,z) \in K_{w} \times U_{1}(x,s)$, there exist $v^{1},v^{2},v^{3} \in \Lambda_{s,1}(x)$ such that $v^{1} = w$, $z \in K_{v^{3}}$ and $K_{v^{i}} \cap K_{v^{i + 1}} \neq \emptyset$ for each $i \in \{ 1,2 \}$.
    Let us fix $x_{i} \in K_{v^{i}} \cap K_{v^{i + 1}}$ and $q_{i} \in V_{0}$ so that $x_{i} = F_{v^{i}}(q_{i})$.
    Then 
    \begin{align*}
        \abs{u(y) - u(z)}^{p}
        &\le 3^{p - 1}\Bigl(\abs{u(y) - u(x^{1})}^{p} + \abs{u(x^{1}) - u(x^{2})}^{p} + \abs{u(x^{2}) - u(z)}^{p}\Bigr) \\
        &\le \bigl(3\diam(K,\pmetric_{p})\bigr)^{p - 1}\sum_{i = 1}^{3}\rweight_{p,v^{i}}^{-1}\Gamma_{\mathcal{E}_{p}}\langle u \rangle(K_{v^{i}}) \\
        &\le Cs^{p - 1}\Gamma_{\mathcal{E}_{p}}\langle u \rangle\left(\bigcup_{i = 1}^{3}K_{v^{i}}\right)
        \le Cs^{p - 1}\Gamma_{\mathcal{E}_{p}}\langle u \rangle(B_{\pmetric_{p}}(x,\alpha_{2}s)).
    \end{align*}
    Therefore, noting that $m(K_{w}) \lesssim s^{\phdim}$ by \eqref{Rp-adapted} and \eqref{Rp-AR}, we have
    \begin{align*}
        \int_{K_{w}}\abs{u(y) - u_{U_{1}(x,s)}}^{p}\,m(dy)
        &\le \int_{K_{w}}\fint_{U_{1}(x,s)}\abs{u(y) - u(z)}^{p}\,m(dx)m(dy) \\
        &\lesssim s^{\phdim + p - 1}\Gamma_{\mathcal{E}_{p}}\langle u \rangle(B_{\pmetric_{p}}(x,\alpha_{2}s)),  
    \end{align*}
    which together with $\sup_{(x,s) \in K \times (0,1]}\#\Lambda_{s,1}(x) < \infty$ (see \cite[Lemma 4.2.3]{Kig01}) yields \eqref{Rp-PI}.
\end{proof}

\subsection{Estimates on self-similar $p$-energy measures and weak monotonicity}
In this subsection, we show localized energy estimates on Korevaar--Schoen $p$-energy forms in terms of their associated self-similar $p$-energy measures and verify \ref{KSwm}.
We continue to follow the setting in the previous subsection, i.e., we suppose that Assumption \ref{assum.pRes} holds.
We consider $\mathcal{E}_{p}$ as a $[0,\infty]$-valued functional defined on $L^{p}(K,m)$ by setting $\mathcal{E}_{p}(f) \coloneqq \infty$ for $f \in L^{p}(K,m) \setminus \mathcal{F}_{p}$.

Similar arguments as in Propositions \ref{prop.upper} and \ref{prop.lower} yield an upper bound on localized Korevaar--Schoen energy functionals in Proposition \ref{prop.upper-Rp} and a lower bound on them in Proposition \ref{prop.lower-Rp} below.
\begin{prop}\label{prop.upper-Rp}
    There exists $C \in (0,\infty)$ such that for any $E \in \mathcal{B}(K)$, any open neighborhood $E'$ of $\overline{E}^{K}$ and any $u \in \mathcal{F}_{p,\mathrm{loc}}(E')$, 
	\begin{equation}\label{KSu.local-Rp}
	    \limsup_{s \downarrow 0}\int_{E}\fint_{B_{\pmetric_{p}}(x, s)}\frac{\abs{u(x) - u(y)}^{p}}{s^{\phdim + p - 1}}\,\measure(dy)\measure(dx)
        \le C\Gamma_{\mathcal{E}_{p}}\langle u \rangle\bigl(\overline{E}^{K}\bigr).
	\end{equation}
	Moreover, with $C \in (0,\infty)$ the same as in \eqref{KSu.local-Rp}, for any $f \in L^{p}(K,m)$, 
	\begin{equation}\label{KSu.global-Rp}
	\sup_{s > 0}\int_{K}\fint_{B_{\pmetric_{p}}(x, s)}\frac{\abs{f(x) - f(y)}^{p}}{s^{\phdim + p - 1}}\,\measure(dy)\measure(dx)
        \le C\mathcal{E}_{p}(f).
	\end{equation}
\end{prop}
\begin{proof}
    Let $V$ be a relatively compact open subset of $E'$ with $V \supseteq \closure{E}^{K}$ and let $u^{\#} \in \mathcal{F}_{p}$ satisfy $u^{\#} = u$ $m$-a.e.\ on $V$. 
    Similar to \cite[(7.2)]{MS+}, by using \eqref{Rp-AR} and \eqref{Rp-PI}, we easily see that for any $s \in (0,\infty)$, 
    \begin{equation}\label{preKS-local-Rp}
        \int_{E}\fint_{B_{\pmetric_{p}}(x, s)}\frac{\abs{u^{\#}(x) - u^{\#}(y)}^{p}}{s^{\phdim + p - 1}}\,\measure(dy)\measure(dx)
        \le C\Gamma_{\mathcal{E}_{p}}\langle u^{\#} \rangle\bigl((E)_{\pmetric_{p},2As}\bigr),
    \end{equation}
    where $A \in [1,\infty)$ is the constant in \eqref{Rp-PI} and $C \in (0,\infty)$ is independent of $x$, $s$ and $f$.
    We get \eqref{KSu.local-Rp} by letting $s \downarrow 0$ since $\Gamma_{\mathcal{E}_{p}}\langle u^{\#} \rangle\bigl((E)_{\pmetric_{p},2As}\bigr) = \Gamma_{\mathcal{E}_{p}}\langle u \rangle\bigl((E)_{\pmetric_{p},2As}\bigr)$ for any $s \in (0,\infty)$ with $(E)_{\pmetric_{p},2As} \subseteq V$ by Proposition \ref{prop.pEMss}-\ref{em.sl}.
    The estimate \eqref{KSu.global-Rp} for $f \in \mathcal{F}_{p}$ is easily implied by $\Gamma_{\mathcal{E}_{p}}\langle f \rangle(K) = \mathcal{E}_{p}(f)$ and \eqref{preKS-local-Rp} with $E = K$. For $f \in L^{p}(K,m) \setminus \mathcal{F}_{p}$, \eqref{KSu.global-Rp} is obvious by $\mathcal{E}_{p}(f) = \infty$, so the proof is completed.
\end{proof}

\begin{prop}\label{prop.lower-Rp}
	There exists $C \in (0,\infty)$ such that for any $E \in \mathcal{B}(K)$, any open neighborhood $E'$ of $\closure{E}^{K}$ and any $u \in \mathcal{F}_{p,\mathrm{loc}}(E')$,
	\begin{equation}\label{KSl.local-Rp}
		\Gamma_{\mathcal{E}_{p}}\langle u \rangle(E) \le C\lim_{\delta \downarrow 0}\liminf_{s \downarrow 0}\int_{(E)_{\pmetric_{p},\delta}}\fint_{B_{\pmetric_{p}}(x,s)}\frac{\abs{u(x) - u(y)}^{p}}{s^{\phdim + p - 1}}\,\measure(dy)\measure(dx).
	\end{equation}
	Furthermore, with $C \in (0,\infty)$ the same as in \eqref{KSl.local-Rp}, for any $f \in L^{p}(K, \measure)$,
	\begin{equation}\label{KSl.global-Rp}
		\mathcal{E}_{p}(f) \le C\liminf_{s \downarrow 0}\int_{K}\fint_{B_{\pmetric_{p}}(x,s)}\frac{\abs{f(x) - f(y)}^{p}}{s^{\phdim + p - 1}}\,\measure(dy)\measure(dx).
	\end{equation}
\end{prop}
\begin{proof}
    Let $s \in (0, 1)$ and fix a maximal $r$-net $N_{s}$ of $(K, \pmetric_{p})$.
    Let $\{ \psi_{z,s} \}_{z \in N_{s}}$ be a partition of unity as given in Lemma \ref{lem.Rp-unity} and define $A_{s} \colon L^{p}(K,m) \to \mathcal{F}_{p}$ by $A_{s}f \coloneqq \sum_{z \in N_{s}}f_{B_{\pmetric_{p}}(z, s/4)}\psi_{z,s}$ for $f \in L^{p}(K,m)$.
    Then we can easily see that $\lim_{r \to 0}\norm{A_{r}f - f}_{L^{p}(K,m)} = 0$ and $\sup_{r > 0}\norm{A_{r}}_{L^{p}(K,m) \to L^{p}(K,m)} < \infty$.
    Using Proposition \ref{prop.pEMss}-\ref{em.sl}, we can show that there exists $C_{1} > 0$ that is independent of $x$, $s$ and $f$ such that
    \begin{align}\label{KSl-Rp.1}
        &\Gamma_{\mathcal{E}_{p}}\langle A_{s}f \rangle\bigl(B_{\pmetric_{p}}(z,5s/4)\bigr) \nonumber \\
        &\le C_{1}\sum_{w \in N_{s} \cap B_{\pmetric_{p}}(z, 11s/4)}\int_{B_{\pmetric_{p}}(w, 3s)}\fint_{B_{\pmetric_{p}}(x, 9s)}\frac{\abs{f(x) - f(y)}^{p}}{s^{\phdim + p - 1}}\,\measure(dy)\measure(dx),
    \end{align}
    for any small enough $s > 0$.
    Let us fix $\delta > 0$ and define $N_{s}(E) \coloneqq \{ z \in N_{s} \mid E \cap B_{\pmetric_{p}}(z,s) \neq \emptyset \}$.
    Since $\bigcup_{z \in N_{s}(E)}\bigcup_{w \in N_{s} \cap B_{\pmetric_{p}}(z,11s/4)}B_{\pmetric_{p}}(w,3s) \subseteq (E)_{\pmetric_{p},\delta}$ for all small enough $s > 0$ and $(K,\pmetric_{p})$ is metric doubling by Lemma \ref{lem.Rp-AR}, we have
    \begin{align}\label{KSl-Rp.2}
        &\Gamma_{\mathcal{E}_{p}}\langle A_{s}f \rangle(E) 
        \le \sum_{z \in N_{s}(E)}\Gamma_{\mathcal{E}_{p}}\langle A_{s}f \rangle\bigl(B_{\pmetric_{p}}(z,5s/4)\bigr) \nonumber \\
        &\quad \overset{\eqref{KSl-Rp.1}}{\le} C\int_{(E)_{\pmetric_{p},\delta}}\fint_{B_{\pmetric_{p}}(x, 9s)}\frac{\abs{f(x) - f(y)}^{p}}{s^{\phdim + p - 1}}\,\measure(dy)\measure(dx), \quad f \in L^{p}(K,m),
    \end{align}
    where $C \in (0,\infty)$ is independent of $x,s$ and $f$.
    Once we get \eqref{KSl-Rp.2}, the argument in the proof of Proposition \ref{prop.lower} with minor modifications proves \eqref{KSl.local-Rp}. 
 	Indeed, for $u \in \mathcal{F}_{p,\mathrm{loc}}(E')$, a relatively compact open subset $V$ of $E'$ with $V \supseteq \closure{E}^{K}$ and $u^{\#} \in \mathcal{F}_{p}$ satisfying $u^{\#} = u$ $m$-a.e.\ on $V$, we have from Proposition \ref{prop.pEMss}-\ref{em.sl} that $\Gamma_{\mathcal{E}_{p}}\langle A_{s}u^{\#} \rangle(E) = \Gamma_{\mathcal{E}_{p}}\langle A_{s}u \rangle(E)$ if $s$ is sufficiently small. Then similar arguments using Mazur's lemma as in the proof of Proposition \ref{prop.lower} implies \eqref{KSl.local-Rp} and \eqref{KSl.global-Rp}. 
\end{proof}

Now we can identify $\mathcal{F}_{p}$ as the $(1,p)$-Korevaar--Schoen--Sobolev space.
\begin{thm}\label{thm.CGQ-KS}
    Let $s_{p}$, $\bm{k} \coloneqq \bm{k}^{s_{p}}$ and $\mathrm{KS}^{1,p}(K,\pmetric_{p},m)$ be as defined in Example \ref{ex.KS} with $\pmetric_{p}$ in place of $d$.   
	Then $s_{p} = (\phdim + p - 1)/p$, $\mathcal{F}_{p} = \mathrm{KS}^{1,p}(K,\pmetric_{p},m)$, and \ref{KSwm} holds.
    Moreover, there exists $C \in [1,\infty)$ such that 
    \begin{equation}\label{KSfull2}
    	C^{-1}\sup_{r > 0}J_{p,r}^{\bm{k}}(f) \le \mathcal{E}_{p}(f) \le C\liminf_{r \downarrow 0}J_{p,r}^{\bm{k}}(f) \quad \text{for any $f \in L^{p}(K,m)$.}
    \end{equation}

\end{thm}
\begin{proof}
    We have $\mathcal{F}_{p} = B_{p,\infty}^{(\phdim + p - 1)/p}$ and \eqref{KSfull2} by \eqref{KSu.global-Rp} and \eqref{KSl.global-Rp}.
    In particular, $s_{p} \ge (\phdim + p - 1)/p$.
    Let $s > (\phdim + p - 1)/p$ and let $f \in \mathcal{F}_{p} \setminus \mathbb{R}\indicator{K}$, which exists by \eqref{Rp-capu}.
    Let $A_{r} \colon L^{p}(K,m) \to \mathcal{F}_{p}$ be the same operator as in the proof of Proposition \ref{prop.lower-Rp} for each $r \in (0,1)$.
    Then, by \eqref{KSl-Rp.2} with $E = K$, for any $r \in(0,1)$ and $f \in L^{p}(K,m)$,
    \begin{equation}\label{diverge-Rp}
        \frac{r^{\phdim + p - 1}}{r^{sp}}\mathcal{E}_{p}(A_{r}f) \le C\int_{K}\fint_{B_{\pmetric_{p}}(x,9r)}\frac{\abs{f(x) - f(y)}^{p}}{r^{sp}}\,m(dy)m(dx),
    \end{equation}
    where $C > 0$ is independent of $f$ and $r$.
    Clearly, $\sup_{r > 0}\mathcal{E}_{p}(A_{r}f) > 0$ and $r^{\phdim + p - 1 - sp} \to \infty$ as $r \downarrow 0$.
    Hence we obtain $s \ge s_{p}$ since $f \not\in B_{p,\infty}^{s}$ by \eqref{diverge-Rp}.
    This implies that $(\phdim + p - 1)/p \ge s_{p}$.
    In particular, we obtain $\mathcal{F}_{p} = \mathrm{KS}^{1,p}(K,\pmetric_{p},m)$.
    Also, \ref{KSwm} follows from \eqref{KSu.global-Rp} and \eqref{KSl.global-Rp}.
\end{proof}

Unfortunately, it is not clear whether Korevaar--Schoen $p$-energy forms $(\mathcal{E}_{p}^{\mathrm{KS}},\mathcal{F}_{p})$ on $(K,\pmetric_{p},m)$, which exist by Theorems \ref{thm.CGQ-KS} and \ref{thm.KS-energy} (recall Example \ref{ex.KS}), are self-similar or not.
However, we can construct a self-similar $p$-resistance form on $\mathcal{L}$ by the same argument as in the proof of Theorem \ref{thm.KSss-Kig}. 
Recall that $\mathcal{F}_{p} \cap \contfunc(K) = \mathcal{F}_{p}$ is dense both in $(\contfunc(K),\norm{\,\cdot\,}_{\sup})$ and in $\mathcal{F}_{p}$ by Proposition \ref{prop.pmetric}-\ref{it.pmetric-Hoelder} and Corollary \ref{cor.regular}.
\begin{thm}\label{thm.KSss-CGQ}
	For each $n \in \mathbb{N}$, define $\bm{k}^{(n)} = \{ k_{r}^{(n)} \}_{r > 0}$ by 
	\begin{equation*}
		k_{r}^{(n)}(x,y) \coloneqq \frac{1}{n + 1}\sum_{l = 0}^{n}\sum_{w \in W_{l}}\rweight_{p,w}^{(2\phdim + p - 1)/(p - 1)}\frac{\indicator{A_{w,r}}(x,y)}{r^{2\phdim + p - 1}}, \quad x,y \in K, 
	\end{equation*}
	where $A_{w,r} \coloneqq \bigl\{ (x,y) \in K_{w} \times K_{w} \bigm| \pmetric_{p}(F_{w}^{-1}(x),F_{w}^{-1}(y)) < r \bigr\}$. 
	Then $\bm{k}^{(n)}$ is asymptotically local, \hyperref[KSwm]{\textup{(WM)$_{p,\bm{k}^{(n)}}$}} holds, $B_{p,\infty}^{\bm{k}^{(n)}} = \mathcal{F}_{p}$, 
	and for any sequence $\{(\mathcal{E}_{p}^{\bm{k}^{(n)}},\mathcal{F}_{p})\}_{n\in\mathbb{N}}$ with $(\mathcal{E}_{p}^{\bm{k}^{(n)}},\mathcal{F}_{p})$ a $\bm{k}^{(n)}$-Korevaar--Schoen $p$-energy form on $(K,m)$ for each $n \in \mathbb{N}$,
	there exists a sequence $\{ n_{j} \}_{j \in \mathbb{N}} \subseteq \mathbb{N}$ with $n_{j} < n_{j + 1}$ for any $j \in \mathbb{N}$ such that the following limit exists in $[0,\infty)$ for any $u \in \mathcal{F}_{p}$: 
	\begin{equation}\label{KSss.constr2-CGQ}
         \breve{\mathcal{E}}_{p}^{\mathrm{KS}}(u) \coloneqq \lim_{j \to \infty}\mathcal{E}_{p}^{\bm{k}^{(n_j)}}(u).
    \end{equation}
    Moreover, for any such $\{ \mathcal{E}_{p}^{\bm{k}^{(n)}} \}_{n \in \mathbb{N}}$ and $\{ n_{j} \}_{j \in \mathbb{N}}$, the functional $\breve{\mathcal{E}}_{p}^{\mathrm{KS}} \colon \mathcal{F}_{p} \to [0,\infty)$ defined by \eqref{KSss.constr2-CGQ} satisfies the following properties: 
    \begin{enumerate}[label=\textup{(\alph*)},align=left,leftmargin=*,topsep=2pt,parsep=0pt,itemsep=2pt]
    	\item\label{it:CGQss.ss} $(\breve{\mathcal{E}}_{p}^{\mathrm{KS}},\mathcal{F}_{p})$ is a self-similar $p$-resistance form on $\mathcal{L}$ with weight $(\rweight_{p,i})_{i \in S}$. 
    	\item\label{it:CGQss.comparable} For any $u \in \mathcal{F}_{p}$, 
    		\begin{equation*}
    			C^{-1}\mathcal{E}_{p}(u) \le \breve{\mathcal{E}}_{p}^{\mathrm{KS}}(u) \le C\mathcal{E}_{p}(u), 
    		\end{equation*}
    		where $C \in [1,\infty)$ is the constant in \eqref{KSfull2}. 
    	\item\label{it:CGQss.two-limit} For any $u,v \in \mathcal{F}_{p}$, $\{ \mathcal{E}_{p}^{\bm{k}^{(n_{j})}}(u; v) \}_{j \in \mathbb{N}}$ is convergent in $\mathbb{R}$ and 
    	\begin{equation}\label{KS-CGQ.compatible}
			\breve{\mathcal{E}}_{p}^{\mathrm{KS}}(u; v) = \lim_{j \to \infty}\mathcal{E}_{p}^{\bm{k}^{(n_{j})}}(u; v). 
		\end{equation}
		\item\label{it:CGQss.basic} Theorem \ref{thm.KS-energy}-\ref{it:difffunc},\ref{KS.conti},\ref{KS.slbdd} with $(\breve{\mathcal{E}}_{p}^{\mathrm{KS}},\mathcal{F}_{p})$ in place of $(\mathcal{E}_{p}^{\bm{k}},B_{p,\infty}^{\bm{k}})$ hold. 
		\item\label{it:CGQss.inv} $\breve{\mathcal{E}}_{p}^{\mathrm{KS}}(u \circ T) = \breve{\mathcal{E}}_{p}^{\mathrm{KS}}(u)$ for any $u \in \mathcal{F}_{p}$ and any $T \in \mathcal{G}$ (recall \eqref{e:defn-symmetry}).
    \end{enumerate}
\end{thm}

In addition, we obtain the $p$-energy measures associated with the $p$-resistance form $(\breve{\mathcal{E}}_{p}^{\mathrm{KS}},\mathcal{F}_{p})$ in the same way as in Theorem \ref{thm.Kigsspem}.
(See also \cite[Sections 4 and 5]{KS.gc} for other basic properties. 
As mentioned before Theorem \ref{thm.Kigsspem}, we do not know whether Theorem \ref{thm.CGQsspem}-\ref{it:CGQ.basic} below holds for general self-similar $p$-resistance forms.)
\begin{thm} \label{thm.CGQsspem}
	Let $(\mathcal{E}_{p}^{\bm{k}^{(n)}},\mathcal{F}_{p})$ be any $\bm{k}^{(n)}$-Korevaar--Schoen $p$-energy form on $(K,m)$ for each $n \in \mathbb{N}$, let $\{ n_{j} \}_{j \in \mathbb{N}} \subseteq \mathbb{N}$ be any sequence as in Theorem \ref{thm.KSss-CGQ}, and let $(\breve{\mathcal{E}}_{p}^{\mathrm{KS}},\mathcal{F}_{p})$ be the $p$-resistance form on $K$ defined by \eqref{KSss.constr2-CGQ}.
	Then for any $u \in \mathcal{F}_{p}$, there exists a unique positive Radon measure $\breve{\Gamma}_{p}^{\mathrm{KS}}\langle u \rangle$ on $K$ such that for any $\varphi \in \mathcal{F}_{p}$, 
	\begin{equation}\label{CGQsspem.defn}
		\int_{K}\varphi\,d\breve{\Gamma}_{p}^{\mathrm{KS}}\langle u \rangle = \breve{\mathcal{E}}_{p}^{\mathrm{KS}}(u; u\varphi) - \left(\frac{p - 1}{p}\right)^{p - 1}\breve{\mathcal{E}}_{p}^{\mathrm{KS}}\bigl(\abs{u}^{\frac{p}{p - 1}}; \varphi\bigr). 
	\end{equation}
	Moreover, the following hold: 
	\begin{enumerate}[label=\textup{(\alph*)},align=left,leftmargin=*,topsep=2pt,parsep=0pt,itemsep=2pt]
		\item\label{it:CGQ.ext} Let $\varphi \colon K \to [0,\infty)$ be a Borel measurable function with $\norm{\varphi}_{\sup} < \infty$. Then $(\int_{K}\varphi\,d\breve{\Gamma}_{p}^{\mathrm{KS}}\langle \,\cdot\, \rangle,\mathcal{F}_{p})$ is a $p$-energy form on $(K,m)$ satisfying \ref{GC}.
		\item\label{it:CGQ.limit} Theorem \ref{thm.KSpem-diffble}, with $\mathcal{F}_{p}$ and $\breve{\Gamma}_{p}^{\bm{k}}$ in place of $\bclosureKS$ and $\KSem$ respectively, holds. In particular, for any $u,v \in \mathcal{F}_{p}$, 
			\begin{equation}\label{e:defn.newKSpem.CGQss}
				\breve{\Gamma}_{p}^{\mathrm{KS}}\langle u;v \rangle(A) \coloneqq \frac{1}{p}\left.\frac{d}{dt}\breve{\Gamma}_{p}^{\mathrm{KS}}\langle u + tv \rangle(A)\right|_{t = 0}, \quad A \in \mathcal{B}(K),
			\end{equation}
			defines a signed Borel measure on $K$ such that $\breve{\Gamma}_{p}^{\mathrm{KS}}\langle u;v \rangle(K) = \breve{\mathcal{E}}_{p}^{\mathrm{KS}}(u;v)$ and $\breve{\Gamma}_{p}^{\mathrm{KS}}\langle u;u \rangle = \breve{\Gamma}_{p}^{\mathrm{KS}}\langle u \rangle$.
			Furthermore, for any $u,v \in \mathcal{F}_{p}$ and any $\varphi \in \contfunc(K)$,
			\begin{equation}\label{CGQpem.compatible}
				\int_{K}\varphi\,d\breve{\Gamma}_{p}^{\mathrm{KS}}\langle u;v \rangle 
				= \lim_{j \to \infty}\int_{K}\varphi\,d\Gamma_{p}^{\bm{k}^{(n_{j})}}\langle u;v \rangle.
			\end{equation}
		\item\label{it:CGQ.basic} Theorem \ref{thm.KSpem-two}-\ref{it:KSpem.difffunc},\ref{it:KSpem.conti}, with $\mathcal{F}_{p}$ and $\breve{\Gamma}_{p}^{\mathrm{KS}}$ in place of $\bclosureKS$ and $\Gamma_{p}^{\bm{\bm{k}}}$ respectively, hold. 
		\item\label{it:CGQ.chain} Theorems \ref{thm.KSpem-chain}, \ref{thm.EIDP} and \ref{thm.KSpem-sl}, with $\mathcal{F}_{p}$ and $\breve{\Gamma}_{p}^{\mathrm{KS}}$ in place of $\KS \cap \contfunc_{b}(K)$ and $\Gamma_{p}^{\bm{\bm{k}}}$ respectively, hold.
	\end{enumerate}
\end{thm}

\appendix
\setcounter{section}{0} 
\setcounter{theorem}{0} 
\setcounter{equation}{0}
\renewcommand{\thesection}{\Alph{section}} 
\section{Appendix: An alternative family of kernels in Example \ref{ex.KS}}\label{app:WMdistance}
In this Appendix, we give a simple sufficient condition for $B_{p,\infty}^{\bm{k}^{\#}} = \mathrm{KS}^{1,p}$ and \hyperref[KSwm]{\textup{(WM)$_{p,\bm{k}^{\#}}$}} where $\bm{k}^{\#} = \{ k^{\#}_{r} \}_{r > 0}$ is defined by \eqref{KSkernel.dist}. 
As in Example \ref{ex.KS}, we fix $p \in (1,\infty)$ and assume that $(K,d)$ is a connected separable metric space with $\#K \ge 2$ and that $m$ is a Borel measure on $K$ with full topological support satisfying $m(B_{d}(x,r)) < \infty$ for any $(x,r) \in K \times (0,\infty)$.  
\begin{prop}\label{prop:KSidentify}
	Let $s \in (0,\infty)$ and let $\bm{k}^{s} = \{ k_{r}^{s} \}_{r > 0}$ be the family of kernels defined by \eqref{KSkernel}. 
	Assume that $m$ is volume doubling, and that the following Poincar\'{e}-type inequality holds: there exist $C \in (0,\infty)$ and $\lambda \in [1,\infty)$ such that for any $u \in B^{\bm{k}^{s}}_{p,\infty}$ and any $(z,r) \in K \times (0,\infty)$, 
	\begin{align}\label{e:KSPI}
		&\int_{B_{d}(z,r)}\abs{u - u_{B_{d}(z,r)}}^{p}\,dm \nonumber \\
		&\quad\le Cr^{ps}\liminf_{\delta \downarrow 0}\int_{B_{d}(z,\lambda r)}\fint_{B_{d}(x,\delta)}\frac{\abs{u(x) - u(y)}^{p}}{\delta^{ps}}\,m(dy)m(dx). 
	\end{align}
	Then there exists $C' \in [1,\infty)$ such that for any $u \in B^{\bm{k}^{s}}_{p,\infty}$, 
	\begin{align}\label{e:KScomparable}
		&\sup_{r > 0}\int_{K}\fint_{B_{d}(x,r)}\frac{\abs{u(x)-u(y)}^p}{d(x,y)^{ps}}\,m(dy)m(dx) \nonumber \\
		&\quad \le C'\liminf_{r \downarrow 0}\int_{K}\fint_{B_{d}(x,r)}\frac{\abs{u(x)-u(y)}^{p}}{r^{ps}}\,m(dy)m(dx).   
	\end{align}
	In particular, if \eqref{e:KSPI} with $s = s_{p}$ holds, then the family of kernels $\bm{k}^{\#} = \{ k^{\#}_{r} \}_{r > 0}$ defined by \eqref{KSkernel.dist} satisfies $B_{p,\infty}^{\bm{k}^{\#}} = \mathrm{KS}^{1,p}$ and \hyperref[KSwm]{\textup{(WM)$_{p,\bm{k}^{\#}}$}}. 
\end{prop}
\begin{rmk}
	If $(K,d,m)$ supports the $p$-Poincar\'{e} inequality in terms of upper gradients, then the estimate \eqref{e:KScomparable} with $s = 1$ follows from \cite[Corollaries 6.3 and 6.5]{LPZ22+}. 
	For $p$-conductively homogeneous compact metric spaces and post-critically finite self-similar sets as in the settings of Sections \ref{sec.Kig} and \ref{sec.CGQ}, we can verify \eqref{e:KScomparable} with $s = s_{p}$; see Propositions \ref{prop.PI-Kig}, \ref{prop.Rp-PI} and \ref{prop.lower-Rp}.  
\end{rmk}
\begin{proof}[Proof of Proposition \ref{prop:KSidentify}]
	Since $m$ is volume doubling and $(K,d)$ is connected, we have the following \emph{reverse volume doubling} property of $m$ (see, e.g., \cite[Corollary 3.8]{BB11} or \cite[Exercise 13.1]{Hei}): there exist $c_{1},\alpha \in (0,\infty)$ depending only on the doubling constant $C_{\mathrm{D}}$ of $m$ such that 
	\begin{equation}\label{e:RVD}
		\frac{m(B_{d}(x,r))}{m(B_{d}(x,R))} \le c_{1}\Bigl(\frac{r}{R}\Bigr)^{\alpha} \quad \text{for any $x \in K$ and any $0 < r \le R < \diam(K,d)$.}
	\end{equation}
	Let $r \in (0,\infty)$. We have 
	\begin{align*}
		&\int_{K}\fint_{B_{d}(x,r)}\frac{\abs{u(x) - u(y)}^{p}}{d(x,y)^{ps}}\,m(dy)m(dx) \\
		&\overset{\eqref{e:RVD}}{\le} \sum_{j = 0}^{\infty} \frac{c_{1}}{2^{\alpha j}} \int_{K}\int_{B_{d}(x,2^{-j}r) \setminus B_{d}(x,2^{-(j + 1)}r)}\frac{\abs{u(x) - u(y)}^{p}}{d(x,y)^{ps}m(B(x,2^{-j}r))}\,m(dy)m(dx). 
	\end{align*}
	Let $j \in \mathbb{N}\cup\{0\}$, and let $N_{j} \subset K$ be a \emph{$2^{-j}r$-net} in $(K,d)$, i.e., a maximal subset of $K$ such that $d(z_{1},z_{2})\geq 2^{-j}r$ for any $z_{1},z_{2}\in N_{j}$ with $z_{1}\not=z_{2}$; such $N_{j}$ exists and is countable since $B_{d}(x,R)$ is totally bounded for any $(x,R) \in K \times (0,\infty)$ thanks to the metric doubling property of $d$ implied by the volume doubling property of $m$. 
	Then we see that 
	\begin{align*}
		&\int_{K}\int_{B_{d}(x,2^{-j}r) \setminus B_{d}(x,2^{-(j + 1)}r)}\frac{\abs{u(x) - u(y)}^{p}}{d(x,y)^{ps}m(B_{d}(x,2^{-j}r))}\,m(dy)m(dx) \\
		&\le \sum_{z \in N_{j}}\int_{B_{d}(z,2^{-j}r)}\int_{B_{d}(x,2^{-j}r) \setminus B(x,2^{-(j + 1)}r)}\frac{\abs{u(x) - u(y)}^{p}}{d(x,y)^{ps}m(B_{d}(x,2^{-j}r))}\,m(dy)m(dx) \\
		&\le \sum_{z \in N_{j}}\frac{C_{\mathrm{D}}^{2}(2^{-(j+1)}r)^{-ps}}{m(B_{d}(z,2^{-j}r))}\int_{B_{d}(z,2^{-j}r)}\int_{B_{d}(x,2^{-j}r)}\abs{u(x) - u(y)}^{p}\,m(dy)m(dx) \\
		&\le \sum_{z \in N_{j}}\frac{2^{p}C_{\mathrm{D}}^{2}(2^{-(j+1)}r)^{-ps}}{m(B_{d}(z,2^{-j}r))}\iint_{B_{d}(z,2^{-j+1}r)^{2}}\abs{u(x) - u_{B_{d}(z,2^{-j+1}r)}}^{p}\,m(dy)m(dx) \\
		&\overset{\eqref{e:KSPI}}{\le} 2^{p + 2ps}C_{\mathrm{D}}^{3}\sum_{z \in N_{j}}\liminf_{\delta \downarrow 0}\int_{B_{d}(z,\lambda 2^{-j+1}r)}\fint_{B_{d}(x,\delta)}\frac{\abs{u(x) - u(y)}^{p}}{\delta^{ps}}\,m(dy)m(dx) \\
		&\le c_{2}\liminf_{\delta \downarrow 0}\int_{K}\fint_{B_{d}(x,\delta)}\frac{\abs{u(x) - u(y)}^{p}}{\delta^{ps}}\,m(dy)m(dx), 
	\end{align*}
	where $c_{2}$ depends only on $p,s,\lambda,C_{\mathrm{D}}$ and $C$ in \eqref{e:KSPI}. 
	Combining the above estimates, we obtain 
	\begin{align*}
		&\int_{K}\fint_{B_{d}(x,r)}\frac{\abs{u(x) - u(y)}^{p}}{d(x,y)^{ps}}\,m(dy)m(dx) \\ 
		&\le c_{1}c_{2}\sum_{j = 0}^{\infty}2^{-\alpha j}\liminf_{\delta \downarrow 0}\int_{K}\fint_{B_{d}(x,\delta)}\frac{\abs{u(x) - u(y)}^{p}}{\delta^{ps}}\,m(dy)m(dx) \\
		&= c_{1}c_{2}(1 - 2^{-\alpha})^{-1}\liminf_{\delta \downarrow 0}\int_{K}\fint_{B_{d}(x,\delta)}\frac{\abs{u(x) - u(y)}^{p}}{\delta^{ps}}\,m(dy)m(dx),
	\end{align*}
	which shows \eqref{e:KScomparable}. 
\end{proof}

\begin{acknowledgement}
	The authors would like to thank Mathav Murugan for his comments on an earlier version of this paper. 
    Naotaka Kajino was supported in part by JSPS KAKENHI Grant Numbers JP21H00989, JP22H01128, JP23K22399. Ryosuke Shimizu was supported in part by JSPS KAKENHI Grant Numbers JP20J20207, JP23KJ2011. This work was supported by the Research Institute for Mathematical Sciences, an International Joint Usage/Research Center located in Kyoto University.
\end{acknowledgement}

\end{document}